\documentclass[12pt]{amsart}
\usepackage{amssymb,amsfonts,amsmath,amsopn,amstext,amscd,color,latexsym,mathabx,mathrsfs,skak,stackrel,xy,url,verbatim}

\input xy
\xyoption{all}

\theoremstyle{remark}

\newtheorem{para}{\bf}[subsection]

\newtheorem{rem}[para]{\bf Remark}
\theoremstyle{definition}

\newtheorem{dfn}[para]{Definition}

\theoremstyle{plain}

\newtheorem{thm}[para]{Theorem}

\newtheorem{lemma}[para]{Lemma}
\newtheorem{cor}[para]{Corollary}
\newtheorem{prop}[para]{Proposition}

\newenvironment{numequation}{\addtocounter{para}{1}
\begin{equation}}{\end{equation}}

\newcommand{\vpi}{{\varpi}}

\newcommand{\vep}{{\varepsilon}}

\newcommand{\bbA}{{\mathbb A}}
\newcommand{\bbB}{{\mathbb B}}

\newcommand{\bbD}{{\mathbb D}}

\newcommand{\bbF}{{\mathbb F}}
\newcommand{\bbG}{{\mathbb G}}

\newcommand{\bbM}{{\mathbb M}}
\newcommand{\bbN}{{\mathbb N}}

\newcommand{\bbP}{{\mathbb P}}
\newcommand{\bbQ}{{\mathbb Q}}

\newcommand{\bbS}{{\mathbb S}}
\newcommand{\bbT}{{\mathbb T}}
\newcommand{\bbU}{{\mathbb U}}
\newcommand{\bbV}{{\mathbb V}}

\newcommand{\bbX}{{\mathbb X}}

\newcommand{\bbZ}{{\mathbb Z}}

\newcommand{\bM}{{\bf M}}

\newcommand{\bT}{{\bf T}}

\newcommand{\frb}{{\mathfrak b}}

\newcommand{\frg}{{\mathfrak g}}

\renewcommand{\frm}{{\mathfrak m}}

\newcommand{\fro}{{\mathfrak o}}

\newcommand{\frt}{{\mathfrak t}}

\newcommand{\frF}{{\mathfrak F}}

\newcommand{\frU}{{\mathfrak U}}

\newcommand{\frX}{{\mathfrak X}}

\newcommand{\cA}{{\mathcal A}}
\newcommand{\cB}{{\mathcal B}}

\newcommand{\cD}{{\mathcal D}}
\newcommand{\cE}{{\mathcal E}}
\newcommand{\cF}{{\mathcal F}}
\newcommand{\cG}{{\mathcal G}}

\newcommand{\cI}{{\mathcal I}}

\newcommand{\cK}{{\mathcal K}}
\newcommand{\cL}{{\mathcal L}}
\newcommand{\cM}{{\mathcal M}}
\newcommand{\cN}{{\mathcal N}}
\newcommand{\cO}{{\mathcal O}}

\newcommand{\cR}{{\mathcal R}}
\newcommand{\cS}{{\mathcal S}}
\newcommand{\cT}{{\mathcal T}}

\newcommand{\cX}{{\mathcal X}}
\newcommand{\cY}{{\mathcal Y}}

\newcommand{\sD}{{\mathscr D}}
\newcommand{\sE}{{\mathscr E}}
\newcommand{\sF}{{\mathscr F}}
\newcommand{\sG}{{\mathscr G}}

\newcommand{\sL}{{\mathscr L}}
\newcommand{\sM}{{\mathscr M}}
\newcommand{\sN}{{\mathscr N}}

\newcommand{\sR}{{\mathscr R}}

\newcommand{\Q}{{\mathbb Q}}
\newcommand{\Z}{{\mathbb Z}}

\newcommand{\Qp}{{\mathbb Q_p}}
\newcommand{\Cp}{{\mathbb C_p}}
\newcommand{\Zp}{{\mathbb Z_p}}

\newcommand{\Fq}{{\mathbb F_q}}

\newcommand{\fronr}{{\hat{\fro}^{\rm nr}}}
\newcommand{\Lnr}{{\hat{L}^{\rm nr}}}

\newcommand{\Ind}{{\rm{Ind}}}
\newcommand{\ind}{{\rm{ind}}}

\newcommand{\Lie}{{\rm{Lie}}}
\newcommand{\Pf}{{\it Proof. }}

\newcommand{\Sp}{{\rm Sp}}
\newcommand{\Spec}{{\rm Spec}}
\newcommand{\Spf}{{\rm Spf}}

\newcommand{\coker}{{\rm coker}}

\newcommand{\eqdef}{\;\stackrel{\text{\tiny def}}{=}\;}
\newcommand{\lra}{\longrightarrow}
\newcommand{\Loc}{{\mathscr Loc}}
\newcommand{\hra}{\hookrightarrow}
\newcommand{\im}{{\rm im}}
\newcommand{\midc}{{\; | \;}}

\newcommand{\ra}{\rightarrow}

\newcommand{\sub}{\subset}

\newcommand{\tsD}{{\widetilde{\sD}}}

\newcommand{\tcD}{{\widetilde{\cD}}}
\newcommand{\tcT}{{\widetilde{\cT}}}

\newcommand{\Coh}{{\rm Coh}}
\newcommand{\GL}{{\rm GL}}

\newcommand{\SL}{{\rm SL}}

\newcommand{\Sym}{{\rm Sym}}

\newcommand{\car}{\stackrel{\simeq}{\longrightarrow}}

\newcommand{\GN}{G_0}
\newcommand{\Gnc}{\bbG(n)^\circ}
\newcommand{\Dgn}{\cD^{\rm an}(\bbG(n)^\circ)}
\newcommand{\Dgnt}{\cD^{\rm an}(\bbG(n)^\circ)_{\theta_0}}
\newcommand{\Dgnn}{D(\bbG(n)^\circ,\GN)}
\newcommand{\Dgnnt}{D(\bbG(n)^\circ,\GN)_{\theta_0}}

\newcommand{\DGNt}{D(\GN)_{\theta_0}}

\newcommand{\ok}{k}

\begin{document}

\title{Locally analytic representations of $\GL(2,L)$ via semistable models of $\bbP^1$}
\author{Deepam Patel}
\address{Department of Mathematics, Purdue University,
150 N. University Street, West Lafayette, IN 47907, U.S.A.}
\email{deeppatel1981@gmail.com}
\author{Tobias Schmidt}
\address{Institut f\"ur Mathematik, Humboldt-Universit\"at zu Berlin,
Rudower Chaussee 25, 12489 Berlin, Germany}
\email{Tobias.Schmidt@mathematik.hu-berlin.de}
\author{Matthias Strauch}
\address{Indiana University, Department of Mathematics, Rawles Hall, Bloomington, IN 47405, U.S.A.}
\email{mstrauch@indiana.edu}

\thanks{M. S. would like to acknowledge the support of the National Science Foundation (award DMS-1202303). T. S. would like to acknowledge support of the Heisenberg programme of Deutsche Forschungsgemeinschaft.}

\begin{abstract} In this paper we study certain sheaves of $p$-adically complete rings of differential operators on semistable models of the projective line over the ring of integers in a finite extension $L$ of ${\mathbb Q}_p$. The global sections of these sheaves can be identified with (central reductions of) analytic distribution algebras of wide open congruence subgroups. It is shown that the global sections functor furnishes an equivalence between the categories of coherent module sheaves and finitely presented modules over the distribution algebras. Using the work of M. Emerton, we then describe admissible representations of $\GL_2(L)$ in terms of sheaves on the projective limit of these formal schemes.
As an application, we show that representations coming from certain equivariant line bundles on Drinfeld's first \'etale covering of the $p$-adic upper half plane are admissible.
\end{abstract}

\maketitle

\tableofcontents

\section{Introduction}

Let $L/\Qp$ be a finite extension with ring of integers $\fro = \fro_L$. In  \cite{PSS2}, and when $L = \Qp$, we considered a particular family $(\bbX_n)_n$ of semistable integral models of $\bbP^1_{\fro}$. These integral models have their natural analogues when we replace the base field $\Qp$ by $L$, and so we will henceforth denote by $\bbX_n$ the analogously defined integral model of $\bbP^1_\fro$, and we let $\frX_n$ denote the completion of $\bbX_n$ along its special fiber. Whereas in \cite{PSS2} we have studied sheaves of {\it logarithmic} arithmetic differential operators on $\frX_n$, we define and study here a different family of sheaves of differential operators. As is shown in \cite{PSS3}, the sheaf $\sD^{{\rm log},\dagger}_{\frX_1,\Q}$ of logarithmic arithmetic differential operators on $\frX_1$ has non-vanishing $H^1$. For this reason we have sought to replace the sheaves of logarithmic arithmetic differential operators by sheaves of differential operators
$\tsD^\dagger_{n,k,\Q}$ which have the property that $\frX_n$ is $\tsD^\dagger_{n,k,\Q}$-affine. By this we mean that every coherent $\tsD^\dagger_{n,k,\Q}$-module $\sM$ is generated (over $\tsD^\dagger_{n,k,\Q}$) by its global sections, and that $H^i(\frX_n,\sM) = 0$ for all $i>0$, cf. \cite{BB81}. In particular, the global sections functor $H^0(\frX_n, -)$ gives an equivalence of categories of coherent $\tsD^\dagger_{n,k,\Q}$-modules and coherent modules over $H^0\Big(\frX_n, \tsD^\dagger_{n,k,\Q}\Big)$. The latter ring is canonically isomorphic to a central reduction of the analytic distribution algebra $\cD^{\rm an}(\bbG(k)^\circ)$, as defined by M. Emerton in \cite[ch. 5]{EmertonA}. The rigid-analytic wide open group $\bbG(k)^\circ$ has the property that $\bbG(k)^\circ(\Cp) = 1+\vpi^k M_2(\frm_\Cp)$, where $\frm_\Cp$ is the maximal ideal of the valuation ring of $\Cp$.

\vskip8pt

Section \ref{models} discusses some preliminaries regarding the integral models $\bbX_n$ and the formal models $\frX_n$. In section \ref{new_sheaves}, we define and study sheaves $\tcD^{(m)}_{n,k}$ (for $k \ge n$) of differential operators on $\bbX_n$. When $m=0$ the sheaf $\tcD^{(0)}_{n,k}$ is generated by the pull-back of $p^k\cT_\bbX$ and the structure sheaf, where $\bbX =\bbX_0 = \bbP^1_\fro$ is the smooth model and $\cT_\bbX$ its relative tangent sheaf over $\bbS = \Spec(\fro)$. By completion, the sheaf $\tcD^{(m)}_{n,k}$ gives rise to a sheaf of $p$-adically complete differential operators on $\frX_n$ which we denote by $\tsD^{(m)}_{n,k}$. Then we consider the inductive limit

$$\tsD^\dagger_{n,k} = \varinjlim_m \tsD^{(m)}_{n,k}$$

\vskip8pt

and we put $\tsD^\dagger_{n,k,\Q} = \tsD^\dagger_{n,k} \otimes_\Z \Q$. The main result of section \ref{loc_n} is that $\frX_n$ is $\tsD^\dagger_{n,k,\Q}$-affine. This partly generalizes the work of C. Huyghe in \cite{Huyghe97}, where she shows $\sD^\dagger_\Q$-affinity for all projective spaces, and, moreover, considers differential operators with coefficients. The sheaf of arithmetic differential operators considered in \cite{Huyghe97} is identical to our sheaf $\tsD^\dagger_{0,0,\Q}$ on the smooth model $\bbX = \bbX_0$. The proofs which lead to $\tsD^\dagger_{n,k,\Q}$-affinity of $\frX_n$ are in general easy variations of the proofs given by Huyghe in the case when $(n,k) = (0,0)$. In many cases we follow her strategy and arguments word for word, and we do not claim any originality here. In some cases we have chosen to slightly expand her arguments. And we have decided to first discuss properties of the sheaves of rings $\tsD^{(m)}_{n,k}$ (in sec. \ref{new_sheaves}) and then discuss coherent modules over these sheaves (in sec. \ref{loc_n}), whereas in \cite{Huyghe97} this is often done in a more parallel way.

 \vskip8pt

In section \ref{loc} we pass to the limit over all $n$ and prove the main technical result of the paper. The sheaves $\tsD^\dagger_{n,n,\Q}$
assemble to a sheaf of infinite order differential operators $\tsD^\dagger_{\infty,\Q}$ on the space $\frX_\infty=\varprojlim_n\frX_n$.
The space $\frX_\infty$ has a natural $\GL_2(L)$-action extending the $\GL_2(\fro)$-action on $\frX_0=\hat{\bbP}^1_{\fro}$ and $\tsD^\dagger_{\infty,\Q}$ is equivariant with respect to this action. In this situation, the above equivalences of categories for each $\frX_n$ assemble to an equivalence of categories between admissible locally analytic $\GL_2(L)$-representations with trivial infinitesimal character and a full subcategory of equivariant $\tsD^\dagger_{\infty,\Q}$-modules.

\vskip8pt

We remark at this point that some of our constructions are related -- in the case of the {\it smooth} formal model -- to the work of K. Ardakov and S. Wadsley, cf. \cite{AW}, and the work of one of us, \cite{SchmidtDIM}. Indeed, the family of sheaves of completed differential operators $\tsD^{(0)}_{0,k,\Q}$, as defined here, is essentially the same family of sheaves as considered in these papers. Furthermore, the sheaf $\tsD^\dagger_{\infty,\Q}$ is a kind of Arens-Michael completion of the sheaf of algebraic differential operators, and as such closely related to the work of Ardakov and Wadsley, cf. \cite{ArdakovICM}. We would also like to point out that our constructions made here, and the arguments used in proofs of the main results, carry over to more general, not necessarily semistable, models, and also to general reductive groups and their flag varieties \cite{HPSS}.

\vskip8pt
In section \ref{examples} we compute the equivariant $\tsD^\dagger_{\infty,\Q}$-modules for a class of examples including
smooth representations, principal series representations and representations coming from the $p$-adic upper half plane.
\vskip8pt

In section \ref{appl} we give an application of our methods and show that representations coming from certain equivariant line bundles on Drinfeld's first \'etale covering \cite{Drinfeld76} of the $p$-adic upper half plane are admissible. We establish this result by verifying, truly in the spirit of localization theory, that the $\tsD^\dagger_{n,n,\Q}$-module associated with the representation is coherent for all $n$. The latter property is verified locally on $\frX_n$ by making use of results of Caro \cite{Caro06}, Noot-Huyghe/Trihan \cite{NH_Trihan07}, Shiho \cite{Shiho_RelativeI} and Tsuzuki \cite{TsuzukiBaseChange} on the coherence of the direct image in crystalline cohomology.
\vskip8pt

We therefore see that, in this instance, techniques and results from the theory
of crystalline cohomology and arithmetic $\sD$-modules can be successfully applied
in the context of locally analytic representation theory, and we hope that this develops into a systematic and fruitful interaction between these fields.

\vskip8pt

{\it Acknowledgements.} We thank Christine Huyghe and Elmar Grosse-Kl\"onne for some interesting discussions related to the material of this paper. Parts of this paper were written at MSRI in Fall 2014. We are grateful to MSRI and its staff for providing excellent working conditions.

\vskip8pt

{\it Notation.} $L$ is a finite extension of $\Qp$, with ring of integers $\fro$, uniformizer $\vpi$. We let $q$ denote the cardinality of the residue field $\fro/(\vpi)$ which we also denote by $\Fq$.

\section{The integral models $\bbX_n$ and their formal completions $\frX_n$}\label{models}

\subsection{Construction via blowing up}\label{subsect-con_blow_up} The integral models $\bbX_n$ have been discussed in detail in \cite{PSS2} in the case when $\fro = \Zp$. Since the generalization to the ring of integers $\fro$ of the finite extension $L/\Qp$ is straightforward, we will only briefly recall the construction. Let $\Lambda \sub L^{\oplus 2}$ be a lattice, and let $\bbX = \bbX_0 = {\bbP}{\rm roj}(\Lambda)$ be the smooth model of $\bbP^1_L = \bbP{\rm roj}(L^{\oplus 2})$ associated to $\Lambda$. While it is eventually useful to note that all constructions in sections \ref{models}-\ref{loc_n} can be done canonically for an arbitrary lattice $\Lambda$ we will henceforth consider only the case when $\Lambda = \fro \oplus \fro$ is the standard lattice.

\vskip8pt

Blowing up $\bbX_0$ in the $\Fq$-rational points of its special fiber (i.e., blowing up the corresponding ideal sheaf) gives a model $\bbX_1$ whose special fiber has $q+2$ irreducible components, all of which are isomorphic to $\bbP^1_\Fq$. Exactly one of these components intersects transversely every other irreducible component, and we call it the 'central' component, whereas the other components are called the 'outer' components. The smooth $\Fq$-rational points lie only on the outer components, and blowing up $\bbX_1$ in these smooth $\Fq$-valued points of its special fiber produces the model $\bbX_2$. In general, $\bbX_{n+1}$ is obtained by blowing up the smooth $\Fq$-rational points of the special fiber of $\bbX_n$. For $n \ge n'$ we denote the blow-up morphism from $\bbX_n$ to $\bbX_{n'}$ by ${\rm pr}_{n,n'}$. The morphism ${\rm pr}_{n,0}: \bbX_n \ra \bbX = \bbX_0$ is often denoted by ${\rm pr}$ when $n$ is clear from the context.

\vskip8pt

Let $\frX_n$ be the completion of $\bbX_n$ along its special fiber. $\frX_n$ has an open affine covering by 'dumbbell' shaped formal schemes of the form

$$\Spf\left(\fro \langle x,z \rangle \Big[\frac{1}{x^{q-1}-1},\frac{1}{z^{q-1}-1}\Big] \Big/ \Big(xz-\vpi\Big)\right) \;,$$

\vskip8pt

together with formal affine lines $\widehat{\bbA}^1_\fro = \Spf(\fro\langle t \rangle)$ (which are only needed for the 'outermost' components). As above we denote the blow-up morphisms by ${\rm pr}_{n,n'}$. Removing the smooth $\Fq$-rational points from $\frX_n$ gives an open subscheme $\frX_n^\circ \sub \frX_n$. The morphism ${\rm pr}_{n+1,n}: \frX_{n+1} \ra \frX_n$ induces then an isomorphism from the preimage ${\rm pr}_{n+1,n}^{-1}(\frX_n^\circ)$ to $\frX_n^\circ$, so that we can consider $\frX_n^\circ$ as an open subscheme of $\frX_{n+1}$, which is actually contained in $\frX_{n+1}^\circ$. The inductive limit $\varinjlim_n \frX_n^\circ$ is a formal model of the $p$-adic upper half plane over $L$, cf. section \ref{the_padic_upper_half_plane}.

\vskip12pt

\subsection{Group actions on $\bbX_n$}\label{groups}

\begin{para} {\it The group schemes $\bbG(k)$.} Put

$$\bbG(0) = \bbG = \GL_{2,\fro} = \Spec\left(\fro\left[a,b,c,d,\frac{1}{\Delta}\right]\right) \;,$$

\vskip8pt

where $\Delta = ad-bc$, and the co-multiplication is the one given by the usual formulas. For $k \ge 1$ let $a_k$, $b_k$, $c_k$, and $d_k$ denote indeterminates. Define an affine group scheme $\bbG(k)$ over $\fro$ by setting

$$\cO(\bbG(k)) = \fro\left[a_k,b_k,c_k,d_k,\frac{1}{\Delta_k}\right] \;, \;\; \mbox{where } \; \Delta_k = (1+\vpi^k a_k)(1+\vpi^k d_k)-\vpi^{2k}b_kc_k \;,$$

\vskip8pt

and let the co-multiplication

$$\cO(\bbG(k)) \lra \cO(\bbG(k))\otimes_{\fro} \cO(\bbG(k)) = \fro\left[a_k,b_k,c_k,d_k,a_k',b_k',c_k',d_k',\frac{1}{\Delta_k},\frac{1}{\Delta_k'}\right]$$

\vskip8pt

be given by the formulas

$$\begin{array}{lcccccccc}
a_k & \mapsto & a_k &+& a_k' &+& \vpi^k a_k a_k' &+& \vpi^k b_k c_k' \;,\\
b_k & \mapsto & b_k &+& b_k' &+& \vpi^k a_k b_k' &+& \vpi^k b_k d_k' \;,\\
c_k & \mapsto & c_k &+& c_k' &+& p^k c_k a_k' &+& \vpi^k d_k c_k' \;,\\
d_k & \mapsto & d_k &+& d_k' &+& \vpi^k d_k d_k' &+& p^k c_k b_k' \;.\\
\end{array}$$

\vskip8pt

These group schemes are connected by homomorphisms $\bbG(k) \ra \bbG(k-1)$ given on the level of algebras as follows:

$$a_{k-1} \mapsto \vpi a_k \;, \;\; b_{k-1} \mapsto \vpi b_k \;, \;\; c_{k-1} \mapsto \vpi c_k \;, \;\; d_{k-1} \mapsto \vpi d_k \;, $$

\vskip8pt

if $n>1$. For $n=1$ we put

$$a \mapsto 1+\vpi a_1 \;, \;\; b \mapsto \vpi b_1 \;, \;\; c \mapsto \vpi c_1 \;, \;\; d \mapsto 1+\vpi d_1 \;. $$

\vskip8pt

For a flat $\fro$-algebra $R$ the homomorphism $\bbG(k) \ra \bbG(0) = \bbG$ induces an isomorphism of $\bbG(k)(R)$ with a subgroup of $\bbG(R)$, namely

$$\bbG(k)(R)  = \left\{\left(\begin{array}{cc} a & b \\ c & d \end{array}\right) \in \bbG(R) \; \Bigg| \; a-1, b, c, d-1 \in \vpi^k R \;\right\} \;.$$

\vskip8pt

Of course, the preceding formulas defining the group schemes are derived formally from this description by setting $a=1+\vpi^k a_k$, $b = \vpi^k b_k$, $c = \vpi^k c_k$, and $d = 1+\vpi^k d_k$.
\end{para}

\begin{para}{The rigid-analytic groups $\bbG(k)^{{\rm rig}}$ and $\bbG(k)^\circ$}\label{analyticG(k)}  Let $\widehat{\bbG}(k)$ be the completion of $\bbG(k)$ along its special fiber $\bbG(k)_{\Fq}$. This is a formal group scheme over $\Spf(\fro)$. Its generic fiber in the sense of rigid geometry is an affinoid rigid-analytic group over $L$ which we denote by $\bbG(k)^{\rm rig}$. We have for any completely valued field $L'/L$ (whose valuation extends the $p$-adic valuation)

$$\bbG(k)^{\rm rig}(L')  = \left\{\left(\begin{array}{cc} a & b \\ c & d \end{array}\right) \in \bbG(\fro_{L'}) \; \Bigg| \; a-1, b, c, d-1 \in \vpi^k \fro_{L'} \;\right\} \;.$$

\vskip8pt

Furthermore, we let $\widehat{\bbG}(k)^\circ$ be the completion of $\bbG(k)$ in the closed point corresponding to the unit element in $\bbG(k)_{\Fq}$. This is a formal group scheme over $\Spf(\fro)$ (not of topologically finite type). Its generic fiber in the sense of Berthelot, cf. \cite[sec. 7.1]{deJongCrystalline}, is a so-called 'wide open' rigid-analytic group over $L$ which we denote by $\bbG(k)^\circ$. We have for any completely valued field $L'/L$ (whose valuation extends the $p$-adic valuation)

$$\bbG(k)^\circ(L')  = \left\{\left(\begin{array}{cc} a & b \\ c & d \end{array}\right) \in \bbG(\fro_{L'}) \; \Bigg| \; a-1, b, c, d-1 \in \vpi^k \frm_{\fro_{L'}} \;\right\} \;.$$

\vskip8pt
\end{para}

\begin{prop} (i) For $k \ge n$ the natural action of $\GL_{2,L}$ on $\bbP^1_L$ extends to an action of the group scheme $\bbG(k)$ on $\bbX_n$.

\vskip8pt

(ii) For $k \ge n$ the formal group scheme $\widehat{\bbG}(k)$ acts on $\frX_n$. The action of the special fiber $\widehat{\bbG}(k)_s$ on the special fiber $\frX_{n,s}$ of $\frX_n$ is trivial if $k>n$.


\end{prop}

\Pf (i) Since there is a morphism of group schemes $\bbG(k) \ra \bbG(n)$ for $k\geq n$, we may restrict to $k=n$. We use induction on $n$, the case of $n=0$ being obvious.
Suppose the claim holds for $n-1$. Let $\pi:\bbX_n \ra \bbX_{n-1}$ be the blow-up morphism and let $Y \sub \bbX_{n-1}$ be an arbitrary open affine subscheme stable under $\bbG(n-1)$. Let $I \sub \cO(Y)$ be the ideal whose blow-up gives rise to $\pi^{-1}(Y)$. Then $\pi \in I$.
Let $A = \cO(\bbG(n))$ and $B = \cO(Y)$. Denote by $I_s, A_s$ and $B_s$ the reductions mod $\pi$.
Let $c: B \ra B \otimes A$ be the comorphism corresponding to the $\bbG(n-1)$-action on $Y$ and let $c'$ be its reduction mod $\pi$.
The special fibre of the natural $\fro$-morphism $\bbG(n)\ra\bbG(n-1)$ factors through the unit section
$\Spec\; \Fq\ra\bbG(n-1)_s$ of the group scheme $\bbG(n-1)_s$. This means $c'(b')=b'\otimes 1$ for all $b'\in B'$ and hence

$$c(I) \sub I \otimes A + \pi( B\otimes A) = I \otimes A \;.$$

\vskip8pt

Since $c$ is a ring homorphism, this implies inductively $c(I^m)\sub I^m \otimes A$ for all $m$. This yields an induced comorphism $c: R(I) \ra R(I)\otimes A$ where $R(I):=\oplus_{m\geq 0} I^m$ denotes the blow-up algebra. Hence, the group scheme $\bbG(n)$ acts naturally on $\pi^{-1}(Y)={\rm Proj}(R(I))$. This implies the claim.

\vskip5pt

 (ii) This follows from {\rm (i)} by passing to formal completions. \qed

\vskip12pt

\subsection{A very ample line bundle on $\bbX_n$} Let ${\rm pr}: \bbX_n \ra \bbX = \bbX_0$ be the blow-up morphism, and let $\cI_n \sub \bbX$ be the ideal sheaf that is blown up. We have $\cI_n \otimes_\Z \Q = \cO_{\bbX} \otimes_\Z \Q$ from which we deduce the existence of some $N = N(n) \in \Z_{>0}$ such that\footnote{In fact, these inclusions hold for $N(n) = n$, but we will not need this.}

\begin{numequation}\label{invert_p_1}
p^N \cO_\bbX \sub \cI_n \sub \cO_{\bbX} \;.
\end{numequation}

Put $\cS = \bigoplus_{s \ge 0} \cI_n^s$, then $\bbX_n$ is glued together from schemes ${\bf Proj}(\cS(U))$ for affine open subsets $U \sub \bbX$. On each ${\bf Proj}(\cS(U))$ there is an invertible sheaf $\cO(1)$, and these glue together to give an invertible sheaf $\cO(1)$ on $\bbX_n$ which we will denote $\cO_{\bbX_n/\bbX}(1)$ (cf. the discussion in \cite[ch. II,\S 7]{HartshorneA}.) This invertible sheaf is in fact the inverse image ideal sheaf ${\rm pr}^{-1}(\cI_n) \cdot \cO_{\bbX_n}$, cf. \cite[ch. II, 7.13]{HartshorneA}. From \ref{invert_p_1} conclude that

\begin{numequation}\label{invert_p_2}
p^N \cO_{\bbX_n} \sub \cO_{\bbX_n/\bbX}(1) \sub \cO_{\bbX_n} \hskip10pt \mbox{and} \hskip10pt \cO_{\bbX_n} \sub \cO_{\bbX_n/\bbX}(-1) \sub p^{-N}\cO_{\bbX_n}\;.
\end{numequation}

And for any $r \ge 0$ we get

\begin{numequation}\label{invert_p_3}
p^{rN} \cO_{\bbX_n} \sub \cO_{\bbX_n/\bbX}(1)^{\otimes r} \sub \cO_{\bbX_n} \hskip10pt \mbox{and} \hskip10pt \cO_{\bbX_n} \sub \cO_{\bbX_n/\bbX}(-1)^{\otimes r} \sub p^{-rN}\cO_{\bbX_n}\;.
\end{numequation}

\begin{lemma}\label{direct_im_str_sh} Let $0 \le n' \le n$. (i) $({\rm pr}_{n,n'})_*(\cO_{\bbX_n}) = \cO_{\bbX_{n'}}$.

\vskip8pt

(ii) For all $i>0$ one has $R^i ({\rm pr}_{n,n'})_*(\cO_{\bbX_n}) = 0$.
\end{lemma}

\Pf (i) The morphism ${\rm pr}_{n,n'}: \bbX_n \ra \bbX_{n'}$ is a birational projective morphism of noetherian integral schemes, and $\bbX_{n'}$ is normal. The assertion then follows exactly as in the proof of Zariski's Main Theorem as given in \cite[ch. III, Cor. 11.4]{HartshorneA}.

\vskip8pt

(ii) We note that $R^i ({\rm pr}_{n,n'})_*(\cO_{\bbX_n})$ vanishes when $i>1$, by \cite[ch. III, 11.2]{HartshorneA}. So we will only consider the case $i=1$ now. Using the Leray spectral sequence for the higher direct images under the composite maps ${\rm pr}_{n-1,n-2} \circ {\rm pr}_{n,n-1}$, together with (i), we see that it suffices to show that
$R^1 ({\rm pr}_{n,n-1})_*(\cO_{\bbX_n}) = 0$. The sheaf in question certainly vanishes on $\bbX_{n-1}^\circ$, because the restriction of ${\rm pr}_{n,n-1}$ is an isomorphism

$$({\rm pr}_{n,n-1})^{-1}\Big(\bbX_{n-1}^\circ\Big) \stackrel{\simeq}{\lra} \bbX_{n-1}^\circ \;. $$

\vskip8pt

The question is local, and by what we have just seen, it suffices to consider the stalk of the sheaf $R^1 ({\rm pr}_{n,n-1})_*(\cO_{\bbX_n})$ at a smooth $\Fq$-rational  point $P$ of the special fiber of $\bbX_{n-1}$. Let $x$ be a local coordinate at $P$, so that $P$ corresponds to the ideal $(x,\vpi) \sub \fro[x]$. Put

$$R = \fro[x] \;, \hskip10pt R' = \fro[x,z]/(xz-\vpi) \;, \hskip10pt R'' = \fro[x,t]/(x-\vpi t) \;.$$

\vskip8pt

Blowing up $\Spec(R)$ in the ideal $(x,\vpi)$ gives a scheme which is covered by

$$V_0 = \Spec(R') \hskip10pt \mbox{and} \hskip10pt V_1 = \Spec(R'') \;,$$

\vskip8pt

where these schemes are glued together via the relation $zt=1$. For $f \in \fro[x] \; \setminus \; (x,p)$ we consider the open affine subset $D_f = \Spec(R_f) \sub \Spec(R)$, which is an open neighborhood of $P$. The family of all $D_f$ forms a fundamental system of open neighborhoods of $P$. The preimage $({\rm pr}_{n,n-1})^{-1}\Big(D_f\Big)$ is then the same as the blow-up of $D_f$ in $P$, i.e.,

$$({\rm pr}_{n,n-1})^{-1}\Big(D_f\Big) =  \Spec(R'_f) \cup \Spec(R''_f) \;.$$

\vskip8pt

We now compute $H^1\left(({\rm pr}_{n,n-1})^{-1}\Big(D_f\Big),\cO_{\bbX_n}\right)$ as \v{C}ech cohomology. What we then have to show is that the canonical map

$$R'_f \oplus R''_f \lra R'_f\left[\frac{1}{z}\right] = R''_f\left[\frac{1}{t}\right] \;, \;\; (f_1,f_2) \mapsto f_1-f_2 \;,$$

\vskip8pt

is surjective. An element in the group on the right hand side can be written as a finite sum $\sum_{j \in \Z} a_j t^j$ with $a_j \in R_f$. Then we have $f_1 \eqdef \sum_{j <0} a_j t^j = \sum_{j>0} a_{-j}z^j \in R'_f$ and $f_2 \eqdef -\sum_{j \ge 0} a_j t^j \in R''_f$. This shows the surjectivity of the map in question and the vanishing of $H^1\left(({\rm pr}_{n,n-1})^{-1}\Big(D_f\Big),\cO_{\bbX_n}\right)$. This in turn means that the stalk of $R^1 ({\rm pr}_{n,n-1})_*(\cO_{\bbX_n})$ at $P$ vanishes, and hence $R^1 ({\rm pr}_{n,n-1})_*(\cO_{\bbX_n}) = 0$. \qed

\vskip8pt

\begin{lemma}\label{v_ample_sh_lemma} There are $a_0, r_0 \in \Z_{>0}$ such that the line bundle

\begin{numequation}\label{v_ample_sh_disp}
\cL_n = \cO_{\bbX_n/\bbX}(1) \otimes {\rm pr}^*\Big(\cO_\bbX(a_0)\Big)
\end{numequation}

on $\bbX_n$ has the following properties:

\vskip8pt

(i) $\cL_n$ is very ample over $\Spec(\fro)$, and it is very ample over $\bbX$;

\vskip8pt

(ii) for all $r \ge r_0$, all $d \in \Z$, and all $i>0$

$$R^i {\rm pr}_* \Big(\cL_n^{\otimes r} \otimes {\rm pr}^*(\cO_\bbX(d))\Big)  = 0 \;;$$

\vskip8pt

(iii) for all $r \ge r_0$, all $d \ge 0$, and all $i>0$

$$H^i\Big(\bbX_n, \cL_n^{\otimes r} \otimes {\rm pr}^*(\cO_\bbX(d))\Big)  = 0 \;.$$

\vskip8pt
\end{lemma}

\Pf (i) {\it Choosing $a_0$.} By \cite[ch. II, ex. 7.14 (b)]{HartshorneA}, the sheaf

\begin{numequation}\label{v_ample_sh}
\cL_n = \cO_{\bbX_n/\bbX}(1) \otimes {\rm pr}^*\Big(\cO_\bbX(a_0)\Big)
\end{numequation}

is very ample on $\bbX_n$ over $\Spec(\fro)$ for suitable $a_0 > 0$. We fix such an $a_0$. By \cite[4.4.10 (v)]{EGA_II} it is then also very ample over $\bbX$, and hence (i) is fulfilled.

\vskip8pt

(ii) {\it Finding $r_0$.} By \cite[ch. III, 5.2, 8.8]{HartshorneA} there is $r_0 >0$ such that for all $r \ge r_0$, and all $i>0$, one has

\begin{numequation}\label{coh_1}
H^i\Big(\bbX_n, \cL_n^{\otimes r}\Big) = H^i\Big(\bbX_n, \cL_n^{\otimes r} \otimes {\rm pr}^*(\cO_\bbX(-1))\Big) = 0
\end{numequation}

and

\begin{numequation}\label{dir_im_1}
R^i {\rm pr}_* \Big(\cL_n^{\otimes r}\Big) = 0 \;.
\end{numequation}

We fix an $r_0$ with the properties \ref{coh_1} an \ref{dir_im_1}.
By \cite[ch. III, ex. 8.3]{HartshorneA} we have for all $i \ge 0$

\begin{numequation}\label{dir_im_2}
R^i {\rm pr}_* \Big(\bbX_n, \cL_n^{\otimes r} \otimes {\rm pr}^*(\cO_\bbX(d))\Big)  = R^i {\rm pr}_* \Big(\cL_n^{\otimes r}\Big) \otimes {\rm pr}^*(\cO_\bbX(d)) \;.
\end{numequation}

Hence \ref{dir_im_1} and \ref{dir_im_2} together give (ii) when $i>0$.

\vskip8pt

(iii) Consider the exact sequence

\begin{numequation}\label{O(1)_ex_seq}
0 \ra \cO_\bbX(-1) \ra \cO_\bbX^{\oplus 2} \ra \cO_\bbX(1) \ra 0 \;,
\end{numequation}

cf. \cite[ch. II, example 8.20.1]{HartshorneA}. Tensorinig with $\cO_\bbX(d)$ gives

\begin{numequation}\label{O(d)_ex_seq}
0 \ra \cO_\bbX(d-1) \ra \Big(\cO_\bbX(d)\Big)^{\oplus 2} \ra \cO_\bbX(d+1) \ra 0 \;,
\end{numequation}

Because the sheaves in this sequence are locally free, taking the pull-back to $\bbX_n$ gives the exact sequence

\begin{numequation}\label{prO(1)_ex_seq}
0 \ra {\rm pr}^*\Big(\cO_\bbX(d-1)\Big) \ra \Big({\rm pr}^*(\cO_\bbX(d))\Big)^{\oplus 2} \ra {\rm pr}^*\Big(\cO_\bbX(d+1)\Big) \ra 0 \;.
\end{numequation}

Tensoring with $\cL_n^{\otimes r}$ gives the exact sequence

\begin{numequation}\label{cL_ex_seq}
0 \ra \cL_n^{\otimes r} \otimes {\rm pr}^*\Big(\cO_\bbX(d-1)\Big) \ra \Big(\cL_n^{\otimes r} \otimes {\rm pr}^*(\cO_\bbX(d))\Big)^{\oplus 2} \ra \cL_n^{\otimes r} \otimes {\rm pr}^*\Big(\cO_\bbX(d+1)\Big) \ra 0 \;.
\end{numequation}

The proof of (iii) is now by induction on $d$, starting with $d=0$ and using \ref{coh_1}. \qed

\vskip12pt

\section{The sheaves $\tsD^{(m)}_{n,k}$ on $\frX_n$}\label{new_sheaves}

\subsection{The sheaves $\tcT^{(m)}_{n,k}$} For $k \ge n$ the scheme $\bbX_n$ is equipped with a right action of the group scheme $\bbG(k)$. This gives rise to a map

\begin{numequation}\label{Lie_map}
\Lie(\bbG(k)) = \vpi^k \frg_\fro \ra H^0(\bbX_n,\cT_{\bbX^{\log}_n})
\end{numequation}

from the Lie algebra $\vpi^k \frg_\fro$ of $\bbG(k)$ to the global sections of the relative logarithmic tangent sheaf $\cT_{\bbX^{\log}_n}$ of $\bbX_n$. Denote by $\tcT_{n,k}$ the $\cO_{\bbX_n}$-submodule of $\cT_{\bbX^{\log}_n}$ generated by the image of \ref{Lie_map}. For $n' \le n$ let

$${\rm pr}_{n,n'}: \bbX_n \ra \bbX_{n'}$$

\vskip8pt

be the blow-up morphism. Denote by $\cT_\bbX$ the relative tangent sheaf on $\bbX = \bbX_0$ over $S$.

\vskip12pt

\begin{lemma}\label{tcT_lemma} (i) $\tcT_{n,k} = \vpi^k ({\rm pr}_{n,0})^*(\cT_{\bbX_0})$. More generally, for $n' \le n$ one has

$$\tcT_{n,k} = \vpi^{k-k'} ({\rm pr}_{n,n'})^*(\tcT_{n',k'}) \;,$$

\vskip8pt

where the latter is to be considered as a subsheaf of $\cT_{\bbX^{\log}_n} \otimes_\fro L$.

(ii) $\tcT_{n,k}$ is locally free of rank one over $\cO_{\bbX_n}$.
\end{lemma}

\Pf (i) We note that $({\rm pr}_{n,0})_*(\cT_{\bbX^{\log}_n})$ is a subsheaf of $\cT_\bbX$, and, in fact, we have inclusions

$$\vpi^n \cT_\bbX \sub ({\rm pr}_{n,0})_*(\cT_{\bbX^{\log}_n}) \sub \cT_\bbX \;,$$

\vskip8pt

where the inclusion on the left follows from \cite[5.2]{PSS2}. For $k \ge n$ there is a commutative diagram

$$\begin{array}{ccccc} \Lie(\bbG(k)) & = & \vpi^k \frg_\fro & \lra & H^0(\bbX_n,\cT_{\bbX^{\log}_n})\\
& & \downarrow  && \downarrow \\
\Lie(\bbG(0)) & = & \frg_\fro & \twoheadrightarrow & H^0(\bbX,\cT_\bbX)
\end{array}$$

\vskip8pt

with injective vertical arrows. The assertion follows.

\vskip8pt

(ii) Immediate consequence of (i). \qed

\vskip12pt

\subsection{The sheaves $\tcD^{(m)}_{n,k}$} Let $\tcD^{(0)}_{n,k}$ be the sheaf of differential operators generated by $\tcT_{n,k}$ over the structure sheaf. Denote by $\tcD^{(m)}_{n,k}$ the sheaf of level $m$ differential operators in $\tcD^{(0)}_{n,k} \otimes_\fro L$. This sheaf is filtered by the subsheaves $\tcD^{(m)}_{n,k;d}$ of operators of order less than or equal to $d$, and we denote the corresponding graded sheaf $\bigoplus_d \tcD^{(m)}_{n,k;d}/\tcD^{(m)}_{n,k;d-1}$ by ${\rm gr}\left(\tcD^{(m)}_{n,k}\right)$.

\vskip12pt

\begin{prop}\label{graded_tcD} (i) For $n' \le n$ one has $({\rm pr}_{n,n'})^*\left(\tcD^{(m)}_{n',k}\right) = \tcD^{(m)}_{n,k}$.

\vskip8pt

(ii) For $n' \le n$ one has $({\rm pr}_{n,n'})_*\left(\tcD^{(m)}_{n,k}\right) = \tcD^{(m)}_{n',k}$.

\vskip8pt

(iii) There is a canonical isomorphism of graded sheaves of algebras

$${\rm gr}\left(\tcD^{(m)}_{n,k}\right) \simeq \Sym(\tcT_{n,k})^{(m)} = \bigoplus_{d \ge 0} \frac{q^{(m)}_d!}{d!}\tcT_{n,k}^{\otimes d} \;.$$

\vskip8pt

(iv) The canonical homomorphism of graded rings

$${\rm gr}\Big(H^0(\bbX_n,\tcD^{(m)}_{n,k})\Big) \lra H^0\left(\bbX_n,{\rm gr}\Big(\tcD^{(m)}_{n,k}\Big)\right)$$

is an isomorphism.

\vskip8pt

(v) The rings ${\rm gr}\Big(H^0(\bbX_n,\tcD^{(m)}_{n,k})\Big)$ and $H^0(\bbX_n,\tcD^{(m)}_{n,k})$ are noetherian.

\vskip8pt

(vi) There is a basis of the topology $\frU$ of $\bbX_n$, consisting of open affine subsets, such that for any $U \in \frU$, the ring $H^0(U,\tcD^{(m)}_{n,k})$ is noetherian.
\end{prop}

\Pf (i) This follows directly from \ref{tcT_lemma}.

\vskip8pt

(ii) The sheaves $\tcD^{(m)}_{0,k;d}$ are locally free of finite rank, and so are the sheaves $\tcD^{(m)}_{n',k;d}$, by (i). We can thus apply the projection formula, together with \ref{direct_im_str_sh}, and get

$$({\rm pr}_{n,n'})_*\Big(\tcD^{(m)}_{n,k;d}\Big) = \tcD^{(m)}_{n',k;d} \;.$$

\vskip8pt

The assertion follows because the direct image commutes with inductive limits on a noetherian space.

\vskip8pt

(iii) We consider the exact sequence on $\bbX$

\begin{numequation}\label{fil_gr_ex_seq_n0} 0 \lra \tcD^{(m)}_{0,k;d-1} \lra \tcD^{(m)}_{0,k;d} \lra \Sym^d(\tcT_{0,k})^{(m)} \lra 0 \;.
\end{numequation}

(Exactness is straightforward in this case.) Now we apply $({\rm pr}_{n,0})^*$ and use (i).

\vskip8pt

(iv) We start again with the case $n=0$ and note that

\begin{numequation}\label{simplification}
\Sym^d(\tcT_{0,k})^{(m)} = \frac{q^{(m)}_d!}{d!} \tcT_{0,k}^{\otimes d} = \frac{q^{(m)}_d!}{d!} \vpi^{kd} \cT_\bbX^{\otimes d} \simeq \cO_\bbX(2d) \;.
\end{numequation}

Now we use the exact sequence \ref{fil_gr_ex_seq_n0} and induction
on $d$ to conclude that the sheaves $\tcD^{(m)}_{0,k;d}$ have vanishing higher cohomology. Hence applying $H^0(-)$ to the sequence \ref{fil_gr_ex_seq_n0} gives an exact sequence and therefore the assertion when $n=0$. Using point (ii) we reduce the case for arbitrary $n$ to the case when $n=0$.

\vskip8pt

(v) We note that $H^0\left(\bbX_n,{\rm gr}\Big(\tcD^{(m)}_{n,k}\Big)\right) = H^0\left(\bbX,{\rm gr}\Big(\tcD^{(m)}_{0,k}\Big)\right)$. Furthermore, by \ref{simplification} we have ${\rm gr}\Big(\tcD^{(m)}_{0,k}\Big) \simeq {\rm gr}\Big(\tcD^{(m)}_{0,0}\Big)$. Because it is well-known that $H^0\left(\bbX,{\rm gr}\Big(\tcD^{(m)}_{0,0}\Big)\right)$ is a finitely generated $H^0\left(\bbX,{\rm gr}\Big(\tcD^{(0)}_{0,0}\Big)\right)$-algebra, and because the latter ring is known to be noetherian (isomorphic to $\fro[x,y,z]/(xz-y^2)$), the ring $H^0\left(\bbX,{\rm gr}\Big(\tcD^{(m)}_{0,0}\Big)\right)$ is noetherian, and so is $H^0\left(\bbX,{\rm gr}\Big(\tcD^{(m)}_{0,k}\Big)\right)$. Now we apply (iv) and can conclude that ${\rm gr}\Big(H^0(\bbX_n,\tcD^{(m)}_{n,k})\Big)$ is noetherian. But then  $H^0(\bbX_n,\tcD^{(m)}_{n,k})$ is noetherian as well.

\vskip8pt

(vi) Take $U$ small enough such that $\tcT_{n,k}|_U$ is free over $\cO_U$. Then ${\rm gr}\Big(H^0(U,\tcD^{(m)}_{n,k})
\Big) \simeq \Sym_{\cO(U)}(\cO(U)[T])^{(m)}$, which is known to be noetherian. It follows that $H^0(U,\tcD^{(m)}_{n,k})$ is noetherian too.\qed

\vskip8pt

\begin{para} {\it Twisting by $\cL_n$.} Recall the very ample line bundle $\cL_n$ from \ref{v_ample_sh_lemma}. In the following we will always use this line bundle to 'twist' $\cO_{\bbX_n}$-modules. If $\cF$ is a $\cO_{\bbX_n}$-module and $r \in \Z$ we thus put

$$\cF(r) = \cF \otimes_{\cO_{\bbX_n}} \cL_n^{\otimes r} \;.$$

\vskip8pt

Some caveat is in order when we deal with sheaves which are equipped with both a left and a right $\cO_{\bbX_n}$-module structure (which may not coincide). For instance, if $\cF_d = \tcD_{n,k;d}^{(m)}$ then we let

$$\cF_d(r) = \tcD_{n,k;d}^{(m)}(r) = \tcD_{n,k;d}^{(m)} \otimes_{\cO_{\bbX_n}} \cL_n^{\otimes r} \;,$$

\vskip8pt

where we consider $\cF_d = \tcD_{n,k;d}^{(m)}$ as a {\it right} $\cO_{\bbX_n}$-module. Similarly we put

$$\tcD_{n,k}^{(m)} \otimes_{\cO_{\bbX_n}} \cL_n^{\otimes r} \;,$$

\vskip8pt

where we consider $\tcD_{n,k}^{(m)}$ as a {\it right} $\cO_{\bbX_n}$-module. Then we have $\tcD_{n,k}^{(m)}(r) = \varinjlim_d \cF_d(r)$. When we consider the associated graded sheaf of $\tcD_{n,k}^{(m)}(r)$, it is with respect to the filtration by the $\cF_d(r)$. The sheaf $\tcD_{n,k}^{(m)}(r)$ is a coherent left $\tcD_{n,k}^{(m)}$-module.
\end{para}

\vskip8pt

\begin{prop}\label{finiteness_twist} (i) $H^0\left(\bbX_n,\Sym \Big({\rm pr}_{n,0}^*(\vpi^k\cO_\bbX(1)^{\oplus 2})\Big)^{(m)}\right)$ is a finitely generated $\fro$-algebra, and in particular noetherian.

\vskip8pt

(ii) For any $r \in \Z$ the $H^0\left(\bbX_n,\Sym\Big({\rm pr}_{n,0}^*(\vpi^k\cO_\bbX(1)^{\oplus 2})\Big)^{(m)}\right)$-module

$$H^0\left(\bbX_n,\Sym\Big({\rm pr}_{n,0}^*(\vpi^k\cO_\bbX(1)^{\oplus 2})\Big)^{(m)}(r)\right)$$

\vskip8pt

is finitely generated.

\vskip8pt

(iii)  For any $r \in \Z$ the $H^0\Big(\bbX_n,\Sym(\tcT_{n,k})^{(m)}\Big)$-module $H^0\Big(\bbX_n,\Sym(\tcT_{n,k})^{(m)}(r)\Big)$ is finitely generated.

\end{prop}

\Pf (i) We note first that by \ref{graded_tcD}

$$H^0\left(\bbX_n,\Sym\Big({\rm pr}_{n,0}^*(\vpi^k\cO_\bbX(1)^{\oplus 2})\Big)^{(m)}\right) = H^0\left(\bbX,\Sym\Big(\vpi^k\cO_\bbX(1)^{\oplus 2}\Big)^{(m)}\right) \;.$$

\vskip8pt

We also note that

$$\Sym\Big(\vpi^k\cO_\bbX(1)^{\oplus 2}\Big)^{(m)} \simeq \Sym\Big(\cO_\bbX(1)^{\oplus 2}\Big)^{(m)} \;.$$

\vskip8pt

The remaining arguments are exactly as in the proof of \cite[4.2.1]{Huyghe97}.

\vskip8pt

(ii) We first note that

$$H^0\left(\bbX_n,\Sym\Big({\rm pr}_{n,0}^*(\vpi^k\cO_\bbX(1)^{\oplus 2})\Big)^{(m)}(r)\right)$$

\vskip8pt

is isomorphic to

$$H^0\left(\bbX,\Sym\Big(\vpi^k\cO_\bbX(1)^{\oplus 2}\Big)^{(m)} \otimes ({\rm pr}_{n,0})_* \Big(\cL^{\otimes r}\Big)\right)$$

\vskip8pt

which is in turn isomorphic to

$$H^0\left(\bbX,\Sym\Big(\cO_\bbX(1)^{\oplus 2}\Big)^{(m)} \otimes ({\rm pr}_{n,0})_* \Big(\cL^{\otimes r}\Big)\right)$$

\vskip8pt

It follows from \ref{invert_p_3} that

$$({\rm pr}_{n,0})_* \Big(\cL^{\otimes r}\Big) \sub \cO_\bbX(r \cdot a_0) \;.$$

\vskip8pt

Therefore,

$$H^0\left(\bbX,\Sym\Big(\cO_\bbX(1)^{\oplus 2}\Big)^{(m)} \otimes ({\rm pr}_{n,0})_* \Big(\cL^{\otimes r}\Big)\right)$$

\vskip8pt

can be considered as a submodule of

$$H^0\left(\bbX,\Sym\Big(\cO_\bbX(1)^{\oplus 2}\Big)^{(m)} \otimes \cO_\bbX(r \cdot a_0)\right) \;.$$

\vskip8pt

By \cite[4.2.1]{Huyghe97}, the latter is a finitely generated module over $H^0\left(\bbX,\Sym\Big(\cO_\bbX(1)^{\oplus 2}\Big)^{(m)}\right)$. As this ring is noetherian, cf. (i), the submodule considered before is finitely generated too.

\vskip8pt

(iii) We note first that by \ref{graded_tcD}

$$H^0\Big(\bbX_n,\Sym(\tcT_{n,k})^{(m)}\Big) = H^0\Big(\bbX,\Sym(\tcT_{0,k})^{(m)}\Big) $$

\vskip8pt

and

$$H^0\Big(\bbX_n,\Sym(\tcT_{n,k})^{(m)}(r)\Big) = H^0\Big(\bbX,\Sym(\tcT_{0,k})^{(m)} \otimes ({\rm pr}_{n,0})_* \Big(\cL^{\otimes r}\Big) \Big) \;.$$

\vskip8pt

We have obvious quotient maps coming from taking $d=1$ in (\ref{O(d)_ex_seq})

$$H^0\left(\bbX,\Sym\Big(\vpi^k\cO_\bbX(1)^{\oplus 2}\Big)^{(m)}\right) \twoheadrightarrow H^0\Big(\bbX,\Sym(\tcT_{0,k})^{(m)}\Big)$$

\vskip8pt

and

$$H^0\left(\bbX,\Sym\Big(\vpi^k\cO_\bbX(1)^{\oplus 2}\Big)^{(m)} \otimes ({\rm pr}_{n,0})_* \Big(\cL^{\otimes r}\Big)\right) \twoheadrightarrow H^0\Big(\bbX,\Sym(\tcT_{0,k})^{(m)} \otimes ({\rm pr}_{n,0})_* \Big(\cL^{\otimes r}\Big)\Big) \;,$$

\vskip8pt

and the module structure of $H^0\Big(\bbX,\Sym(\tcT_{0,k})^{(m)} \otimes ({\rm pr}_{n,0})_* \Big(\cL^{\otimes r}\Big)\Big)$
over \linebreak $H^0\Big(\bbX,\Sym(\tcT_{0,k})^{(m)}\Big)$ comes from the module structure of

$$H^0\left(\bbX,\Sym\Big(\vpi^k\cO_\bbX(1)^{\oplus 2}\Big)^{(m)} \otimes ({\rm pr}_{n,0})_* \Big(\cL^{\otimes r}\Big) \right)$$

\vskip8pt

over $H^0\left(\bbX,\Sym\Big(\vpi^k\cO_\bbX(1)^{\oplus 2}\Big)^{(m)}\right)$ via these quotient maps. Assertion (iii) now follows from assertion (ii). \qed

\vskip12pt

\subsection{Global sections of $\tcD^{(m)}_{n,k}$, $\tsD^{(m)}_{n,k}$, and $\tsD^\dagger_{n,k,\Q}$}\label{global_sec}

\begin{para} {\it Divided power enveloping algebras.} We denote by $U^{(m)}_{\vpi^k}(\frg_\fro)$ the $\fro$-subalgebra of $U(\frg) = U(\frg_\fro) \otimes_\fro L$ generated by the elements

\begin{numequation}\label{generators_n}
q^{(m)}_{\nu_1}! \frac{(\vpi^k e)^{\nu_1}}{\nu_1!}  \cdot q^{(m)}_{\nu_2}! \vpi^{k \nu_2}{h_1 \choose \nu_2} \cdot q^{(m)}_{\nu_3}! \vpi^{k \nu_3} {h_2 \choose \nu_3} \cdot q^{(m)}_{\nu_4}! \frac{(\vpi^k f)^{\nu_4}}{\nu_4!} \;.
\end{numequation}

We considered the same algebra in \cite[2.7]{PSS2} where it is denoted by $U^{(m)}(p^k \frg_\Zp)$ (there we have $\vpi = p$). Because this latter notation is potentially misleading, we have changed it here. We denote by $\widehat{U}^{(m)}_{\vpi^k}(\frg_\fro)$ the $p$-adic completion of $U^{(m)}_{\vpi^k}(\frg_\fro)$. In the following the subscript $\theta_0$ indicates the same central reduction as in \cite[sec. 3]{PSS2}.
\end{para}

\vskip8pt

\begin{prop}\label{global_sections_tcD} There is a canonical injective homomorphism of $\fro$-algebras

\begin{numequation}\label{xi_uncompl} \xi^{(m)}_{n,k}:
U^{(m)}_{\vpi^k}(\frg_\fro)_{\theta_0} \ra H^0(\bbX_n, \tcD^{(m)}_{n,k}) \;,
\end{numequation}

and $H^0(\bbX_n, \tcD^{(m)}_{n,k})$ is a finitely generated module over $U^{(m)}_{\vpi^k}(\frg_\fro)_{\theta_0}$ via $\xi^{(m)}_{n,k}$.
In particular, the ring $H^0(\bbX_n, \tcD^{(m)}_{n,k})$ is noetherian.
Moreover, $\coker(\xi^{(m)}_{n,k})$ is annihilated by a finite power of $p$.
\end{prop}

\Pf By \ref{graded_tcD} we have $H^0(\bbX_n, \tcD^{(m)}_{n,k}) = H^0(\bbX_0, \tcD^{(m)}_{0,k})$. The remaining arguments are as in \cite[3.7]{PSS2}. \qed

\vskip8pt

Now let $\tsD^{(m)}_{n,k}$ be the $p$-adic completion of $\tcD^{(m)}_{n,k}$, which we will always consider as a sheaf on the formal scheme $\frX_n$. We also put $\tsD^{(m)}_{n,k,\Q} = \tsD^{(m)}_{n,k} \otimes_\Z \Q$, and

\begin{numequation}\label{tsD_dagger}
\tsD^\dagger_{n,k} = \varinjlim_m \tsD^{(m)}_{n,k} \;,
\end{numequation}

and

\begin{numequation}\label{tsD_dagger_Q}
\tsD^\dagger_{n,k,\Q} = \tsD^\dagger_{n,k} \otimes_\Z \Q = \varinjlim_m \tsD^{(m)}_{n,k,\Q} \;.
\end{numequation}

\begin{prop}\label{global_sec_tsD} (i)  There is a basis of the topology $\frU$ of $\bbX_n$, consisting of open affine subsets, such that for any $U \in \frU$, the ring $H^0(U,\tsD^{(m)}_{n,k})$ is noetherian.

\vskip8pt

(ii) The transition map $\tsD^{(m)}_{n,k,\Q} \ra \tsD^{(m+1)}_{n,k,\Q}$ is flat.

\vskip8pt

(iii) The sheaf $\tsD^\dagger_{n,k,\Q}$ is coherent.

\vskip8pt

(iv) The ring $H^0(\frX_n,\tsD^{(m)}_{n,k,\Q})$ is canonically isomorphic to the noetherian $L$-algebra \linebreak $\widehat{U}^{(m)}_{\vpi^k}(\frg_\fro)_{\Q,\theta_0}$.

\vskip8pt

(v) The ring $H^0(\frX_n,\tsD^\dagger_{n,k,\Q})$ is canonically isomorphic to the coherent $L$-algebra \linebreak $\cD^{\rm an}(\bbG(k)^\circ)_{\theta_0}$, where $\cD^{\rm an}(\bbG(k)^\circ)$ is the analytic distribution algebra in the sense of Emerton, cf. \cite[ch. 5]{EmertonA}.
\end{prop}

\Pf (i) Let $U$ be open affine and such that $\tcT_{n,k}|_U$ is free over $\cO_U$. Then use that $H^0(U,\tsD^{(m)}_{n,k})$ is the $p$-adic completion of $H^0(U,\tcD^{(m)}_{n,k})$. (The proof is similar to the one given in \ref{completion}.) Because the latter is noetherian, cf. \ref{graded_tcD}, the former is noetherian as well.

\vskip8pt

(ii) Can be proved as in \cite[sec. 3.5]{BerthelotDI}.

\vskip8pt

(iii) Follows from (i) and (ii).

\vskip8pt

(iv) Follows from \ref{completion} (iii) and \ref{global_sections_tcD}. (Note that using \ref{completion} here does not introduce a circular argument.)

\vskip8pt

(v) Follows from (iv) and the description of the analytic distribution algebra as given in \cite[ch. 5]{EmertonA}. \qed

\vskip12pt

\section{Localization on $\frX_n$ via $\tsD^\dagger_{n,k,\Q}$}\label{loc_n}

The general line of arguments follows fairly closely \cite{Huyghe97}. The numbers $k \ge n \ge 0$ are fixed throughout this section, and ${\rm pr} = {\rm pr}_{n,0}: \bbX_n \ra \bbX = \bbX_0$ is the blow-up morphism.

\vskip12pt

\subsection{Cohomology of coherent $\tcD^{(m)}_{n,k}$-modules}

\begin{lemma}\label{van_coh} Let $\cE$ be a quasi-coherent $\cO_{\bbX_n}$-module on $\bbX_n$. Then, for all $i \ge 2$ one has $H^i(\bbX_n,\cE) =0$.
\end{lemma}

\Pf Since $\cE$ is the inductive limit of its coherent $\cO_{\bbX_n}$-submodules, and because cohomology commutes with inductive limits on a noetherian space, we may assume that $\cE$ is coherent. Let $f: \bbX_n \ra \Spec(\fro)$ be the structure morphism. The sheaves $R^j f_*(\cE)$ are coherent on $\Spec(\fro)$, cf. \cite[ch. III, 8.8]{HartshorneA}, and have thus vanishing cohomology in positive degrees. But since the fibers of $f$ are one-dimensional one has $R^j f_*(\cE) = 0$ for $j \ge 2$, cf. \cite[ch. III, 11.2]{HartshorneA}. We can now conclude by the Leray spectral sequence for $f$. \qed

\vskip8pt

\begin{prop}\label{vanishing_coh_Dnk} Let $r_0$ be as in \ref{v_ample_sh_lemma}.
Then, for all $r \ge r_0$ and all $i \ge 1$ one has

\begin{numequation}\label{vanishing_coh_Dnk_disp}
H^i\left(\bbX_n,\tcD_{n,k}^{(m)}(r)\right) = 0 \;.
\end{numequation}

\end{prop}

\Pf Note that

\begin{numequation}\label{gr_0}
{\rm gr}_0\left(\tcD_{n,k}^{(m)}(r)\right) = \cF_0(r) = \cL_n^{\otimes r} \;.
\end{numequation}

For $d>0$ we consider the exact sequence

\begin{numequation}\label{ex_fil_seq_1}
0 \ra \cF_{d-1} \ra \cF_d  \ra {\rm gr}_d\left(\tcD_{n,k}^{(m)}\right) \ra 0
\end{numequation}

\vskip8pt

from which we deduce the exact sequence

\begin{numequation}\label{ex_fil_seq_2}
0 \ra \cF_{d-1}(r) \ra \cF_d(r)  \ra {\rm gr}_d\left(\tcD_{n,k}^{(m)}\right)(r) \ra 0
\end{numequation}

\vskip8pt

because tensoring with a line bundle is an exact functor. It thus follows that

$${\rm gr}_d\left(\tcD_{n,k}^{(m)}(r)\right) \simeq \frac{q^{(m)}_d!}{d!} \vpi^{dk} {\rm pr}^*(\cT_\bbX) \otimes \cL_n^{\otimes r} \simeq  {\rm pr}^*\Big(\cO_\bbX(2d)\Big) \otimes \cL_n^{\otimes r} \;.$$

\vskip8pt

Hence by \ref{v_ample_sh_lemma} we find that the cohomology of this sheaf vanishes for all $d \ge 0$. Using the sequence \ref{ex_fil_seq_2} we can then deduce by induction on $d$ that for all $i>0$

$$H^i(\bbX_n, \cF_d(r)) = 0 \;.$$

\vskip8pt

Because cohomology commutes with inductive limits on a noetherian scheme we obtain the asserted vanishing result. \qed

\vskip8pt

\begin{prop}\label{surjection} Let $\cE$ be a coherent $\tcD^{(m)}_{n,k}$-module.

\vskip8pt

(i) There is $r_1(\cE) \in \Z$ such that for all $r \ge r_1(\cE)$ there is $s \in \Z_{\ge 0}$ and an epimorphism of $\tcD^{(m)}_{n,k}$-modules

$$\Big(\tcD^{(m)}_{n,k}(-r)\Big)^{\oplus s} \twoheadrightarrow \cE \;.$$

\vskip8pt

(ii) There is $r_2(\cE) \in \Z$ such that for all $r \ge r_2(\cE)$ and all $i >0$

$$H^i\Big(\bbX_n, \cE(r)\Big) = 0 \;.$$

\vskip8pt
\end{prop}

\Pf (i) As $\bbX_n$ is a noetherian scheme, $\cE$ is the inductive limit of its coherent subsheaves.  There is thus a coherent $\cO_{\bbX_n}$-submodule $\cF \sub \cE$ which generates $\cE$ as a $\tcD^{(m)}_{n,k}$-module, i.e., an epimorphism of sheaves

$$\tcD^{(m)}_{n,k} \otimes_{\cO_{\bbX_n}} \cF \stackrel{\alpha}{\lra} \cE \;,$$

\vskip8pt

where $\tcD^{(m)}_{n,k}$ is considered with its right $\cO_{\bbX_n}$-module structure. Next, there is $r_1(\cE)$ such that for all $r \ge r_1(\cE)$ the sheaf

$$\cF(r) = \cF \otimes_{\cO_{\bbX_n}} \cL_n^{\otimes r}$$

\vskip8pt

is generated by its global sections. Hence there is $s > 0$ and an epimorphism $\cO_{\bbX_n}^{\oplus s} \twoheadrightarrow \cF(r)$, and thus an epimorphism of $\cO_{\bbX_n}$-modules

$$\left(\cO_{\bbX_n}(-r)\right)^{\oplus s} \twoheadrightarrow \cF \;.$$

\vskip8pt

From this morphism get an epimorphism of $\tcD^{(m)}_{n,k}$-modules

$$\left(\tcD^{(m)}_{n,k}(-r)\right)^{\oplus s}  = \tcD^{(m)}_{n,k} \otimes_{\cO_{\bbX_n}} \left(\cO_{\bbX_n}(-r)\right)^{\oplus s} \twoheadrightarrow \tcD^{(m)}_{n,k} \otimes_{\cO_{\bbX_n}} \cF  \stackrel{\alpha}{\lra} \cE \;.$$

\vskip8pt

(ii) We prove this by descending induction on $i$. When $i>1$ there is nothing to show, cf. \ref{van_coh}. We will now deduce it for $i=1$. Using (i) we find an epimorphism of $\tcD^{(m)}_{n,k}$-modules

$$\beta: \left(\tcD^{(m)}_{n,k}\right)^{\oplus s}  \twoheadrightarrow \cE\Big(r_1(\cE)\Big) \;.$$

\vskip8pt

By \ref{graded_tcD}, the kernel $\cR = \ker(\beta)$ is a coherent $\tcD^{(m)}_{n,k}$-module and we have an exact sequence

$$0 \ra \cR \ra \left(\tcD^{(m)}_{n,k}\right)^{\oplus s} \ra \cE\Big(r_1(\cE)\Big) \ra 0 \;,$$

\vskip8pt

which gives for any $t \in \Z$ the exact sequence

\begin{numequation}\label{ex_seq_van_coh} 0 \ra \cR(t) \ra \left(\tcD^{(m)}_{n,k}(t)\right)^{\oplus s} \ra \cE\Big(t+r_1(\cE)\Big) \ra 0 \;.
\end{numequation}

Now choose $r_2(\cE) = \max\{r_0+r_1(\cE),r_2(\cR)\}$, where $r_0$ is as in \ref{vanishing_coh_Dnk}.
Then we can conclude that from the long exact cohomology sequence associated to \ref{ex_seq_van_coh} that for $r \ge r_2(\cE)$

$$H^1(\bbX_n,\cE(r)) = 0 \;.$$

\vskip8pt

\qed

\vskip8pt

\begin{prop}\label{finite_power} Fix $r \in \Z$.

\vskip8pt

(i) There is $c_1 = c_1(r) \in \Z_{\ge 0}$ such that for all $i>0$ the sheaf $R^i {\rm pr}_*\Big(\tcD_{n,k}^{(m)}(r)\Big)$ on $\bbX$ is annihilated by $p^{c_1}$. Furthermore, it is a skyscraper sheaf, and it is a coherent module over $\tcD^{(m)}_{0,k}$.

\vskip8pt

(ii) There is $c_2 = c_2(r) \in \Z_{\ge 0}$ such that for all $i>0$ the cohomology group $H^i(\bbX_n,\tcD_{n,k}^{(m)}(r))$ is annihilated by $p^{c_2}$. Furthermore, it is a finitely generated module over $H^0(\bbX_n,\tcD_{n,k}^{(m)})$.

\vskip8pt

(iii) Let $\cE$ be a coherent $\tcD_{n,k}^{(m)}$-module. There is $c_3 = c_3(\cE) \in \Z_{\ge 0}$ such that for all $i>0$ the cohomology group $H^i(\bbX_n,\cE)$ is annihilated by $p^{c_3}$. Furthermore, it is a finitely generated module over $H^0(\bbX_n,\tcD_{n,k}^{(m)})$.

\end{prop}

\Pf (i) The direct image functor commutes with inductive limits, and we will thus consider the sheaves

$$\tcD_{n,k;d}^{(m)}(r) = \tcD_{n,k;d}^{(m)} \otimes_{\cO_{\bbX_n}} \cL_n^{\otimes r} \;,$$

\vskip8pt

where, as before, we consider $\tcD_{n,k;d}^{(m)}$ as a {\it right} $\cO_{\bbX_n}$-module. Because the fibers of ${\rm pr}$ are of dimension zero or one, we have

$$R^i {\rm pr}_*\Big(\tcD_{n,k;d}^{(m)}(r)\Big) = 0$$

\vskip8pt

for $i>1$, by \cite[ch. III, 11.2]{HartshorneA}. Now let $i=1$. Using the projection formula and \ref{graded_tcD} we deduce that

\begin{numequation}\label{dir_im_tcDr}
\begin{array}{rcl} R^1 {\rm pr}_*\Big(\tcD_{n,k;d}^{(m)}(r)\Big) & = & \tcD_{0,k;d}^{(m)} \otimes R^1 {\rm pr}_*\Big(\cL_n^{\otimes r}\Big)  \\
&&\\
& = & \tcD_{0,k;d}^{(m)} \otimes \cO_\bbX(r \cdot a_0) \otimes  R^1 {\rm pr}_*\Big(\cO_{\bbX_n/\bbX}(1)^{\otimes r}\Big)\;.
\end{array}
\end{numequation}

By \ref{invert_p_2} and \ref{direct_im_str_sh} we can conclude that $R^1 {\rm pr}_*\Big(\cO_{\bbX_n/\bbX}(1)^{\otimes r}\Big)$ is a torsion sheaf, and because it is coherent, it is thus annihilated by a $p^{c_1}$ for some $c_1=c_1(r) \ge 0$. Furthermore, because ${\rm pr}$ is an isomorphism from ${\rm pr}^{-1}(\bbX \setminus \bbX(\Fq))$ to $\bbX \setminus \bbX(\Fq)$, the sheaf $R^1 {\rm pr}_*\Big(\tcD_{n,k;d}^{(m)}(r)\Big)$ is supported on $\bbX(\Fq)$. It follows from \ref{dir_im_tcDr} that

$$R^1 {\rm pr}_*\Big(\tcD_{n,k}^{(m)}(r)\Big) = \tcD_{0,k}^{(m)} \otimes \cO_\bbX(r \cdot a_0) \otimes  R^1 {\rm pr}_*\Big(\cO_{\bbX_n/\bbX}(1)^{\otimes r}\Big)\;,$$

\vskip8pt

which shows that it is a coherent $\tcD_{0,k}^{(m)}$-module.

\vskip8pt

(ii) By \ref{van_coh} there is nothing to show for $i>1$. In order to compute $H^1$ we consider the exact sequence of low-degree terms for the Leray spectral sequence for the morphism ${\rm pr}$ and the sheaf $\tcD_{n,k;d}^{(m)}(r)$ on $\bbX_n$:

\begin{numequation}\label{ex_seq_low_deg_terms}
\begin{array}{ccccl}
0 & \ra & H^1\left(\bbX,\tcD^{(m)}_{0,k;d} \otimes {\rm pr}_*\Big(\cL_n^{\otimes r}\Big)\right) & \ra & H^1\Big(\bbX_n,\tcD^{(m)}_{n,k;d}(r)\Big) \\
&&&& \\
&&& \ra & H^0\left(\bbX, \tcD^{(m)}_{0,k;d} \otimes R^1 {\rm pr}_*\Big(\cL_n^{\otimes r}\Big)\right) \ra 0 \;.
\end{array}
\end{numequation}

In order to show the first assertion in (ii) we use (i) and are thus reduced to showing that

$$H^i\left(\bbX, \tcD_{0,k}^{(m)} \otimes {\rm pr}_*\Big(\cL_n^{\otimes r}\Big)\right)$$

\vskip8pt

is annihilated by a fixed power of $p$. We will first show that it is a finitely generated $\fro$-module, and then use the fact that coherent $\tcD_{0,k,\Q}^{(m)}$-modules have vanishing higher cohomology, by \cite{BB81}.
We put $\cF_d = \tcD_{0,k;d}^{(m)}$ and

$$\cF_{d,r} = \tcD_{0,k;d}^{(m)} \otimes_{\cO_\bbX} {\rm pr}_*\Big(\cL_n^{\otimes r}\Big) \;,$$

\vskip8pt

where, as before, we consider $\tcD_{0,k;d}^{(m)}$ as a {\it right} $\cO_\bbX$-module.\footnote{We note that, with the convention introduced above, the sheaf $\cF_d(r)$ is not the same as $\cF_{d,r}$.} More generally, when $\cG$ is a $\cO_\bbX$-module, we put $\cG_r = \cG \otimes {\rm pr}_*\Big(\cL_n^{\otimes r}\Big)$. We let $d,e \ge 0$ and consider the exact sequence

\begin{numequation}\label{fil_1}
0 \lra \cF_d \lra \cF_{d+e} \lra \cF_{d+e}/\cF_d \lra 0 \;.
\end{numequation}

Because the sheaves in this sequence are locally free, tensoring with ${\rm pr}_*\Big(\cL_n^{\otimes r}\Big)$ gives an exact sequence

\begin{numequation}\label{fil_1_twist}
0 \lra \cF_{d,r} \lra \cF_{d+e,r} \lra \Big(\cF_{d+e}/\cF_d\Big)_r \lra 0 \;.
\end{numequation}

and thus an exact sequence

\begin{numequation}\label{fil_2}
0 \lra \Big(\cF_{d+e}/\cF_d\Big)_r \lra \Big(\cF_{d+e+1}/\cF_d\Big)_r \lra \Big(\cF_{d+e+1}/\cF_{d+e}\Big)_r \lra 0 \;.
\end{numequation}

It follows from \ref{fil_1}, together with \ref{graded_tcD}, that the sheaf on the right of \ref{fil_2} is isomorphic to

$$\begin{array}{rcl}\frac{q^{(m)}_d!}{d!} \tcT_{0,k}^{\otimes (d+e+1)} \otimes {\rm pr}_*\Big(\cL_n^{\otimes r}\Big) & = & \frac{q^{(m)}_d!}{d!} \vpi^{k(d+e+1)} \cT_\bbX^{\otimes (d+e+1)} \otimes {\rm pr}_*\Big(\cL_n^{\otimes r}\Big) \\
&&\\
& \simeq & \cO_\bbX(2(d+e+1)) \otimes {\rm pr}_*\Big(\cL_n^{\otimes r}\Big) \;.
\end{array}$$

\vskip8pt

Let $d_0$ be such that for all $d \ge d_0$ the sheaf

$$\cO_\bbX(2(d+1)) \otimes {\rm pr}_*\Big(\cL_n^{\otimes r}\Big)$$

\vskip8pt

is acyclic. Using this fact and \ref{fil_2}, we see by induction on $e$ that all sheaves

$$\cF_{d_0+e,r}/\cF_{d_0,r} \;,$$

\vskip8pt

$e \ge 0$, are acyclic as well. The long exact cohomology sequence associated to \ref{fil_1_twist} shows that, for all $e \ge 0$, the canonical map

\begin{numequation}\label{surj_inj}
H^i(\bbX,\cF_{d_0,r}) \lra H^i(\bbX,\cF_{d_0+e,r})
\end{numequation}

is an isomorphism for $i \ge 2$ and surjective for $i=1$. This proves that

$$H^i\left(\bbX, \tcD_{0,k}^{(m)} \otimes {\rm pr}_*\Big(\cL_n^{\otimes r}\Big)\right)$$

\vskip8pt

is a finitely generated $\fro$-module, as cohomology commutes with inductive limits on a noetherian scheme, and $\cF_{d_0}$ is coherent over $\cO_\bbX$.

\vskip8pt

By \ref{graded_tcD} the sheaf $\tcD_{0,k}^{(m)} \otimes {\rm pr}_*\Big(\cL_n^{\otimes r}\Big)$ is a coherent $\tcD_{0,k}^{(m)}$-module which is quasi-coherent as $\cO_\bbX$-module.
By \cite{BB81}, coherent $\tcD_{0,k,\Q}^{(m)}$-modules which are quasi-coherent as $\cO_{\bbX,\Q}$-modules have vanishing higher cohomology. This implies that

$$H^i\left(\bbX, \tcD_{0,k}^{(m)} \otimes {\rm pr}_*\Big(\cL_n^{\otimes r}\Big)\right)$$

\vskip8pt

is annihilated by a fixed power of $p$.

\vskip8pt

Analogous to \ref{ex_seq_low_deg_terms} we have an exact sequence

$$\begin{array}{ccccl}
0 & \ra & H^1\left(\bbX,\tcD^{(m)}_{0,k} \otimes {\rm pr}_*\Big(\cL_n^{\otimes r}\Big)\right) & \ra & H^1\Big(\bbX_n,\tcD^{(m)}_{n,k}(r)\Big) \\
&&&& \\
&&& \ra & H^0\left(\bbX, \tcD^{(m)}_{0,k} \otimes R^1 {\rm pr}_*\Big(\cL_n^{\otimes r}\Big)\right) \ra 0 \;.
\end{array}$$

\vskip8pt

As we have just seen, the cohomology group on the left is a finite group of $p$-power order. Now we consider the cohomology group on the right. Because $R^1 {\rm pr}_*\Big(\cL_n^{\otimes r}\Big)$ is a skyscraper sheaf, the canonical homomorphism

$$H^0\left(\bbX, \tcD^{(m)}_{0,k}\right) \otimes_\fro H^0\left(\bbX,R^1 {\rm pr}_*\Big(\cL_n^{\otimes r}\Big)\right) \lra H^0\left(\bbX, \tcD^{(m)}_{0,k} \otimes R^1 {\rm pr}_*\Big(\cL_n^{\otimes r}\Big)\right) \;,$$

\vskip8pt

is easily seen to be surjective. This shows that $H^0\left(\bbX, \tcD^{(m)}_{0,k} \otimes R^1 {\rm pr}_*\Big(\cL_n^{\otimes r}\Big)\right)$ is a finitely generated module over $H^0(\bbX,\tcD^{(m)}_{0,k})$, but the latter is equal to $H^0(\bbX_n,\tcD^{(m)}_{n,k})$, by \ref{graded_tcD}. Therefore, $H^1\Big(\bbX_n,\tcD^{(m)}_{n,k}(r)\Big)$ is also a finitely generated $H^0(\bbX_n,\tcD^{(m)}_{n,k})$-module.

\vskip8pt

(iii) The proof proceeds by descending induction on $i$. For $i \ge 2$ there is nothing to show by \ref{van_coh}. Using \ref{surjection} we obtain a surjection of $\tcD^{(m)}_{n,k}$-modules

$$\beta: \Big(\tcD^{(m)}_{n,k}(-r)\Big)^{\oplus s} \twoheadrightarrow \cE $$

\vskip8pt

for suitable $s$ and $r$. The $\tcD^{(m)}_{n,k}$-module $\cR = \ker(\beta)$ is coherent too, by \ref{graded_tcD}. Because the statement is true for $i=2$ (for all coherent $\tcD^{(m)}_{n,k}$-modules, hence for $\cR$), we get from the exact sequence

$$H^1\left(\bbX_n,\left(\tcD^{(m)}_{n,k}(-r)\right)^{\oplus s}\right) \ra H^1(\bbX_n,\cE) \ra H^2(\bbX_n,\cR) = 0 \;,$$

\vskip8pt

together with part (ii), that the assertion true for $i=1$. \qed

\vskip8pt

\begin{prop}\label{fg_cok} Fix $r \in \Z$.

\vskip8pt

(i) The canonical injective homomorphism

$$H^0\Big(\bbX_n,\tcD^{(m)}_{n,k;d}(r)\Big) \Big/ H^0\Big(\bbX_n,\tcD^{(m)}_{n,k;d-1}(r)\Big) \lra H^0\Big(\bbX_n,{\rm gr}_d\Big(\tcD^{(m)}_{n,k}(r)\Big)\Big)$$

\vskip8pt

is an isomorphism for $d \gg 0$.

\vskip8pt

(ii) The canonical map of graded abelian groups

$${\rm gr}\Big(H^0\Big(\bbX_n,\tcD^{(m)}_{n,k}(r)\Big)\Big) \lra H^0\Big(\bbX_n,{\rm gr}\Big(\tcD^{(m)}_{n,k}(r)\Big)\Big)$$

\vskip8pt

is injective and its cokernel is finitely generated over $\fro$.
\end{prop}

\Pf The arguments are similar to those used in the proof of \ref{finite_power} (ii). As there, put

$$\cF_{d,r} = \tcD^{(m)}_{0,k;d} \otimes_{\cO_\bbX} {\rm pr}_*\Big(\cL_n^{\otimes r}\Big) \;.$$

\vskip8pt

Note that $H^0\Big(\bbX_n,\tcD^{(m)}_{n,k;d}(r)\Big) = H^0(\bbX,\cF_{d,r})$, by \ref{graded_tcD}. From \ref{fil_1_twist} we see that there is a map of exact sequences

$$\begin{array}{ccccccccc}
0 & \ra & \cF_{d-1,r} & \ra & \cF_{d,r} & \ra & \Big(\cF_d/\cF_{d-1}\Big)_r & \ra & 0 \\
&& \downarrow && \downarrow && \downarrow && \\
0 & \ra & \cF_{d,r} & \ra & \cF_{d+1,r} & \ra & \Big(\cF_{d+1}/\cF_d\Big)_r & \ra & 0 \\
\end{array}$$

\vskip8pt

where the vertical arrow on the right is the zero map. From this we get a map of long exact sequences

\begin{numequation}\label{map_les}
\xymatrix{
0 \ar[r] & H^0(\cF_{d-1,r}) \ar[r] \ar[d] & H^0(\cF_{d,r}) \ar[r] \ar[d] & H^0\Big((\cF_d/\cF_{d-1})_r\Big) \ar[r] \ar[d]^{\sigma_d} & H^1(\cF_{d-1,r}) \ar[d]^{\tau_d}\\
0 \ar[r] & H^0(\cF_{d,r}) \ar[r] & H^0(\cF_{d+1,r}) \ar[r] & H^0\Big((\cF_{d+1}/\cF_d)_r\Big) \ar[r] & H^1(\cF_{d,r})\\
}
\end{numequation}

(We write $H^i(-)$ for $H^i(\bbX,-)$ here.) There is $d_0$ such that for all $d \ge d_0$ one has

$$H^1\Big(\bbX,(\cF_d/\cF_{d-1})_r\Big) = 0 \;.$$

\vskip8pt

and the map $H^1(\cF_{d-1,r}) \ra H^1(\cF_{d,r})$ is thus surjective for $d \ge d_0$. Hence $H^1(\cF_{d_0-1,r}) \ra H^1(\cF_{d,r})$ is surjective for $d \ge d_0$. For $d \ge d_0$ consider the submodules

$$N'_d = \ker\Big(H^1(\cF_{d_0-1,r}) \ra H^1(\cF_{d,r})\Big)$$

\vskip8pt

and

$$N_d = \ker\Big(H^1(\cF_{d,r}) \ra H^1(\cF_{d+1,r})\Big) \;.$$

\vskip8pt

The submodules $N'_d$ form an increasing sequence, and since $H^1\Big(\cF_{d_0-1,r}\Big)$ is noetherian, this sequence must be stationary, i.e., there exists $d_1 \ge d_0$ such that $N'_d = N'_{d+1}$ for all $d \ge d_1$. But then $N_d = 0$ for all $d \ge d_1$, and the maps $H^1(\cF_{d,r}) \ra H^1(\cF_{d+1,r})$ are isomorphisms for $d \ge d_1$. Now we consider the map $\tau_d$ in \ref{map_les}. For $d \ge d_1+1$ the map $\tau_d$ is an isomorphism, by what we have just seen. Since the map $\sigma_d$ in \ref{map_les} is the zero map, we conclude that the map

$$H^0\Big((\cF_d/\cF_{d-1})_r\Big) \ra H^1(\cF_{d,r})$$

\vskip8pt

is the zero map for $d \ge d_1+1$, and the sequence

$$0 \lra H^0(\cF_{d-1,r})  \lra H^0(\cF_{d,r}) \lra  H^0\Big((\cF_d/\cF_{d-1})_r\Big) \lra 0$$

\vskip8pt

is therefore exact. This proves (i), and (ii) follows immediately from (i). \qed

\begin{cor}\label{fg_H0} Let $\cE$ be a coherent $\tcD_{n,k}^{(m)}$-module. Then $H^0(\bbX_n,\cE)$ is a finitely generated module over $H^0(\bbX_n,\tcD_{n,k}^{(m)})$.
\end{cor}

\Pf Using \ref{surjection} we get an exact sequence of coherent $\tcD_{n,k}^{(m)}$-modules

$$0 \lra \cR \lra \Big(\tcD^{(m)}_{n,k}(-r)\Big)^{\oplus s} \lra \cE \lra 0 \;.$$

\vskip8pt

passing to the long exact cohomology sequence gives

$$0 \ra H^0(\bbX_n,\cR) \ra H^0\Big(\bbX_n,\Big(\tcD^{(m)}_{n,k}(-r)\Big)^{\oplus s}\Big) \ra H^0(\bbX_n,\cE) \stackrel{\delta}{\lra} H^1(\bbX_n,\cR) \;.$$

\vskip8pt

By \ref{finite_power} the group on the right is a finitely generated $H^0(\bbX_n,\tcD^{(m)}_{n,k})$-module. Because $H^0(\bbX_n,\tcD^{(m)}_{n,k})$ is noetherian, cf. \ref{graded_tcD}, it follows that $\im(\delta)$ is also finitely generated over this ring. We are therefore reduced to showing that $H^0\Big(\bbX_n,\tcD^{(m)}_{n,k}(-r)\Big)$ is a finitely generated $H^0(\bbX_n,\tcD_{n,k})$-module. This is the case if

$${\rm gr}\Big(H^0\Big(\bbX_n,\tcD^{(m)}_{n,k}(-r)\Big)\Big)$$

\vskip8pt

is a finitely generated ${\rm gr}\Big(H^0\Big(\bbX_n,\tcD^{(m)}_{n,k}\Big)\Big)$-module. We know that

$$H^0\Big(\bbX_n,{\rm gr}\Big(\tcD^{(m)}_{n,k}\Big)(-r)\Big)$$

\vskip8pt

is a finitely generated $H^0\Big(\bbX_n,{\rm gr}\Big(\tcD^{(m)}_{n,k}\Big)\Big)$-module, by \ref{finiteness_twist} and \ref{graded_tcD}. By \ref{graded_tcD} we have
$H^0\Big(\bbX_n,{\rm gr}\Big(\tcD^{(m)}_{n,k}\Big)\Big) = {\rm gr}\Big(H^0(\bbX_n,\tcD^{(m)}_{n,k})\Big)$, and the injective canonical map

$${\rm gr}\Big(H^0\Big(\bbX_n,\tcD^{(m)}_{n,k}(-r)\Big)\Big) \lra
H^0\Big(\bbX_n,{\rm gr}\Big(\tcD^{(m)}_{n,k}\Big)(-r)\Big)$$

\vskip8pt

is a homomorphism of ${\rm gr}\Big(H^0(\bbX_n,\tcD^{(m)}_{n,k})\Big)$-modules. Because the latter ring is noetherian, we conclude that ${\rm gr}\Big(H^0\Big(\bbX_n,\tcD^{(m)}_{n,k}(-r)\Big)\Big)$ is a finitely generated ${\rm gr}\Big(H^0\Big(\bbX_n,\tcD^{(m)}_{n,k}\Big)\Big)$-module. \qed

\vskip12pt

\subsection{Cohomology of coherent $\tsD^{(m)}_{n,k,\Q}$-modules}\label{coh_coh_tsD_mod}

We denote by $\bbX_{n,j}$ the reduction of $\bbX_n$ modulo $p^{j+1}$.

\vskip8pt

\begin{prop}\label{completion}\hskip-4pt\footnote{Statement and proof are as in \cite[3.2]{Huyghe97}.} Let $\cE$ be a coherent $\tcD^{(m)}_{n,k}$-module on $\bbX_n$ and $\widehat{\cE} = \varprojlim_j \cE/p^{j+1}\cE$ its $p$-adic completion, which we consider as a sheaf on $\frX_n$.

\vskip8pt

(i) For all $i \ge 0$ one has $H^i(\frX_n,\widehat{\cE}) = \varprojlim_j H^i\left(\bbX_{n,j},\cE/p^{j+1}\cE \right)$.

\vskip8pt

(ii) For all $i>0$ one has $H^i(\frX_n,\widehat{\cE}) = H^i(\bbX_n,\cE)$.

\vskip8pt

(iii) $H^0(\frX_n,\widehat{\cE}) = \varprojlim_j H^0(\bbX_n,\cE)/p^{j+1}H^0(\bbX_n,\cE)$.
\end{prop}

\Pf (i) Put $\cE_j = \cE/p^{j+1}\cE$. Let $\cE_t$ be the subsheaf defined by

$$\cE_t(U) = \cE(U)_{\rm tor} \;,$$

\vskip8pt

where the right hand side denotes the group of torsion elements in $\cE(U)$. This is indeed a sheaf (and not only a presheaf) because $\bbX_n$ is a noetherian space. Furthermore, $\cE_t$ is a $\tcD^{(m)}_{n,k}$-submodule of $\cE$. Because the sheaf $\tcD^{(m)}_{n,k}$ has noetherian rings of sections over open affine subsets of $\bbX_n$, cf. \ref{graded_tcD}, the submodule $\cE_t$ is a coherent $\tcD^{(m)}_{n,k}$-module. $\cE_t$ is thus generated by a coherent $\cO_{\bbX_n}$-submodule $\cF$ of $\cE_t$. The submodule $\cF$ is annihilated by a fixed power $p^c$ of $p$, and so is $\cE_t$. Put $\cG = \cE/\cE_t$, which is again a coherent $\tcD^{(m)}_{n,k}$-module. Using \ref{finite_power}, we can then assume, after possibly replacing $c$ by a larger number, that

\vskip8pt

$$\begin{array}{rl}
(a) & p^c\cE_t = 0 \;,\\
(b) & \mbox{for all } i>0: p^cH^i(\bbX_n,\cE) = 0 \;,\\
(c) & \mbox{for all } i>0: p^cH^i(\bbX_n,\cG) = 0 \;.\\
\end{array}$$

\vskip8pt

Let $j,\ell$ be integers which are greater or equal to $c$. Denote by $\nu_\ell: \cG \ra \cE$ the map induced by multiplication by $p^{\ell+1}$, so that we have an exact sequence

\begin{numequation}\label{ex_seq_mult_p_power}
0 \lra \cG \stackrel{\nu_\ell}{\lra} \cE \lra \cE_\ell \lra 0 \;,
\end{numequation}

where $\cE_j = \cE/p^{j+1}\cE$. Consider the morphism of two such sequences

\begin{numequation}\label{morph_ex_seq}
\xymatrixcolsep{3pc}\xymatrix{
0 \ar[r] & \cG \ar[r]^{\nu_{\ell+j}} \ar[d]^{p^j} & \cE \ar[r] \ar[d]^{{\rm id}} & \cE_{\ell+j} \ar[r] \ar[d] & 0\\
0 \ar[r] & \cG \ar[r]^{\nu_{\ell}} & \cE \ar[r] & \cE_{\ell} \ar[r] & 0
}
\end{numequation}

We get hence a morphism of long exact sequences

\begin{numequation}\label{coh_ex_seq}
\xymatrixcolsep{3.5pc}\xymatrix{
H^i(\bbX_n,\cG) \ar[r]^{H^i(\nu_{\ell+j})} \ar[d]^{p^j} & H^i(\bbX_n,\cE) \ar[r]^-{\alpha_{\ell+j}} \ar[d]^{{\rm id}} & H^i(\bbX_{n,\ell+j},\cE_{\ell+j}) \ar[r] \ar[d]^{\lambda_{\ell+j,\ell}} & H^{i+1}(\bbX_n,\cG) \ar[d]^{p^j}\\
H^i(\bbX_n,\cG) \ar[r]^{H^i(\nu_\ell)} & H^i(\bbX_n,\cE) \ar[r]^-{\alpha_\ell} & H^i(\bbX_{n,\ell},\cE_{\ell}) \ar[r]^-{\beta_\ell} & H^{i+1}(\bbX_n,\cG)
}
\end{numequation}

Because we assume $j \ge c$ the vertical map on the right is the zero map, and hence $\beta_\ell \circ \lambda_{\ell+j,\ell} = 0$, which means that $\im(\lambda_{\ell+j,\ell})$ is contained in $\im(\alpha_\ell)$. Since $\lambda_{\ell+j,\ell} \circ \alpha_{\ell+j} = \alpha_\ell$, we find that $\im(\lambda_{\ell+j,\ell}) = \im(\alpha_\ell)$ for all $j \ge c$. Therefore, the projective system $(H^i(\bbX_{n,j},\cE_j))_j$, with the transition maps given by the $\lambda_{j',j}$ with $j' \ge j$, satisfies the Mittag-Leffler condition for any $i \ge 0$.

\vskip8pt

We now want to apply the proposition \cite[ch. 0, 13.3.1]{EGA_III}. The transition maps of the system $(\cE_j)_j$ are obviously surjective, which gives the third condition of that proposition. Moreover, if $U \sub \bbX_n$ is an open affine subset, then we have $H^i(U,\cE_j) = 0$ for $i>0$, because $\cE_j$ is quasi-coherent over $\cO_{\bbX_n}$. This implies the second condition of that proposition. It follows that the exact sequence

$$0 \lra \cE_{\ell} \stackrel{p^j}{\lra} \cE_{\ell+j} \lra \cE_j \lra 0$$

\vskip8pt

stays exact after taking sections over $U$. Hence the system $(H^0(U,\cE_j))_j$ satisfies the Mittag-Leffler condition, which, together with the vanishing of higher cohomology on $U$ observed before, means that the first condition of \cite[ch. 0, 13.3.1]{EGA_III} is fulfilled. Because we have proven above that the system $(H^i(\bbX_{n,j},\cE_j))_j$ satisfies the Mittag-Leffler condition for any $i \ge 0$, it now follows from \cite[ch. 0, 13.3.1]{EGA_III} that for all $i \ge 0$

$$H^i(\frX_n,\widehat{\cE}) = \varprojlim_j H^i\left(\bbX_{n,j},\cE/p^{j+1}\cE \right) \;.$$

\vskip8pt

Hence we have proved the first assertion.

\vskip8pt

(ii) For the second assertion
we consider again the diagram \ref{coh_ex_seq} and notice that the maps $H^i(\nu_\ell)$ are zero maps for $i>0$ and $\ell \ge c$, and thus $\alpha_\ell$ is an isomorphism onto its image, for these $i$ and $\ell$. Therefore, the projective limit of the $(H^i(\bbX_{n,j},\cE_j))_j$ is equal to $H^i(\bbX_n,\cE)$ when $i>0$.

\vskip8pt

(iii) As above, we let $\ell,j \ge c$ be integers, and we consider the exact sequence

$$0 \ra \cE_t \ra \cE \stackrel{p^{\ell+1}}{\lra} \cE_\ell \ra 0 \;.$$

\vskip8pt

This sequence splits into two exact sequences

$$0 \ra \cE_t \ra \cE \stackrel{u}{\lra} \cG \ra 0 \hskip16pt \mbox{and} \hskip16pt 0 \ra \cG \stackrel{\nu_\ell}{\lra} \cE \ra \cE_\ell \ra 0 \;.$$

\vskip8pt

The long exact cohomology sequences to these sequences begin as follows:

$$\begin{array}{ccccccccc}

0 & \ra & H^0(\bbX_n,\cE_t) & \ra & H^0(\bbX_n,\cE) & \stackrel{u}{\lra} & H^0(\bbX_n,\cG) & \ra & H^1(\bbX_n,\cE_t)\\
\\
0 & \ra & H^0(\bbX_n,\cG) & \stackrel{\nu_\ell}{\lra} & H^0(\bbX_n,\cE) & \ra & H^0(\bbX_n,\cE_\ell) & \ra & H^1(\bbX_n,\cG) \;.
\end{array}$$

\vskip8pt

From the second of these exact sequences and the morphism of exact sequences \ref{morph_ex_seq} we obtain the following morphism of exact sequences

\begin{numequation}\label{H0_ex_seq}
\xymatrixcolsep{2.7pc}\xymatrix{
0 \ar[r] & H^0(\bbX_n,\cG) \ar[r]^{\nu_{\ell+j}} \ar[d]^{p^j} & H^0(\bbX_n,\cE) \ar[r]^-{\alpha_{\ell+j}} \ar[d]^{{\rm id}} & H^0(\bbX_{n,\ell+j},\cE_{\ell+j}) \ar[r] \ar[d]^{\lambda_{\ell+j,\ell}} & H^1(\bbX_n,\cG) \ar[d]^{p^j}\\
0 \ar[r] & H^0(\bbX_n,\cG) \ar[r]^{\nu_\ell} & H^0(\bbX_n,\cE) \ar[r]^-{\alpha_\ell} & H^0(\bbX_{n,\ell},\cE_{\ell}) \ar[r]^-{\beta_\ell} & H^1(\bbX_n,\cG)
}
\end{numequation}

The composition $\nu_\ell \circ u$ is equal to the multiplication by $p^{\ell+1}$, and we get therefore a canonical surjection

$$\sigma_\ell: H^0(\bbX_n,\cE)/p^{\ell+1} H^0(\bbX_n,\cE) \twoheadrightarrow H^0(\bbX_n,\cE) \Big/ \nu_\ell\Big(H^0(\bbX_n,\cG)\Big) \;.$$

\vskip8pt

The homomorphisms $\sigma_\ell$ form a morphism of projective systems. Because $\nu_\ell$ is injective, we have a canonical isomorphism

$$\begin{array}{rcl}
\ker(\sigma_\ell) & = & \nu_\ell\Big(H^0(\bbX_n,\cG)\Big)\Big/p^{\ell+1} H^0(\bbX_n,\cE) \\
&&\\
& = &\nu_\ell\Big(H^0(\bbX_n,\cG)\Big)\Big/\nu_\ell\Big(u\Big( H^0(\bbX_n,\cE)\Big)\Big) \\
&&\\
& \simeq & H^0(\bbX_n,\cG)\Big/u\Big( H^0(\bbX_n,\cE)\Big) = \coker(H^0(u)) \;,
\end{array}$$

\vskip8pt

and $\coker(H^0(u))$ embeds into $H^1(\bbX_n,\cG)$ which is annihilated by $p^c$. Furthermore, the morphism of exact sequences

$$\xymatrixcolsep{3pc}\xymatrix{
0 \ar[r] & \cE_t \ar[r] \ar[d]^{p^j} & \cE \ar[r]^{p^{\ell+j+1}} \ar[d]^{p^j} & \cE \ar[r] \ar[d]^{\rm id} & \cE_{\ell+j} \ar[r] \ar[d] & 0\\
0 \ar[r] & \cE_t \ar[r] & \cE \ar[r]^{p^{\ell+1}} & \cE \ar[r] & \cE_\ell \ar[r] & 0
}
$$

\vskip8pt

induces a morphism of short exact sequences

$$\xymatrixcolsep{1pc}\xymatrix{
0 \ar[r] & \coker(H^0(u)) \ar[r] \ar[d]^{p^j} & H^0(\bbX_n,\cE)/p^{\ell+j+1}H^0(\bbX_n,\cE) \ar[r] \ar[d] & H^0(\bbX_n,\cE)\Big/\nu_{\ell+j}\Big(H^0(\bbX_n,\cG \Big) \ar[r] \ar[d] & 0\\
0 \ar[r] & \coker(H^0(u)) \ar[r] & H^0(\bbX_n,\cE)/p^{\ell+1} H^0(\bbX_n,\cE) \ar[r] & H^0(\bbX_n,\cE)\Big/\nu_\ell \Big(H^0(\bbX_n,\cG \Big) \ar[r] & 0
}
$$

\vskip8pt

The projective limit $\varprojlim_\ell \ker(\sigma_\ell)$ vanishes thus, and the system $(\sigma_\ell)_\ell$ induces an isomorphism

$$\varprojlim_\ell H^0(\bbX_n,\cE)/p^{\ell+1} H^0(\bbX_n,\cE) \stackrel{\simeq}{\lra} \varprojlim_\ell H^0(\bbX_n,\cE)\Big/\nu_\ell \Big(H^0(\bbX_n,\cG \Big)$$

\vskip8pt

and, by looking at \ref{H0_ex_seq}, we see that the right hand side is in turn canonically isomorphic to $\varprojlim_\ell H^0(\bbX_{n,\ell},\cE_\ell) = H^0(\frX_n,\widehat{\cE})$, by assertion (i). \qed

\vskip8pt

\begin{cor}\label{noetherian_tsD}
The ring $H^0(\frX_n,\tsD^{(m)}_{n,k})$ is noetherian.
\end{cor}
\begin{proof}
This follows from (iii) of the preceding proposition together with \ref{global_sections_tcD}.
\end{proof}

\vskip8pt

\begin{prop}\label{surj_compl} Let $\sE$ be a coherent $\tsD^{(m)}_{n,k}$-module.

\vskip8pt

(i) There is $r_1(\sE) \in \Z$ such that for all $r \ge r_1(\sE)$ there is $s \in \Z_{\ge 0}$ and an epimorphism of $\tsD^{(m)}_{n,k}$-modules

$$\Big(\tsD^{(m)}_{n,k}(-r)\Big)^{\oplus s} \twoheadrightarrow \sE \;.$$

\vskip8pt

(ii) There is $r_2(\sE) \in \Z$ such that for all $r \ge r_2(\sE)$ and all $i >0$

$$H^i\Big(\frX_n, \sE(r)\Big) = 0 \;.$$

\vskip8pt
\end{prop}

\Pf (i) Because $\sE$ is a coherent $\tsD^{(m)}_{n,k}$-module, and because $H^0(U,\tsD^{(m)}_{n,k})$ is a noetherian ring for all open affine subsets $U \sub \frX_n$, cf. \ref{global_sec_tsD}, the torsion submodule $\sE_t \sub \sE$ is again a coherent $\tsD^{(m)}_{n,k}$-module. As $\frX_n$ is quasi-compact, there is $c \in \Z_{\ge 0}$ such that $p^c \sE_t = 0$.
Put $\sG = \sE/\sE_t$ and $\sG_0 = \sG/p\sG$. For $j \ge c$ one has an exact sequence

$$0 \ra \sG_0 \stackrel{p^{j+1}}{\lra} \sE_{j+1} \ra \sE_j \ra 0 \;.$$

\vskip8pt

We note that the sheaf $\sG_0$ is a coherent module over
$\tsD^{(m)}_{n,k}/p\tsD^{(m)}_{n,k}$. Because the canonical map of sheaves of rings

\begin{numequation}\label{iso_uncompl_compl}
\tcD^{(m)}_{n,k}/p\tcD^{(m)}_{n,k} \lra \tsD^{(m)}_{n,k}/p\tsD^{(m)}_{n,k}
\end{numequation}

is an isomorphism, $\sG_0$ can be considered as a module over the sheaf on the left hand side of \ref{iso_uncompl_compl}, and thus as a $\tcD^{(m)}_{n,k}$-module. It is then automatically a coherent $\tcD^{(m)}_{n,k}$-module. Hence we can apply \ref{surjection} and deduce that there is $r_2(\sG_0)$ such that for all $r \ge r_2(\sG_0)$ one has $H^1(\frX_n,\sG_0(r)) = 0$. The canonical maps

\begin{numequation}\label{surj_H0}
H^0(\frX_n,\sE_{j+1}(r)) \lra H^0(\frX_n,\sE_j(r))
\end{numequation}

are thus surjective for $r \ge r_2(\sG_0)$ and $j \ge c$. Similarly, $\sE_c$ is a coherent module over $\tcD^{(m)}_{n,k}/p^c\tcD^{(m)}_{n,k}$-module, in particular a coherent $\tcD^{(m)}_{n,k}$-module. By \ref{surjection} there is $r_1(\sE_c)$ such that for every $r \ge r_1(\sE_c)$ there is $s \in \Z_{\ge 0}$ and a surjection

$$\lambda: \Big(\tcD^{(m)}_{n,k}/p^c\tcD^{(m)}_{n,k}\Big)^{\oplus s} \twoheadrightarrow \sE_c(r) \;.$$

\vskip8pt

Let $r_1(\sE) = \max\{r_2(\sG_0),r_1(\sE_c)\}$, and assume from now on that $r \ge r_1(\sE)$. Let $e_1, \ldots, e_s$ be the standard basis of the domain of $\lambda$, and use \ref{surj_H0} to lift each $\lambda(e_t)$, $1 \le t \le s$, to an element of

$$\varprojlim_j H^0(\frX_n,\sE_j(r)) \simeq H^0(\frX_n,\widehat{\sE(r)}) \;,$$

\vskip8pt

by \ref{completion} (i). But $\widehat{\sE(r)} = \widehat{\sE}(r)$, and $\widehat{\sE} = \sE$, as follows from \cite[3.2.3 (v)]{BerthelotDI}. This defines a morphism

$$\Big(\tsD^{(m)}_{n,k}\Big)^{\oplus s} \lra \sE(r)$$

\vskip8pt

which is surjective because, modulo $p^c$, it is a surjective morphism of sheaves coming from coherent $\tsD^{(m)}_{n,k}$-modules by reduction modulo $p^c$, cf. \cite[3.2.2 (ii)]{BerthelotDI}.

\vskip8pt

(ii) We deduce from \ref{vanishing_coh_Dnk} and \ref{completion} that for all $i>0$

$$H^i\Big(\frX_n,\tsD^{(m)}_{n,k}(r)\Big) = 0 \;,$$

\vskip8pt

whenever $r \ge r_0$, where $r_0$ is as in \ref{v_ample_sh_lemma}. We prove the statement in (ii) by descending induction on $i$. When $i \ge 2$ there is nothing to show because $\frX_n$ is a one-dimensional noetherian space. We will now deduce it for $i=1$. Using (i) we find an epimorphism of $\tsD^{(m)}_{n,k}$-modules

$$\beta: \left(\tsD^{(m)}_{n,k}\right)^{\oplus s}  \twoheadrightarrow \sE\Big(r_1(\sE)\Big) \;.$$

\vskip8pt

The kernel $\sR = \ker(\beta)$ is a coherent $\tsD^{(m)}_{n,k}$-module and we have an exact sequence

$$0 \ra \sR \ra \left(\tsD^{(m)}_{n,k}\right)^{\oplus s} \ra \sE\Big(r_1(\sE)\Big) \ra 0 \;,$$

\vskip8pt

which gives for any $t \in \Z$ the exact sequence

\begin{numequation}\label{ex_seq_van_coh_compl} 0 \ra \sR(t) \ra \left(\tsD^{(m)}_{n,k}(t)\right)^{\oplus s} \ra \sE\Big(t+r_1(\sE)\Big) \ra 0 \;.
\end{numequation}

Now choose $r_2(\sE) = \max\{r_0+r_1(\sE),r_2(\sR)\}$, where $r_0$ is as before. Then we can conclude that from the long exact cohomology sequence associated to \ref{ex_seq_van_coh_compl} that for $r \ge r_2(\sE)$

$$H^1(\frX_n,\sE(r)) = 0 \;.$$

\vskip8pt

\qed

\begin{prop}\label{compl_finite_power_torsion} Let $\sE$ be a coherent $\tsD^{(m)}_{n,k}$-module.

\vskip8pt

(i) There is $c=c(\sE) \in \Z_{\ge 0}$ such that for all $i>0$
the cohomology group $H^i(\frX_n,\sE)$ is annihilated by $p^c$. Furthermore, it is a finitely generated module over $H^0(\frX_n,\tsD^{(m)}_{n,k})$.

\vskip8pt

(ii) $H^0(\frX_n,\sE) = \varprojlim_j H^0(\frX_n,\sE)/p^jH^0(\frX_n,\sE)$.

\end{prop}

\Pf (i) There is nothing to show for $i \ge 2$. We now deduce the case $i=1$ from the case $i=2$. Use \ref{surj_compl} (i) to find a surjection of $\tsD^{(m)}_{n,k}$-modules

$$\beta: \Big(\tsD^{(m)}_{n,k}(-r)\Big)^{\oplus s} \twoheadrightarrow \sE \;.$$

\vskip8pt

Let $\sR = \ker(\beta)$ which is again a coherent $\tsD^{(m)}_{n,k}$-module. Applying the long exact cohomology sequence to the exact sequence

$$0 \lra \sR \lra  \Big(\tsD^{(m)}_{n,k}(-r)\Big)^{\oplus s} \lra \sE \lra 0$$

\vskip8pt

we find that $H^1\Big(\frX_n,\tsD^{(m)}_{n,k}(-r)\Big)^{\oplus s}$ surjects onto $H^1(\frX_n,\sE)$. By \ref{completion} we have

$$H^1(\frX_n,\tsD^{(m)}_{n,k}(-r)) =  H^1(\bbX_n,\tcD^{(m)}_{n,k}(-r)) \;,$$

\vskip8pt

and this is annihilated by a finite power of $p$, by \ref{finite_power}. Furthermore, $H^1(\bbX_n,\tcD^{(m)}_{n,k}(-r))$ is a finitely generated $H^0(\bbX_n,\tcD^{(m)}_{n,k})$-module, by \ref{finite_power}, and hence a finitely generated $H^0(\bbX_n,\tsD^{(m)}_{n,k})$-module.

\vskip8pt

(ii) Let $\sE_t \sub \sE$ be the subsheaf of torsion elements and $\sG = \sE/\sE_t$. Then the discussion in the beginning of the proof of \ref{completion} shows that there is $c \in \Z_{\ge 0}$ such that $p^c\sE_t = 0$. Part (i) gives that $p^cH^1(\frX_n,\sE) = p^cH^1(\frX_n,\sG) = 0$, after possibly increasing $c$. Now we can apply the same reasoning as in the proof of \ref{completion} (iii) to conclude that assertion (ii) is true. \qed

\vskip8pt

\begin{prop}\label{finiteness} Let $\sE$ be a coherent $\tsD^{(m)}_{n,k}$-module. Then $H^0(\frX_n,\sE)$ is a finitely generated $H^0(\frX_n,\tsD^{(m)}_{n,k})$-module.
\end{prop}

\Pf By \ref{surj_compl} there is an exact sequence of coherent $\tsD^{(m)}_{n,k}$-modules

$$0 \lra \sR \lra \Big( \tsD^{(m)}_{n,k}(-r)\Big)^{\oplus s} \lra \sE \lra 0 \;.$$

\vskip8pt

The long exact cohomology sequence begins as follows

$$0 \ra H^0(\bbX_n,\sR) \ra H^0\left(\bbX_n,\Big( \tsD^{(m)}_{n,k}(-r)\Big)^{\oplus s}\right) \ra H^0(\bbX_n,\sE) \ra H^1(\bbX_n,\sR) \;.$$

\vskip8pt

By \ref{compl_finite_power_torsion} the group on the right is finitely generated over $H^0(\frX_n,\tsD^{(m)}_{n,k})$. Because \linebreak $H^0(\frX_n,\tsD^{(m)}_{n,k})$ is noetherian, cf. \ref{noetherian_tsD},

$$\im\Big(H^0(\bbX_n,\sE) \ra H^1(\bbX_n,\sR)\Big)$$

\vskip8pt

is finitely generated over $H^0(\frX_n,\tsD^{(m)}_{n,k})$ as well.
It follows from \ref{fg_H0} and \ref{completion} that the group

$$H^0\left(\bbX_n,\Big( \tsD^{(m)}_{n,k}(-r)\Big)^{\oplus s}\right)$$

\vskip8pt

is finitely generated over $H^0(\frX_n,\tsD^{(m)}_{n,k})$, and so is then $H^0(\bbX_n,\sE)$. \qed

\vskip8pt

\begin{para} Let $\Coh(\tsD^{(m)}_{n,k})$ (resp.  $\Coh(\tsD^{(m)}_{n,k,\Q})$) be the category of coherent $\tsD^{(m)}_{n,k}$-modules (resp. $\tsD^{(m)}_{n,k,\Q}$-modules). Let $\Coh(\tsD^{(m)}_{n,k})_\Q$ be the category of coherent $\tsD^{(m)}_{n,k}$-modules up to isogeny. We recall that this means that $\Coh(\tsD^{(m)}_{n,k})_\Q$ has the same class of objects as $\Coh(\tsD^{(m)}_{n,k})$, and for any two objects $\cM$ and $\cN$ one has

$$Hom_{\Coh(\tsD^{(m)}_{n,k})_\Q}(\cM,\cN) = Hom_{\Coh(\tsD^{(m)}_{n,k})}(\cM,\cN) \otimes_\Z \Q \;.$$

\vskip8pt
\end{para}

\vskip8pt

\begin{prop}\label{integral_models} (i) The functor $\cM \mapsto \cM_\Q = \cM \otimes_\Z \Q$ induces an equivalence between $\Coh(\tsD^{(m)}_{n,k})_\Q$ and $\Coh(\tsD^{(m)}_{n,k,\Q})$.

\vskip8pt

(ii) For every coherent $\tsD^\dagger_{n,k,\Q}$-module $\sM$ there is $m \ge 0$ and a coherent $\tsD^{(m)}_{n,k,\Q}$-module $\sM_m$ and an isomorphism of $\tsD^\dagger_{n,k,\Q}$-modules

$$\vep: \tsD^\dagger_{n,k,\Q} \otimes_{\tsD^{(m)}_{n,k,\Q}} \sM_m \stackrel{\simeq}{\lra} \sM \;.$$

\vskip8pt

If $(m', \sM_{m'},\vep')$ is another such triple, then there is $\ell \ge \max\{m,m'\}$ and an isomorphism of $\tsD^{(\ell)}_{n,k,\Q}$-modules

$$\vep_\ell: \tsD^{(\ell)}_{n,k,\Q} \otimes_{\tcD^{(m)}_{n,k,\Q}} \sM_m \stackrel{\simeq}{\lra} \tsD^{(\ell)}_{n,k,\Q} \otimes_{\tcD^{(m')}_{n,k,\Q}} \sM_{m'}$$

\vskip8pt

such that $\vep' \circ \Big({\rm id}_{\tcD^\dagger_{n,k,\Q}} \otimes \vep_\ell\Big) = \vep$.
\end{prop}

\Pf (i) This is \cite[3.4.5]{BerthelotDI}. Note that the sheaf $\tsD^{(m)}_{n,k}$ satisfies the conditions in \cite[3.4.1]{BerthelotDI}, by \ref{global_sec_tsD}. We point out that the formal scheme $\cX$ in \cite[sec. 3.4]{BerthelotDI} is not supposed to be smooth over a discrete valuation ring, but only locally noetherian, cf. \cite[sec. 3.3]{BerthelotDI}.
\vskip8pt

(ii) This is \cite[3.6.2]{BerthelotDI}. In this reference the formal scheme is supposed to be noetherian and quasi-separated, but not necessarily smooth over a discrete valuation ring. \qed

\vskip8pt

\begin{cor}\label{finiteness_Q} Let $\sE$ be a coherent $\tsD^{(m)}_{n,k,\Q}$-module. Then $H^0(\frX_n,\sE)$ is a finitely generated $H^0(\frX_n,\tsD^{(m)}_{n,k,\Q})$-module.
\end{cor}

\Pf By \ref{integral_models} there is a coherent $\tsD^{(m)}_{n,k}$-module $\sF$ such that $\sF \otimes_\Z \Q \simeq \sE$. Now we can apply \ref{finiteness} to $\cF$. \qed

\vskip8pt

\begin{thm}\label{acycl_tsD} Let
 $\sE$ be a coherent $\tsD^{(m)}_{n,k,\Q}$-module (resp. $\tsD^\dagger_{n,k,\Q}$-module).

\vskip8pt

(i) There is $r(\sE) \in \Z$ such that for all $r \ge r(\sE)$ there is $s \in \Z_{\ge 0}$ and an epimorphism of $\tsD^{(m)}_{n,k,\Q}$-modules (resp. $\tsD^\dagger_{n,k,\Q}$-modules)

$$\Big(\tsD^{(m)}_{n,k,\Q}(-r)\Big)^{\oplus s} \twoheadrightarrow \sE \; \hskip16pt (\; \mbox{resp.} \; \Big(\tsD^\dagger_{n,k,\Q}(-r)\Big)^{\oplus s} \twoheadrightarrow \sE \;) \;.
$$

\vskip8pt

(ii) For all $i >0$ one has $H^i(\frX_n, \sE) = 0$.

\vskip8pt
\end{thm}

\Pf (a) We first show both assertions (i) and (ii) for a coherent $\tsD^{(m)}_{n,k,\Q}$-module $\sE$. By \ref{integral_models} (i) there is a coherent  $\tsD^{(m)}_{n,k}$-module $\sF$ such that $\sF \otimes_\Z \Q  = \sE$. We use \ref{surj_compl} to find for every $r \ge r_1(\sF)$ a surjection

$$\Big(\tsD^{(m)}_{n,k}(-r)\Big)^{\oplus s} \twoheadrightarrow \sF \;,$$

\vskip8pt

for some $s$ (depending on $r$).  Tensoring with $\Q$ gives then the desired surjection onto $\sE$. Hence assertion (i). Furthermore, for $i>0$

$$H^i(\frX_n,\sE) = H^i(\frX_n,\sF) \otimes_\Z \Q = 0 \;,$$

\vskip8pt

by \ref{compl_finite_power_torsion}, and this proves (ii).

\vskip8pt

(b) Now suppose $\sE$ is a coherent $\tcD^\dagger_{n,k,\Q}$-module.
By \ref{integral_models} (ii) there is $m \ge 0$ and a coherent module $\sE_m$ over $\tsD^{(m)}_{n,k,\Q}$ and an isomorphism of $\tsD^\dagger_{n,k,\Q}$-modules

$$\tsD^\dagger_{n,k,\Q} \otimes_{\tsD^{(m)}_{n,k,\Q}} \sE_m \stackrel{\simeq}{\lra} \sE \;.$$

\vskip8pt

Now use what we have just shown for $\sE_m$ in (a) and get the sought for surjection after tensoring with $\tsD^\dagger_{n,k,\Q}$. This proves the first assertion. We have

$$\sE = \tsD^\dagger_{n,k,\Q} \otimes_{\tsD^{(m)}_{n,k,\Q}} \sE_m = \varinjlim_{\ell \ge m} \tsD^{(\ell)}_{n,k,\Q} \otimes_{\tsD^{(m)}_{n,k,\Q}} \sE_m  = \varinjlim_{\ell \ge m} \sE_\ell$$

\vskip8pt

where $\sE_\ell = \tsD^{(\ell)}_{n,k,\Q} \otimes_{\tsD^{(m)}_{n,k,\Q}} \sE_m$ is a coherent $\tsD^{(\ell)}_{n,k,\Q}$-module. Then we have for $i>0$

$$H^i(\frX_n,\sE) = \varinjlim_{\ell \ge m} H^i(\frX_n,\sE_\ell) = 0 \;,$$

\vskip8pt

by part (a). And this proves assertion (ii). \qed

\vskip12pt

\subsection{$\frX_n$ is $\tsD^{(m)}_{n,k,\Q}$-affine and $\tsD^\dagger_{n,k,\Q}$-affine}

\begin{prop}\label{prop-genglobal} (i) Let $\sE$ be a coherent $\tsD^{(m)}_{n,k,\Q}$-module. Then $\sE$ is generated by its global sections as $\tsD^{(m)}_{n,k,\Q}$-module. Furthermore, $\sE$ has a resolution by finite free $\tsD^{(m)}_{n,k,\Q}$-modules.

\vskip8pt

(ii) Let $\sE$ be a coherent $\tsD^\dagger_{n,k,\Q}$-module. Then $\sE$ is generated by its global sections as $\tsD^\dagger_{n,k,\Q}$-module. $H^0(\frX_n,\sE)$ is a $H^0(\frX_n,\tsD^\dagger_{n,k,\Q})$-module of finite presentation.
Furthermore, $\sE$ has a resolution by finite free $\tsD^{\dagger}_{n,k,\Q}$-modules.
\end{prop}

\Pf Using \ref{acycl_tsD} we can argue exactly as in \cite{Huyghe97}. \qed

\begin{para} {\it The functors $\Loc^{(m)}_{n,k}$ and $\Loc^\dagger_{n,k}$.} Let $E$ be a finitely generated $H^0(\frX_n,\tsD^{(m)}_{n,k,\Q})$-module (resp. a finitely presented $H^0(\frX_n,\tsD^\dagger_{n,k,\Q})$-module). Then we let $\Loc^{(m)}_{n,k}(E)$ (resp. $\Loc^\dagger_{n,k}(E)$) be the sheaf on $\bbX_n$ associated to the presheaf

$$U \mapsto \tsD^{(m)}_{n,k,\Q}(U) \otimes_{H^0(\frX_n,\tsD^{(m)}_{n,k,\Q})} E \hskip16pt ({\rm resp.} \;\; U \mapsto \tsD^\dagger_{n,k,\Q}(U) \otimes_{H^0(\frX_n,\tsD^\dagger_{n,k,\Q})} E \;) \;.$$

\vskip8pt

It is obvious that $\Loc^{(m)}_{n,k}$ (resp. $\Loc^\dagger_{n,k}$) is a functor from the category of finitely generated $H^0(\frX_n,\tsD^{(m)}_{n,k,\Q})$-modules (resp. finitely presented $H^0(\frX_n,\tsD^\dagger_{n,k,\Q})$-modules) to the category of sheaves of modules over $\tsD^{(m)}_{n,k,\Q}$ (resp. $\tsD^\dagger_{n,k,\Q}$).
\end{para}

\vskip8pt

\begin{thm}\label{thm-equivalence} (i) The functors $\Loc^{(m)}_{n,k}$ and $H^0$ (resp. $\Loc^\dagger_{n,k}$ and $H^0$) are quasi-inverse equivalences between the categories of finitely generated $H^0(\frX_n,\tsD^{(m)}_{n,k,\Q})$-modules and coherent $\tsD^{(m)}_{n,k,\Q}$-modules (resp. finitely presented $H^0(\frX_n,\tsD^\dagger_{n,k,\Q})$-modules and coherent $\tsD^\dagger_{n,k,\Q}$-modules).

\vskip8pt

(ii) The functor $\Loc^{(m)}_{n,k}$ (resp. $\Loc^\dagger_{n,k}$) is an exact functor.
\end{thm}

\Pf The proofs of \cite[5.2.1, 5.2.3]{Huyghe97} for the first and the second assertion, respectively, carry over word for word. \qed

\vskip12pt

\section{Localization of representations of $\GL_2(L)$}\label{loc}

\subsection{Finitely generated modules over compact type algebras} A space of {\it compact type} is an $L$-vector space $V=\varinjlim_{m\in\bbN} V_m$ equal to a locally convex inductive limit over a countable system of Banach spaces $V_m$ where the transition maps $V_m \ra V_{m+1}$ are injective compact linear maps \cite[Def. 1.1.16]{EmertonA}. Any space of compact type is Hausdorff. A {\it compact type topological algebra} (for short, a compact type algebra) is a topological $L$-algebra $A = \varinjlim_m A_m$ equal to a locally convex inductive limit of Banach algebras $A_m$ with injective, continuous, compact linear and multiplicative transition maps $A_m \rightarrow A_{m+1}$ \cite[Def. 2.3.1]{EmertonJI}. We have the following generalization of \cite[Prop. 2.1]{ST5}. All modules are {\it left} modules.

\begin{prop}\label{prop-cantop} Let $A=\varinjlim_m A_m$ be a compact type algebra with noetherian Banach algebras $A_m$.

\vskip8pt

(i) Each finitely generated $A$-module $M$ carries a unique compact type topology (called its canonical topology) $M=\varinjlim_m M_m$ such that the $A$-module structure map $A\times M\rightarrow M$ becomes continuous; each $M_m$ can be chosen to be a finitely generated $A_m$-submodule of $M$.

\vskip8pt

(ii) every finitely generated $A$-submodule of $M$ is closed in the canonical topology;

\vskip8pt

(iii) any linear map between finitely generated $A$-modules is continuous and strict with closed image for the canonical topologies;

\vskip8pt

(iv) if $M$ admits a finite presentation $A^{\oplus s}\rightarrow A^{\oplus r}\stackrel{f}{\rightarrow} M\rightarrow 0$, then the quotient topology on $M\simeq A^{\oplus r}/\ker(f)$ coincides with the canonical topology.
\end{prop}

\Pf (i) Let $x_1,...,x_r$ be generators for $M$ and put $M_m:= \sum_i A_m x_i\sub M$. According to \cite[Prop. 2.1]{ST5} the $A_m$-module $M_m$ has a unique structure as Banach module and the natural injection $M_m\rightarrow M_{m+1}$ is continuous. Moreover, the latter injection is a compact linear map as follows from the compactness of $A_m\rightarrow A_{m+1}$ along the lines of the proof of \cite[Lem. 6.1]{ST5}. Hence, the inductive limit topology on $M=\varinjlim_{m} M_m$ is as claimed and it remains to show uniqueness. However, if $M$ has a second such topology, say $\tau$, then each inclusion $M_m\hookrightarrow M^{\tau}$ is continuous. The identity map $M^{can}\rightarrow M^{\tau}$ is therefore a continuous bijection. By the open mapping theorem, in the version of \cite[Thm. 1.1.17]{EmertonA}, it is therefore strict. Thus, it is a topological isomorphism. This shows (i).

\vskip8pt

(ii) If $N\subseteq M$ is any submodule of $M$, then $N_m:=M_m\cap N$ is a submodule of the $A_m$-module $M_m$ and hence closed by \cite[Prop.2.1]{ST5}. The inductive limit topology on $N'=\varinjlim_m N_m$ makes the natural map $N' \rightarrow M$ a closed embedding \cite[Prop. 1.1.41]{EmertonA}.
However, if $N$ is finitely generated, then $N=N'$ which proves (ii).

\vskip8pt

(iii) Given a linear map $f: M\rightarrow N$ between finitely generated modules, we have $f(M_m)\subseteq N_m$ for all large enough $m$. For such $m$, the induced map $f_m: M_n\rightarrow N_m$ is continuous, cf. \cite[Prop. 2.1]{ST5}, and hence, so is $f=\varinjlim_m f_m$. The image $\im(f)\subseteq N$ is closed by (ii). Moreover, $f: M\rightarrow\im(f)$ is a continuous surjection and therefore strict by \cite[Thm. 1.1.7]{EmertonA}. This proves (iii).

\vskip8pt

(iv) According to (ii) the submodule $\ker(f)\subseteq A^{\oplus r}$ is closed and the quotient topology on $M\simeq A^{\oplus r}/\ker(f)$ is therefore of compact type, cf. \cite[Prop. 1.2]{ST4}. Since it makes the $A$-module structure map $A\times M\rightarrow M$ continuous, it must coincide with the canonical topology by the uniqueness part in (i). \qed

\begin{rem} Suppose there is a central ideal $I_m\sub A_m$ for each $m$ with $I_{m+1}=A_{m+1}I_m$ and $I_{m+1}\cap A_m=I_m$. This yields a central ideal $I\sub A$ with $A/I=\varinjlim_m A_m/I_m.$ The algebra $A/I$ is then a compact type algebra.
\end{rem}

\vskip8pt

We will need the following lemma which justifies the replacement of completed tensor products which occur in certain situations when working with weak Frechet-Stein structures, cf. \cite[Def. 1.2.6]{EmertonA}, by ordinary tensor products. The key point here is that we consider finitely presented modules over compact type algebras $A=\varinjlim_m A_m$ with noetherian defining Banach algebras $A_m$. Recall that, if $A$ is a locally convex topological algebra and $M$ is a locally convex topological module, then $M$ is said to be a {\it finitely generated topological module}, if there is an $A$-module map $A^{\oplus r}\rightarrow M$, which is a continuous and strict surjection \cite[Def. 1.2.1]{EmertonA}. In this situation, the module is called {\it a finitely generated free topological module} if the map $A^{\oplus r}\rightarrow M$ is additionally injective, i.e. a topological isomorphism. Finally, given a continuous homomorphism between locally convex topological $L$-algebras $
 A\rightarrow B
 $ and a locally convex topological $A$-module $M$, the completed tensor product $B\hat{\otimes}_A M$ will always be understood as in \cite[Lem. 1.2.3]{EmertonA}, i.e. by regarding $B\otimes_A M$ as a quotient of $B\otimes_{K,\pi} M$\footnote{The symbol $\otimes_{K,\pi}$ refers to the projective tensor product, cf. \cite[(0.9)]{EmertonA}.}, giving it the quotient topology and then forming the Hausdorff completion.

\begin{lemma}\label{lem-decomplete} Let $A'$ and $B'$ be compact type algebras with noetherian defining Banach algebras. Let further $A'\subseteq A$ (resp. $B'\subseteq B$) be a topological algebra which is a finitely generated free topological module over $A'$ (resp. $B'$). Let $A\rightarrow B$ be a continuous homomorphism. Suppose $M$ is a finitely presented (abstract) $A$-module. Giving $M$ its canonical topology as finitely generated $A'$-module, the natural map $B\otimes_A M\car B\hat{\otimes}_A M$ is bijective.
\end{lemma}

\Pf We choose a finite presentation of the $A$-module $M$

$$(\cR)\hskip30pt A^{\oplus s}\rightarrow A^{\oplus r}\stackrel{f}{\rightarrow}M \rightarrow 0 \;.$$

\vskip8pt

This is a finite presentation of $M$ as $A'$-module, so that the map $f$ is in fact strict \ref{prop-cantop}. This shows that $M$ is a finitely generated topological $A$-module. According to \cite[Prop. 1.2.5]{EmertonA}, the natural map
$B\otimes_A M\rightarrow B\hat{\otimes}_A M$ is surjective and the map $B\otimes f$ occurring in

$$(B\otimes\cR)\hskip30pt B^{\oplus s}\rightarrow B^{\oplus r}\stackrel{B\otimes f}{\rightarrow}B\otimes_A M \rightarrow 0$$

\vskip8pt

is strict. It remains to see that $B\otimes_A M$ is Hausdorff. But $B\otimes_A M$ is a finitely presented $B$-module and therefore has its canonical topology as finitely generated $B'$-module. Since $B\otimes\cR$ is a finite presentation of $B\otimes_A M$ as $B'$-module, the map $B\otimes f$ is a strict surjection for the canonical topology, again by \ref{prop-cantop}. Hence the original topology on $B\otimes_A M$ coincides with the canonical topology and hence, is Hausdorff. \qed

\vskip12pt

\subsection{Modules over  $\cD^{\rm an}(\bbG(n)^\circ)_{\theta_0}$} We will apply these considerations to the algebra of global sections $H^0(\frX_n,\tsD^\dagger_{n,n,\Q})$. Recall that it is canonically isomorphic to the coherent $L$-algebra $\cD^{\rm an}(\bbG(n)^\circ)_{\theta_0}$, cf. Prop. \ref{global_sec_tsD}.

\begin{lemma}\label{lem-compacttype} The algebras $\cD^{\rm an}(\bbG(n)^\circ)$ and $\cD^{\rm an}(\bbG(n)^\circ)_{\theta_0}$ are compact type algebras with noetherian defining Banach algebras.
\end{lemma}

\Pf The discussion in \cite[5.3]{HS13} applied to the smooth and affine $\fro$-group scheme $\bbG(n)$ shows the property for $\cD^{\rm an}(\bbG(n)^\circ)$. According to the remark after Prop. \ref{prop-cantop} the property then passes to the central reduction $\cD^{\rm an}(\bbG(n)^\circ)_{\theta_0}$. \qed

\vskip8pt

We consider the locally $L$-analytic compact group $\GN=\GL_2(\fro)$ with its series of congruence subgroups $G_{n+1}=\bbG(n)^\circ(L)$. The group $\GN$ acts by translations on the space $C^{\rm cts}(\GN,K)$ of continuous $K$-valued functions. Following \cite[(5.3)]{EmertonA} let $\Dgnn$ be the strong dual of the space of $\bbG(n)^\circ$-analytic vectors

$$\Dgnn:= (C^{\rm cts}(\GN,K)_{\bbG(n)^\circ-\rm an})'_b \;.$$

\vskip8pt

It is a locally convex topological $L$-algebra naturally isomorphic to the crossed product of the ring $\cD^{\rm an}(\bbG(n)^\circ)$ with the finite group $\GN/G_{n+1}$. In particular,

\begin{numequation}\label{equ-finitefree}\Dgnn=\oplus_{g\in \GN/G_{n+1}} \cD^{\rm an}(\bbG(n)^\circ)*\delta_g\end{numequation}

is a finitely generated free topological module over $\cD^{\rm an}(\bbG(n)^\circ)$. Denoting by $C^{\rm la}(\GN,K)$ the space of $K$-valued locally analytic functions and dualizing the isomorphism

$$  \varinjlim_n C^{\rm cts}(\GN,K)_{\bbG(n)^\circ-\rm an}\car C^{\rm la}(\GN,K)$$

\vskip8pt

yields an isomorphism of topological algebras

$$ D(\GN)\car \varprojlim_n \Dgnn \;.$$

\vskip8pt

This is the weak Fr\'echet-Stein structure on the locally analytic distribution algebra $D(\GN)$ as introduced by Emerton in \cite[Prop. 5.3.1]{EmertonA}. In an obviously similar manner, we may construct the ring $\Dgnnt$ and obtain an isomorphism $\DGNt\car \varprojlim_n \Dgnnt.$

\vskip8pt


\vskip8pt

We consider an admissible locally analytic $\GN$-representation $V$, its coadmissible module $M:=V'_b$ and its subspace of $\Gnc$-analytic vectors $V_{\Gnc-\rm an}\subseteq V$.
The latter subspace is naturally a nuclear Fr\'echet space \cite[Lem. 6.1.6]{EmertonA} and we let $(V_{\Gnc-\rm an})'_b$ be its strong dual. It is a space of compact type and a topological $\Dgnn$-module which is finitely generated \cite[Lem. 6.1.13]{EmertonA}. According to \cite[Thm. 6.1.20]{EmertonA} the modules $M_n:= (V_{\Gnc-\rm an})'$ form a $(\Dgnn)_{n\in\bbN}$-sequence, in the sense of \cite[Def. 1.3.8]{EmertonA}, for the coadmissible module $M$ relative to the weak Fr\'echet-Stein structure on $D(\GN).$ This implies that one has

\begin{numequation}\label{equ-weakfamily}
M_n=\Dgnn\hat{\otimes}_{D(\GN)} M
\end{numequation}

as $\Dgnn$-modules for any $n$. Furthermore, there are natural isomorphisms

$$D(\bbG(n-1)^\circ,\GN)\hat{\otimes}_{\Dgnn} M_n\car M_{n-1} \;,$$

\vskip8pt

of $D(\bbG(n-1)^\circ,\GN)$-modules for any $n$. Here, the completed tensor product are understood in the sense of \cite[Lem. 1.2.3]{EmertonA} as we have explained above. We have the following slight refinements of these results.

\begin{lemma}\label{lem-refine} (i) The $\Dgnn$-module $M_n$ is finitely presented.

\vskip8pt

(ii) The natural map

$$ D(\bbG(n-1)^\circ,\GN)\otimes_{\Dgnn} M_n \car D(\bbG(n-1)^\circ,\GN)\hat{\otimes}_{\Dgnn} M_n$$

\vskip8pt

is bijective.
\end{lemma}

\Pf (i) Let $(M'_n)_{n\in\bbN}$ be any $(\Dgnn)_{n\in\bbN}$-sequence for the Fr\'echet-Stein structure of $D(\GN)$ exhibited
in \cite[Cor. 5.3.19]{EmertonA}. Fix a number $n$. The space $M'_n$ is a finitely generated module over the noetherian Banach algebra $A$ equal to the crossed product of $\cD^{\rm an}(\bbG(n)^\circ)^{(m_n)}$, the level-$m_n$ part of $\Dgn$, with the finite group $\GN/G_{n+1}$. Here, $m_n$ is a sequence of natural numbers with $m_n\rightarrow\infty$ for $n\rightarrow \infty$. In particular,
\begin{numequation}\label{equ-aux} M'_n= A \otimes_{D(\GN)} M = A \hat{\otimes}_{D(\GN)} M.\end{numequation}
The algebra $A$ is naturally contained in $B:=\Dgnn$ as shown in equation (5.3.20) of loc.cit. By \ref{equ-weakfamily} and \ref{equ-aux} we have $M_n=B \hat{\otimes}_{A} M'_n$. However, $M'_n$ is a finitely presented $A$-module and the topological algebras $A$ and $B$ satisfy the assumptions of Lem. \ref{lem-decomplete}: take $A'=A$ and $B'=\Dgn$, cf. \ref{equ-finitefree}. Hence, the natural map $B\otimes_{A} M'_n\car B \hat{\otimes}_{A} M'_n$ is bijective. This proves (i).

\vskip8pt

(ii) The $\Dgnn$-module $M_n$ is finitely presented according to (i). As we have indicated above, $M_n$ is a space of compact type and a topological $\Dgnn$-module. By the uniqueness part in \ref{prop-cantop} its topology therefore coincides with the canonical topology as finitely generated $\Dgn$-module. The assertion follows therefore from Lem. \ref{lem-decomplete}. \qed

\vskip8pt

Remark: These results have obvious analogues when the character $\theta_0$ is involved.

\subsection{The functor $\Loc^\dagger$}
For every $n \ge 0$ we put $\tsD^\dagger_{n,\Q} := \tsD^\dagger_{n,n,\Q}$, and we denote the functor $\Loc^\dagger_{n,n}$ by $\Loc^\dagger_n$ in the following. The group $\GN=\GL_2(\fro)$ acts on the formal scheme $\frX_n$, and there is an induced (left) action of $\GN$ on the sheaf $\tsD^\dagger_{n,\Q}$. Given $g \in \GN$ and a local section $s$ of $\tsD^\dagger_{n,\Q}$, there is thus a local section $g.s$ of $\tsD^\dagger_{n,\Q}$. We can then consider the abelian category ${\rm Coh}(\tsD^\dagger_{n,\Q},\GN)$ of (left) $\GN$-equivariant coherent $\tsD^\dagger_{n,\Q}$-modules. Furthermore, the group $G_{n+1}$ is contained in $\cD^{\rm an}(\bbG(n)^\circ)$ as a set of delta distributions, and for $h \in G_{n+1}$ we write $\delta_h$ for its image in $H^0(\frX_n, \tsD^\dagger_{n,\Q}) = \cD^{\rm an}(\bbG(n)^\circ)_{\theta_0}$.
For $g \in \GN$, $h \in G_{n+1}$, we have $g.\delta_h = \delta_{ghg^{-1}}$, and for a local section $s$ of $\tsD^\dagger_{n,\Q}$ we have then the relation

\begin{numequation}\label{non_commuting} g.(s\delta_h) = (g.s)(g.\delta_h) = (g.s) \delta_{ghg^{-1}} \;.
\end{numequation}

We denote by ${\rm pr}_n:={\rm pr}_{n,n-1}$ the blow-up morphism $\frX_n \ra \frX_{n-1}$. There is a natural morphism of sheaves of rings

\begin{numequation}\label{equ-transit_sheaf}({\rm pr}_n)_*\tsD^\dagger_{n,\Q}\rightarrow \tsD^\dagger_{n-1,\Q}
\end{numequation}

which is $\GN$-equivariant. Given a coherent $\tsD^\dagger_{n,\Q}$-module $\sM_n$ the $\tsD^\dagger_{n-1,\Q}$-module

$$\tsD^\dagger_{n-1,\Q} \otimes_{({\rm pr}_n)_* \tsD^\dagger_{n,\Q}} ({\rm pr}_n)_*\sM_n$$

\vskip8pt

is $\GN$-equivariant via $g.(s\otimes m)=(g.s) \otimes (g.m)$ for local sections $s,m$ and $g \in \GN$. Consider its submodule $\cR_{n-1}$ locally generated by all elements $s \delta_h \otimes m - s\otimes (h.m)$ for $h \in G_n$. Because of \ref{non_commuting} the submodule $\cR_{n-1}$ is $\GN$-stable. We put

$$\tsD^\dagger_{n-1,\Q} \otimes_{({\rm pr}_n)_* \tsD^\dagger_{n,\Q},G_n} ({\rm pr}_n)_*\sM_n:=
\left(\tsD^\dagger_{n-1,\Q} \otimes_{({\rm pr}_n)_* \tsD^\dagger_{n,\Q}} ({\rm pr}_n)_*\sM_n\right)/\;\cR_{n-1} \;.$$

\vskip8pt

\begin{dfn} A {\it $\GN$-equivariant coadmissible module} on $\frX_\infty:=\varprojlim_n \frX_n$ consists of a family $\sM:=(\sM_n)_{n\in\bbN}$ of objects $\sM_n\in {\rm Coh}(\tsD^\dagger_{n,\Q},\GN)$ together with isomorphisms

\begin{numequation}\label{equ-transit_mod}
\tsD^\dagger_{n-1,\Q} \otimes_{({\rm pr}_n)_* \tsD^\dagger_{n,\Q},G_n} ({\rm pr}_n)_*\sM_n\car \sM_{n-1}
\end{numequation}

of $\GN$-equivariant $\tsD^\dagger_{n-1,\Q}$-modules for all $n\geq 1$.
\end{dfn}

A morphism $\sM\rightarrow\sN$ between two such modules consists of morphisms $\sM_n\rightarrow\sN_n$ in ${\rm Coh}(\tsD^\dagger_{n,\Q},\GN)$ compatible with the isomorphisms above.

\vskip8pt

Let $\sM$ be a $\GN$-equivariant coadmissible module on $\frX_\infty$. The isomorphisms \ref{equ-transit_mod} induce morphisms $({\rm pr}_n)_*\sM_n\rightarrow \sM_{n-1}$ having global sections $H^0(\frX_n,\sM_n)\rightarrow H^0(\frX_{n-1},\sM_{n-1})$. We let

$$ H^0(\frX_\infty,\sM):=\varprojlim_n H^0(\frX_n,\sM_n) \;.$$

\vskip8pt

On the other hand, we consider the category of coadmissible $\DGNt$-modules. Given such a module $M$ we have its associated admissible locally analytic $\GN$-representation $V=M'_b$ together with its subspace of $\Gnc$-analytic vectors $V_{\Gnc-\rm an}$. The latter is stable under the $\GN$-action and its dual $M_n:= (V_{\Gnc-\rm an})'_b$ is a finitely presented $\Dgnnt$-module, cf. \ref{lem-refine}.

\vskip8pt

According to Thm. \ref{thm-equivalence} we have the coherent $\tsD^\dagger_{n,\Q}$-module

$$\Loc^\dagger_n(M_n)= \tsD^\dagger_{n,\Q}\otimes_{\Dgnt} M_n$$

\vskip8pt

on $\frX_n$. Using the contragredient $\GN$-action on the dual space $M_n$, we put

$$g.(s \otimes m) := (g.s) \otimes (g.m)$$

\vskip8pt

for $g\in G_0, m\in M_n$ and a local section $s$. In this way, $\Loc^\dagger_n(M_n)$ becomes an object of ${\rm Coh}(\tsD^\dagger_{n,\Q},\GN)$.

\begin{prop}\label{prop-equivalenceII} (i) The family $\Loc^\dagger_n(M_n), n\in\bbN$, forms a $\GN$-equivariant coadmissible module on $\frX_\infty$. Call it
 $\Loc^\dagger(M)$. The formation of $\Loc^\dagger(M)$ is functorial in $M$.

\vskip8pt

(ii) The functors $\Loc^\dagger$ and $H^0(\frX_\infty,\cdot)$ are quasi-inverse equivalences between the categories of coadmissible $\DGNt$-modules and $\GN$-equivariant coadmissible modules on $\frX_\infty$.

\end{prop}

\vskip8pt

\Pf We let $H_n:=G_{n}/G_{n+1}$ and we denote a system of representatives in $G_{n}$ for the cosets in $H_n$ by the same symbol. For simplicity, we abbreviate in this proof

$$D(n):=\Dgnt \hskip10pt {\rm and}\hskip10pt D(n,G_0):=\Dgnnt \;.$$

\vskip8pt

We have the natural inclusion $D(n)\hookrightarrow D(n,G_0)$ from \ref{equ-finitefree} which is compatible with variation in $n$. Now suppose $M$ is a $D(n,G_0)$-module. We then have the free $D(n-1)$-module

$$D(n-1)^{\oplus M\times H_n}$$

\vskip8pt

on a basis $e_{m,h}$ indexed by the elements $(m,h)$ of the set $M\times H_n$. Its formation is functorial in $M$: if $M'$ is another module and $f: M\rightarrow M'$ a linear map, then $e_{m,h}\rightarrow e_{f(m),h}$ induces a linear map between the corresponding free modules. The free module comes with a linear map

$$f_M: D(n-1)^{\oplus M\times H_n}\rightarrow D(n-1)\otimes_{D(n)} M$$

\vskip8pt

given by

$$\oplus_{(m,h)}\lambda_{m,h}e_{m,h}\mapsto (\lambda_{m,h} \delta_h) \otimes m - \lambda_{m,h} \otimes (h \cdot m)$$

\vskip8pt

for $\lambda_{m,h}\in D(n-1)$ where we consider $M$ a $D(n)$-module via the inclusion $D(n)\hookrightarrow D(n,G_0)$.
The map is visibly functorial in $M$ and gives rise to the sequence of linear maps

$$ D(n-1)^{\oplus M\times H_n}\stackrel{f_M}{\longrightarrow} D(n-1)\otimes_{D(n)} M \stackrel{can_M}{\longrightarrow}
D(n-1,G_0)\otimes_{D(n,G_0)} M\longrightarrow 0$$

\vskip8pt

where the second map is induced from the inclusion $D(n)\hookrightarrow D(n,G_0)$.
The sequence is functorial in $M$, since so are both occuring maps.

\vskip8pt

{\it Claim 1: If $M$ is a finitely presented $D(n,G_0)$-module, then the above sequence is exact.}

\vskip8pt

\Pf We show $\ker(can_M)\subseteq {\rm im}(f_M)$, the rest is clear. Choose a finite presentation of the $D(n,G_0)$-module $M$. Since the sequence is functorial in $M$, we obtain a commutative diagram with exact columns. The $3\times 3$-lemma reduces us to the case of a finite free module $M=D(n,G_0)^{\oplus r}$. Since it suffices to see that $\ker(can_M)$ lies in the submodule generated by the images of the elements $e_{m_i,h}$, for generators $m_1, \ldots ,m_r$ of $M$ and $h\in H_n$, we are further reduced to the case $r=1$. In this case the statement follows directly from \ref{equ-finitefree}. \qed

\vskip8pt

{\it Claim 2: Suppose $M$ is a finitely presented $D(n)$-module and let $\sM:=  \Loc^\dagger_n(M)$. The natural morphism

$$\Loc^\dagger_{n-1}(D(n-1)\otimes_{D(n)} M)\car \tsD^\dagger_{n-1,\Q} \otimes_{({\rm pr}_n)_*\tsD^\dagger_{n,\Q}} ({\rm pr}_n)_*\sM $$

\vskip8pt

is bijective.}

\vskip8pt

\Pf We have the equality of sheaves on $\bbX_{n-1}$

\begin{numequation}\label{eq_R1}
R^1({\rm pr}_n)_*\tsD^{(m)}_{n,n}=R^1({\rm pr}_n)_*\tcD^{(m)}_{n,n}=0
\end{numequation}

for any $m$. Indeed, it suffices to verify this over affine open neighborhoods $U$ of points $P \in \bbX_{n-1}^{\rm sm}(\bbF_q)$. A \v{C}ech cohomology argument as in the proof of Lem.\ref{direct_im_str_sh} shows that

$$H^1({\rm pr}_n^{-1}(U), \widetilde{\cT}^{\otimes d}_{n,n})=0$$

\vskip8pt

for any $d\geq 0$. The cohomology sequence
associated to the exact sequence of sheaves over ${\rm pr}_n^{-1}(U)$

$$ 0\rightarrow \tcD^{(m)}_{n,n;d-1}\rightarrow \tcD^{(m)}_{n,n;d}\rightarrow \frac{q_d^{(m)}!}{d!}\widetilde{\cT}^{\otimes d}_{n,n}\rightarrow 0 \;,$$

\vskip8pt

cf. Prop. \ref{graded_tcD}, shows by induction on $d$ that

$$H^1({\rm pr}_n^{-1}(U), \widetilde{\cD}^{(m)}_{n,n;d})=0$$

\vskip8pt

for any $d\geq 0$. Passing to the inductive limit over $d$ gives

\begin{numequation}\label{ML} H^1({\rm pr}_n^{-1}(U), \widetilde{\cD}^{(m)}_{n,n})=0 \;.\end{numequation}

In particular, $R^1({\rm pr}_n)_*\tcD^{(m)}_{n,n}=0$. Let $V={\rm pr}_n^{-1}(U)$ and let $V_j$ be the reduction of $U$ modulo $p^{j+1}$. We may now argue as in the proof of (i) and (ii) of Prop. \ref{completion}: using \ref{ML} together with a commutative diagram similar to \ref{coh_ex_seq} shows that the projective limit of the system

$$(H^1(V_j, \widetilde{\cD}^{(m)}_{n,n}/p^{j+1} \widetilde{\cD}^{(m)}_{n,n}))_{j\geq 0}$$

\vskip8pt

is equal to $H^1(V, \widetilde{\cD}^{(m)}_{n,n})$ and that the system

$$(H^0(V_j, \widetilde{\cD}^{(m)}_{n,n}/p^{j+1} \widetilde{\cD}^{(m)}_{n,n}))_{j\geq 0}$$

\vskip8pt

satisfies the Mittag-Leffler condition. The latter condition allows us to apply \cite[ch. 0, 13.3.1]{EGA_III} where we use a base for the topology of $V$ consisting of open affines. This yields a bijection between the projective limit of the $H^1(V_j, \widetilde{\cD}^{(m)}_{n,n}/p^{j+1} \widetilde{\cD}^{(m)}_{n,n})$ and $H^1(V, \widetilde{\sD}^{(m)}_{n,n})$. This completes the proof of \ref{eq_R1}. We conclude by inverting $p$ and
passing to the inductive limit over $m$ that
$R^1({\rm pr}_n)_*\tsD^\dagger_{n,\Q}=0$. Moreover, since $\frX_n$ is a noetherian space of dimension one, we have $R^2({\rm pr}_n)_*\sE=0$ for any abelian sheaf $\sE$ on $\frX_{n}$ by Grothendieck's vanishing theorem. According to \ref{thm-equivalence}, we therefore have  $R^1({\rm pr}_n)_*\sE=0$ for any coherent $\tsD^\dagger_{n,\Q}$-module $\sE$, in other words, the functor $({\rm pr}_n)_*$ is right-exact on such modules. Choosing a finite presentation of $M$ reduces us now to the case $M=D(n)$ which is obvious. \qed

\vskip8pt

Now let $M$ be a finitely presented $D(n,G_0)$-module. Let $m_1, \ldots ,m_r$ be generators for $M$ as a $D(n)$-module. We have a sequence of $D(n-1)$-modules

$$\bigoplus_{i,h} D(n-1)e_{m_i,h}\stackrel{f'_M}{\longrightarrow} D(n-1)\otimes_{D(n)} M \stackrel{can_M}{\longrightarrow}
D(n-1,G_0)\otimes_{D(n,G_0)} M\longrightarrow 0$$

\vskip8pt

where $f'_M$ denotes the restriction of the map $f_M$ to the free submodule of $D(n-1)^{\oplus M\times H_n}$
generated by the finitely many vectors $e_{m_i,h}, i=1,\ldots,r$, $h \in H_n$. Since ${\rm im}(f'_M)={\rm im}(f_M)$ the sequence is exact by the first claim. Since it consists of finitely presented $D(n-1)$-modules, we may apply the exact functor $\Loc^\dagger_{n-1}$ to it. By the second claim, we get an exact sequence

$$ (\tsD^{\dagger}_{n-1,\Q})^{\oplus r|H_n|} \ra \tsD^\dagger_{n-1,\Q} \otimes_{({\rm pr}_n)_*\tsD^\dagger_{n,\Q}} ({\rm pr}_n)_*\sM \ra
\Loc^\dagger_{n-1}(D(n-1,G_0)\otimes_{D(n,G_0)} M) \ra 0$$

\vskip8pt

where $\sM = \Loc^\dagger_{n}(M)$. The cokernel of the first map in this sequence equals by definition

$$\tsD^\dagger_{n-1,\Q} \otimes_{({\rm pr}_n)_* \tsD^\dagger_{n,\Q},G_n} ({\rm pr}_n)_*\sM \;,$$

\vskip8pt

whence an isomorphism

$$ \tsD^\dagger_{n-1,\Q} \otimes_{({\rm pr}_n)_* \tsD^\dagger_{n,\Q},G_n} ({\rm pr}_n)_*\sM \car \Loc^\dagger_{n-1}(D(n-1,G_0)\otimes_{D(n,G_0)} M) \;.$$

\vskip8pt

This implies both parts of the proposition. \qed

\vskip8pt

\subsection{The functor $\Loc_\infty^\dagger$ and $\GL_2(L)$-equivariance} Let $G:=\GL_2(L)$. Denote the canonical projection map $\frX_\infty\rightarrow\frX_n$ by ${\rm sp}_n$ for each $n$. We define the following sheaf of rings on $\frX_\infty$. Assume $V\subseteq\frX_\infty$ is an open subset of the form ${\rm sp}_N^{-1}(U)$ with an open subset $U\subseteq\frX_N$ for some $N$. We have that

$${\rm sp}_n(V)={\rm pr}_{n,N}^{-1}(U)$$

\vskip8pt

for any $n\geq N$ and so, in particular, ${\rm sp}_n(V)\subseteq\frX_n$ is an open subset for these $n$.
Moreover,

$${\rm pr}_{n,n-1}^{-1}({\rm sp}_{n-1}(V))={\rm sp}_n(V)$$

\vskip8pt

for all $n>N$. The morphism \ref{equ-transit_sheaf} induces the ring homomorphism

\begin{numequation}\label{equ-transit_hom} \tsD^{\dagger}_{n,\Q}({\rm sp}_n(V))=({\rm pr}_{n,n-1})_*\tsD^{\dagger}_{n,\Q}(({\rm sp}_{n-1}(V)))\rightarrow
\tsD^{\dagger}_{n-1,\Q}({\rm sp}_{n-1}(V))
\end{numequation}

for all $n>N$ and we form the projective limit

$$\tsD^\dagger_{\infty,\Q}(V):=\varprojlim_{n>N} \tsD^{\dagger}_{n,\Q}({\rm sp}_n(V))$$

\vskip8pt

over all these maps. The open subsets of the form $V$ form a
basis for the topology on $\frX_\infty$ and $\tsD^\dagger_{\infty,\Q}$ is a presheaf on this basis. We denote the associated sheaf
on $\frX_\infty$ by the symbol $\tsD^\dagger_{\infty,\Q}$ as well. It is a $G_0$-equivariant sheaf of rings on $\frX_\infty$.

\vskip8pt

Suppose $\sM:=(\sM_n)_n$ is a $G_0$-equivariant coadmissible module on $\frX_\infty$ as defined in the previous subsection.
The isomorphisms \ref{equ-transit_mod} induce $G_0$-equivariant maps $({\rm pr}_n)_*\sM_n\rightarrow\sM_{n-1}$ which are linear relative to the morphism \ref{equ-transit_sheaf}. In a completely analogous manner as above, we obtain a sheaf $\sM_\infty$ on $\frX_\infty$.
It is a $G_0$-equivariant (left) $\tsD^\dagger_{\infty,\Q}$-module on $\frX_\infty$ whose formation is functorial in $\sM$.

\vskip8pt

\begin{prop}
The functor $\sM\rightarrow\sM_\infty$ from $G_0$-equivariant coadmissible modules on $\frX_\infty$ to $G_0$-equivariant $\tsD^\dagger_{\infty,\Q}$-modules is a
fully faithful embedding.
\end{prop}

\Pf We have ${\rm sp}_n(\frX_\infty)=\frX_n$ for all $n$.
The global sections of $M_\infty$ are therefore equal to

$$\Gamma(\frX_\infty,\sM_\infty)=\varprojlim_n\Gamma(\frX_n,\sM_n)=H^0(\frX,\sM)$$

\vskip8pt

in the notation of the previous section. Thus, the functor $\Loc^\dagger \circ \Gamma(\frX_\infty,-)$ is a left quasi-inverse according to Prop. \ref{prop-equivalenceII}.\qed

\vskip8pt

We tentatively call the abelian subcategory of all $G_0$-equivariant (left) $\tsD^\dagger_{\infty,\Q}$-modules equal to the essential image of the functor $(\cdot)_\infty$ the {\it coadmissible} $G_0$-equivariant $\tsD^\dagger_{\infty,\Q}$-modules. We denote by $\Loc^\dagger_\infty$ the composite of the functor $\Loc^\dagger$ with $(\cdot)_\infty$.


\vskip8pt

\begin{para}
The $G_0$-action on the space $\frX_\infty$ admits a natural extension to an action of the full group $G$. Indeed, let $\cT$ be the Bruhat-Tits tree for the group ${\rm PGL}_2(L)$ with its natural $G$-action. Its vertices correspond to the homothety classes of lattices $\Lambda \sub L^{\oplus 2}$. Given a vertex $v\in\cT$ we denote by $|\cT(v)|_{\le n} \sub \cT$ the set of points whose distance to $v$ is less or equal to $n$. There is a unique maximal compact subgroup $G_v\subseteq G$ fixing the vertex $v$. We denote by $v_0$ the vertex corresponding to the class of the standard lattice $\Lambda_0 = \fro \oplus \fro$. In this case $G_{v_0}=G_0$. As mentioned in \ref{subsect-con_blow_up}, all constructions in sections \ref{models}-\ref{loc_n} are associated to the standard lattice $\Lambda_0$ but can be done canonically for (the homothety class of) an arbitrary lattice $\Lambda \sub L^{\oplus 2}$. We distinguish the various constructions from each other by adding the correspon
 ding vertex $v=[\Lambda]$ to them, i.e. we write $\bbX_n(v), \frX_n(v), \tsD^{(m)}_{n,k}(v)$ and so on.

\vskip8pt

An element $g\in G$ induces a morphism

$$\bbX_0(v)={\rm Proj}(\Lambda)\stackrel{g.}{\longrightarrow}{\rm Proj}(g\Lambda)=\bbX_0(gv)$$

\vskip8pt

which satisfies $(gh).=(g.)\circ(h.)$ and $1.={\rm id}$ for $g,h\in G$. By the universal property of blow-up, this induces a morphism $\bbX_n(v)\stackrel{g.}{\longrightarrow}\bbX_n(gv)$ compatible with variation in $n$. These morphisms pass to the formal completions $\frX_n(v)$. Given an inclusion

$$|\cT(v)|_{\le n}\subseteq |\cT(v')|_{\le n'}$$

\vskip8pt

the space $\bbX_{n'}(v')$ equals the blow-up of $\bbX_n(v)$ in the $\bbF_q$-valued points of $\bbX_n(v)^{\rm sm}$ which correspond to the vertices in the complement of $|\cT(v)|_{\le n}$ in $|\cT(v')|_{\le n'}$. The action of $G$ is compatible with the blow-up morphism. The same holds for the formal completions. We obtain thus a $G$-action on the projective limit over all $\frX_n(v)$. Since a given set $|\cT(v)|_{\le n}$ is contained in $|\cT(v_0)|_{\le n'}$ for large enough $n'$, this projective limit equals $\frX_\infty$ by cofinality.
We thus have indeed a natural $G$-action on $\frX_\infty$.

\vskip8pt

We note that, by definition of the action, there is an equality

\begin{numequation}\label{equivariant}{\rm sp}_{n,gv}(g.V)=g.{\rm sp}_{n,v}(V)
\end{numequation}

in $\frX_n(g.v)$ for any subset $V\subseteq\frX_\infty$ and the specialization map ${\rm sp}_{n,v}: \frX_\infty\rightarrow\frX_n(v)$.
\end{para}

\vskip8pt

\begin{prop}
The $G_0$-equivariant structure on the sheaf $\tsD^\dagger_{\infty,\Q}$ extends to a $G$-equivariant structure.
\end{prop}

\Pf Let $g\in G$. The isomorphism $\frX_n(v)\stackrel{g.}{\longrightarrow}\frX_n(gv)$ induces a ring isomorphism

\begin{numequation}\label{equ-ringiso}\tsD^\dagger_{n,\Q}(v)(U)\stackrel{g.}{\longrightarrow}\tsD^\dagger_{n,\Q}(gv)(g.U)\end{numequation}

for any open subset $U\subseteq\frX_n(v)$. In particular, for an open subset $V\subseteq\frX_\infty$ of the form $V={\rm sp}_{N,v}^{-1}(U)$ with $U\subseteq\frX_N(v)$ open, this gives a ring homomorphism for $n>N$

\begin{numequation}\label{equ-action} \tsD^\dagger_{n,\Q}(v)({\rm sp}_{n,v}(V))\stackrel{g.}{\longrightarrow}\tsD^\dagger_{n,\Q}(gv)(g.{\rm sp}_{n,v}(V))
=\tsD^\dagger_{n,\Q}(gv)({\rm sp}_{n,gv}(gV)) \;.
\end{numequation}

where we have used \ref{equivariant}. Given an inclusion
$|\cT(v)|_{\le n}\subseteq |\cT(v')|_{\le n'}$ the blow-up morphism

$${\rm pr}_{(n',v'),(n,v)}:\frX_{n'}(v')\longrightarrow \frX_n(v)$$

\vskip8pt

induces a morphism of sheaves of rings

\begin{numequation}\label{equ-transit_sheafII} ({\rm pr}_{(n',v'),(n,v)})_* \tsD^\dagger_{n',\Q}(v')\longrightarrow \tsD^\dagger_{n,\Q}(v)\end{numequation}

in a manner analogous to \ref{equ-transit_sheaf}. Indeed, let $\bbG_v$ denote the smooth affine Bruhat-Tits group scheme over $\fro$ associated with the vertex $v$ and let $\frg_v$ be its Lie-algebra. From the commutative diagram

$$\begin{array}{ccccc} \Lie(\bbG_{v'}(n')) & = & \vpi^{n'} \frg_{v'} & \lra & H^0(\bbX_{n'}(v'),\cT_{\bbX_{n'}(v')^{\log}})\\
& & \downarrow  && \downarrow \\
\Lie(\bbG_v(n)) & = & \vpi^n\frg_v & \lra & H^0(\bbX_{n}(v),\cT_{\bbX_{n}(v)^{\log}})
\end{array}$$

\vskip8pt

we deduce a morphism

\vskip8pt

$$({\rm pr}_{(n',v'),(n,v)})_*\tcT^{(m)}_{n',n'}(v') \ra \tcT^{(m)}_{n,n}(v)$$

\vskip8pt

for any $m$. Passing to differential operators yields the claimed morphism.

\vskip8pt

Given $V \sub \frX_\infty$ of the form $V={\rm sp}_{\tilde{n},\tilde{v}}^{-1}(U)$ with an open set $U\subseteq\frX_{\tilde{n}}(\tilde{v})$, the morphism \ref{equ-transit_sheafII} induces a ring homomorphism

\begin{numequation}\label{equ-transit_homII} \tsD^{\dagger}_{n',\Q}(v')({\rm sp}_{n',v'}(V))
=({\rm pr}_{(n',v'),(n,v)})_*\tsD^{\dagger}_{n',\Q}(v')({\rm sp}_{n,v}(V))\rightarrow  \tsD^{\dagger}_{n,\Q}(v)({\rm sp}_{n,v}(V))
\end{numequation}

whenever there are inclusions

$$|\cT(\tilde{v})|_{\le \tilde{n}}\subseteq |\cT(v)|_{\le n}\subseteq |\cT(v')|_{\le n'}$$

\vskip8pt

and we may form the projective limit

$$ \varprojlim_{(n,v)}\tsD^{\dagger}_{n,\Q}(v)({\rm sp}_{n,v}(V)) \;.$$

\vskip8pt

However, again by cofinality, this projective limit equals $\tsD^\dagger_{\infty,\Q}(V)$.
Since the homomorphism \ref{equ-action} is compatible with variation in $n$ we deduce for a given $g\in G$ a ring homomorphism

\vskip8pt

$$\tsD^\dagger_{\infty,\Q}(V)=\varprojlim_{(n,v)}\tsD^{\dagger}_{n,\Q}(v)({\rm sp}_{n,v}(V))\stackrel{g.}{\rightarrow}
\varprojlim_{(n,gv)}\tsD^{\dagger}_{n,\Q}(gv)({\rm sp}_{n,gv}(gV))=\tsD^\dagger_{\infty,\Q}(gV) \;.$$

\vskip8pt

It implies that the sheaf $\tsD^{\dagger}_{\infty,\Q}$ is $G$-equivariant. It is clear from the construction that the $G$-equivariant structure extends the $G_0$-structure.\qed

\vskip8pt

A coadmissible $G_0$-equivariant $\tsD^\dagger_{\infty,\Q}$-module whose equivariant structure extends to the full group $G$, will simply be called a coadmissible $G$-equivariant $\tsD^\dagger_{\infty,\Q}$-module.

\begin{thm}\label{thm-equivariance}
The functors $\Loc^\dagger_\infty$ and $\Gamma(\frX_\infty,\cdot)$ are quasi-inverse equivalences between the categories of coadmissible $D(G_0)_{\theta_0}$-modules and coadmissible $G_0$-equivariant $\tsD^\dagger_{\infty,\Q}$-modules. The subcategories
of coadmissible $D(G)_{\theta_0}$-modules and coadmissible $G$-equivariant $\tsD^\dagger_{\infty,\Q}$-modules correspond to each other.
\end{thm}

\Pf We only need to show the second statement. It is clear that a coadmissible $D(G_0)_{\theta_0}$-module which comes from a
coadmissible $G$-equivariant $\tsD^\dagger_{\infty,\Q}$-module is a $D(G)_{\theta_0}$-module. For the converse, we consider a vertex $v\in\cT$ and the corresponding localisation functor $\Loc^\dagger_{n,\Q}(v)$ which is an equivalence between finitely presented $\cD^{an}(\bbG_v(n)^\circ)_{\theta_0}$-modules and coherent $\tsD^\dagger_{n,\Q}(v)$-modules on $\frX_n(v)$. Here, $\bbG_v$ denotes the smooth affine Bruhat-Tits group scheme over $\fro$ associated with the vertex $v$. The adjoint action of $G$ on its Lie algebra induces a ring isomorphism

\begin{numequation}\label{equ-ringisoII}
\cD^{an}(\bbG_{v}(n)^\circ)\stackrel{g.}{\rightarrow}\cD^{an}(\bbG_{gv}(n)^\circ)\end{numequation}

for any $g\in G$. Now consider a coadmissible $D(G)_{\theta_0}$-module $M$ with dual space $V=M'$. We have the family $(\sM_{n,v})_{(n,v)}$ where

$$\sM_{n,v} = \Loc^\dagger_{n,\Q}(v)(M_{n,v}) = \tsD^\dagger_{n,\Q}(v)\otimes_{\cD^{an}(\bbG_v(n)^\circ)_{\theta_0}} M_{n,v}$$

\vskip8pt

and $M_{n,v}=(V_{\bbG_v(n)^\circ-\rm an})'$.
Let $g\in G$. The map $m \mapsto gm$ on $M$ induces a map $M_{n,v}\rightarrow M_{n,gv}$ which is linear relative to \ref{equ-ringisoII}.  We therefore have for any open subset $U\subseteq\frX_n(v)$ a homomorphism

$$\sM_{n,v}(U)\stackrel{g.}{\longrightarrow}\sM_{n,gv}(g.U)$$

\vskip8pt

which is induced by the map

$$s \otimes m \mapsto (g.s) \otimes gm \;.$$

\vskip8pt

for $s\in \tsD^\dagger_{n,\Q}(v)(U), m\in M_{n,v}$ and where
$g.$ is the ring isomorphism \ref{equ-ringiso}. In particular, for an open subset $V\subseteq\frX_\infty$ of the form $V={\rm sp}_{N,v}^{-1}(U)$ with $U\subseteq\frX_N(v)$ open, this gives a homomorphism for $n>N$

\begin{numequation}\label{equ-actionII}
\sM_{n,v}({\rm sp}_{n,v}(V)) \stackrel{g.}{\longrightarrow} \sM_{n,gv}(g.{\rm sp}_{n,v}(V)) = \sM_{n,gv}({\rm sp}_{n,gv}(gV))
\end{numequation}

which is linear relative to the ring homomorphism \ref{equ-action}.

\vskip8pt

Now for every inclusion

$$|\cT(v)|_{\le n}\subseteq |\cT(v')|_{\le n'} \;,$$

\vskip8pt

we have a morphism

\begin{numequation}\label{equ-mor}({\rm pr}_{(n',v'),(n,v)})_*\sM_{n',v'}\longrightarrow \sM_{n,v}\end{numequation}

compatible with the morphism of rings \ref{equ-transit_sheafII} as follows. First of all, one has

$$({\rm pr}_{(n',v'),(n,v)})_*\left(\Loc^\dagger_{n',\Q}(v')(M_{n',v'})\right)=
\left(({\rm pr}_{(n',v'),(n,v)})_*(\tsD^\dagger_{n',\Q}(v')\right)\otimes_{\cD^{an}(\bbG_{v'}(n')^\circ)_{\theta_0}} M_{n',v'}$$

\vskip8pt

which can be proved along the lines of claim 2 in the proof of Prop. \ref{prop-equivalenceII}. Moreover, we have inclusions $\bbG_{v'}(n')\subseteq \bbG_{v}(n)$ and thus

$$V_{\bbG_{v}(n)^\circ-\rm an}\subseteq V_{\bbG_{v'}(n')^\circ-\rm an} \;.$$

\vskip8pt

The dual map $M_{n',v'}\rightarrow M_{n,v}$ is linear relative to the natural inclusion

$$\cD^{an}(\bbG_{v'}(n')^\circ)\rightarrow \cD^{an}(\bbG_{v}(n)^\circ) \;.$$

\vskip8pt

The latter inclusion is compatible with the morphism of rings \ref{equ-transit_sheafII} via taking global sections. Hence, we have a morphism \ref{equ-mor} as claimed. We now have everything at hand to follow the arguments in the proof of the preceding proposition word for word and to conclude
that the projective limit $\sM_\infty$ has a $G$-action which extends its $G_0$-action and which makes it a $G$-equivariant $\tsD^\dagger_{\infty,\Q}$-module. This completes the proof of the theorem. \qed

\vskip12pt

Remark: Let $L\subset K$ be a complete and discretely valued extension field. If we consider the $K$-algebras $\cD^{an}(\bbG(k)^\circ)\hat{\otimes}_L K, D(G_0)\hat{\otimes}_L K$ etc.
as well as the sheaves of $K$-algebras $\tsD^{(m)}_{n,k,\bbQ}\hat{\otimes}_L K, \tsD^{\dagger}_{n,k,\bbQ}\hat{\otimes}_L K$, then there are obvious
versions 'over $K$' of the theorems \ref{thm-equivalence} and \ref{thm-equivariance}.

\section{Examples of localizations}\label{examples}

\subsection{Smooth representations}\label{smooth_reps}
\setcounter{para}{0}

If $V$ is a smooth $G$-representation (i.e. the stabilizer of each vector $v\in V$ is an open subgroup of $G$), then
$V_{\bbG(n)^\circ-an}$ equals the space of fixed vectors $V^{G_{n+1}}$ in $V$ under the action of the compact subgroup $G_{n+1}$.
Suppose $V$ is admissible (i.e. the vector space $V^{G_{n+1}}$ is finite-dimensional). Since $\frg V=0$, one finds

\begin{numequation}\label{smooth_loc}
\Loc^\dagger_n ((V^{G_{n+1}})')= \cO_{\frX_n,\bbQ}\otimes_L (V^{G_{n+1}})' \;,
\end{numequation}

\vskip8pt

where $G_0$ acts diagonally and $\tsD^\dagger_{n,\bbQ}$ acts through its natural action on $\cO_{\frX_n,\bbQ}$.

\vskip8pt

\subsection{Representations attached to $U(\frg)$-modules}\label{reps_att_to_Lie_reps}

In this section, we will compute the $G_0$-equivariant coadmissible module $\Loc^\dagger(\bM)$ on $\frX_\infty$ for a class of coadmissible $D(G)$-modules $\bM$ related to the pair $(\frg,B)$. Here $\bbB\subset\GL_{2,L}$ denotes the Borel subgroup of upper triangular matrices and $B$ denotes its group of $L$-rational points. Let $\frb$ be its Lie-algebra. Let $T\sub B$ be the subgroup of diagonal matrices with Lie algebra $\frt$.

Let $L\subset K$ be a complete and discretely valued extension field. To simplify notation, we make the convention that, when dealing with universal enveloping algebras, distribution algebras, differential operators etc. we write $U(\frg)$, $D(G_0)$, $\tsD^{\dagger}_{n,k,\bbQ}$ etc. to denote the corresponding objects {\it after base change to $K$}, i.e., what is precisely $U(\frg_K)$, $D(G_0)\hat{\otimes}_L K$, $\tsD^{\dagger}_{n,k,\bbQ}\hat{\otimes}_L K$ and so on (compare also final remark in sect. $5$).

\vskip8pt
The group $G$ and its subgroup $B$ act via the adjoint representation on $U(\frg)$ and we denote by \begin{numequation}\label{equ-DF} D(\frg,B):=D(B)\otimes_{U(\frb)} U(\frg)\end{numequation}

the corresponding skew-product ring. The skew-multiplication here is induced by
$$(\delta_{b'} \otimes x')\cdot (\delta_{b}\otimes x)=\delta_{b'b}\otimes \delta_{b^{-1}}(x')x$$
for $b,b'\in B$ and $x,x'\in U(\frg)$. If $\lambda$ denotes a locally analytic character of $B$, then we have a canonical algebra isomorphism

\begin{numequation}\label{equ-reduction} D(\frg,B)/D(\frg,B)I(\lambda)\simeq U(\frg)/U(\frg)I(d\lambda)
\end{numequation}

where $I(\lambda)$ resp. $I(d\lambda)$ denote the ideals equal to the kernel of $D(B)\stackrel{\lambda}{\longrightarrow}K$ resp. $\frb\stackrel{d\lambda}{\longrightarrow}K$. Replacing $B$ by $B_0=B\cap G_0$, we obtain a skew-product ring $D(\frg,B_0)$ with similar properties. Given a ${D(\frg,B)}$-module $M$ one has

\begin{numequation}\label{equ-G0}D(G)\otimes_{D(\frg,B)} M= D(G_0)\otimes_{D(\frg,B_0)} M\end{numequation}

as $D(G_0)$-modules \cite[4.2]{SchmSt}. We consider the functor

$$M\mapsto \bM:= D(G)\otimes_{D(\frg,B)} M$$

\vskip8pt

from ${D(\frg,B)}$-modules to $D(G)$-modules \cite{OrlikStrauchJH}. If $M$ is finitely generated as $U(\frg)$-module, then $\bM$ is coadmissible by \cite[4.3]{SchmSt}. From now on we assume that $M$ is a finitely generated $U(\frg)$-module. We let $V := \bM'_b$ be the locally analytic $G$-representation corresponding to $\bM$ and denote by

$$ \bM_n:=(V_{\bbG(n)^\circ-\rm an})'_b$$

\vskip8pt

the dual of the subspace of its $\bbG(n)^\circ$-analytic vectors. According to \ref{lem-refine} the $\Dgnn$-module $\bM_n$ is finitely presented and has its canonical topology.

\begin{lemma}
The canonical map

$$\Dgnn\otimes_{D(\GN)} \bM\car \bM_n$$

\vskip8pt

is an isomorphism.
\end{lemma}

\Pf It suffices to show that the left-hand side is complete. We have

$$\Dgnn\otimes_{D(\GN)} \bM = \Dgnn \otimes_{D(\frg,B_0)} M \;,$$

\vskip8pt

according to \ref{equ-G0}. The group $B_0$ is topologically finitely generated. As in the proof of \cite[4.3]{SchmSt} we may write the module $\Dgnn\otimes_{D(\GN)} \bM$ as the quotient of the finitely presented module
 $\Dgnn\otimes_{U(\frg)} M$ by a finitely generated submodule. By (ii) of \ref{prop-cantop}, this quotient is complete. \qed

\vskip8pt

Recall the congruence subgroup $G_{n+1}=\bbG(n)^\circ(L)$ of $G_0$. Put $B_{n+1}:=G_{n+1}\cap B_0$. The corresponding skew-product ring $D(\frg,B_{n+1})$ is contained in $\cD^{an}(\bbG(n)^\circ)$ according to \ref{equ-finitefree}. Let $C(n)$ be a system of representatives in $G_0$ containing $1$
for the residue classes in $G_0/G_{n+1}$ modulo the subgroup $B_0/B_{n+1}$. Note that for an element $g\in G_0$ and a $\cD^{an}(\bbG(n)^\circ)$-submodule $N$ of $D(G_0)$, the abelian group $\delta_gN$ is again a $\cD^{an}(\bbG(n)^\circ)$-submodule because of the formula $x \delta_g = \delta_g {\rm Ad}(g^{-1})(x)$  for any $x\in\cD^{an}(\bbG(n)^\circ)$.

\begin{lemma}
The natural map of $(\cD^{an}(\bbG(n)^\circ),D(\frg,B_0))$-bimodules

$$\sum: \bigoplus_{g\in C(n)} \delta_g\Big(\cD^{an}(\bbG(n)^\circ)\otimes_{D(\frg,B_{n+1})} D(\frg,B_0)\Big)\car \Dgnn$$

\vskip8pt

is an isomorphism.
\end{lemma}

\Pf Since $Lie(B_{n+1})=Lie(B_0)=\frb$, the definition \ref{equ-DF} implies that $D(\frg,B_0)$ equals the direct sum over $B_0/B_{n+1}$ copies of
$D(\frg,B_{n+1}).$ Hence, \ref{equ-finitefree} implies the claim.
\qed

\vskip8pt

The two lemmas allow us to write

$$ \bM_n = \oplus_{g\in C(n)} \delta_g \Big(\cD^{an}(\bbG(n)^\circ) \otimes_{D(\frg,B_{n+1})} M\Big) = \oplus_{g\in C(n)}\delta_gM^{an}_n$$

\vskip8pt

as modules over $\cD^{an}(\bbG(n)^\circ).$
Here $$M^{an}_n:= \cD^{an}(\bbG(n)^\circ) \otimes_{D(\frg,B_{n+1})} M,$$

a finitely presented $\cD^{an}(\bbG(n)^\circ)$-module. If $M$ has character $\theta_0$, so has $M^{an}_n$. As explained above, the 'twisted' module
$\delta_gM^{an}_n$ can and will be viewed as having the same underlying group as $M^{an}_n$ but with an action of $\cD^{an}(\bbG(n)^\circ)$ pulled-back by the
automorphism ${\rm Ad}(g^{-1})$. Since the adjoint action of $G_0$ fixes the center in $U(\frg)$, the character of the module $\delta_gM^{an}_n$ (if existing) does not depend on $g$.

 \vskip8pt

If $M$ has character $\theta_0$, then the $\tsD^\dagger_{n,\bbQ}$-module $\Loc^\dagger_n(\delta_gM^{an}_n)$ on $\frX_n$ can be described as follows. For any $g\in G_0$ let $g_*$ denote the direct image functor coming from the automorphism $g$ of $\frX_n$. If $N$ denotes a (coherent) $\tsD^\dagger_{n,\bbQ}$-module, then $g_*N$ is a (coherent) $\tsD^\dagger_{n,\bbQ}$-module via
the isomorphism $\tsD^\dagger_{n,\bbQ}\car g_*\tsD^\dagger_{n,\bbQ}$.

\begin{lemma}
One has

$$ \Loc^\dagger_n(\delta_gM^{an}_n) = g_* \Loc_n^\dagger (M^{an}_n) = g_* \Big( \tsD^\dagger_{n,\bbQ}\otimes_{D(\frg,B_{n+1})} M\Big) \;.$$

\vskip8pt

\end{lemma}

\Pf The second identity follows from the definition of $\Loc^{\dagger}_n$ by contracting tensor products.
The automorphism $\tsD^\dagger_{n,\bbQ}\car g_*\tsD^\dagger_{n,\bbQ}$
equals ${\rm Ad}(g^{-1})$ on global sections. The identity map is a $\cD^{an}(\bbG(n)^\circ)$-linear isomorphism between the global sections of $\Loc^\dagger_n(\delta_gM^{an}_n)$ and
$g_* \Loc_n^\dagger (M^{an}_n)$. Since both sheaves are coherent $\tsD^\dagger_{n,\bbQ}$-modules, the isomorphism extends to the sheaves by \ref{thm-equivalence}.
\qed

\vskip8pt

Since $\Loc^\dagger_n$ commutes with direct sums the lemma implies

$$ \Loc^\dagger_n (\bM_n)= \oplus_{g\in C(n)}  \;g_*  \Big( \tsD^\dagger_{n,\bbQ}\otimes_{D(\frg,B_{n+1})} M\Big)$$
as $\tsD^\dagger_{n,\bbQ}$-modules. Of course,  $\Loc^\dagger_n (\bM_n)$ is $G_0$-equivariant and the collection of the $\Loc^\dagger_n (\bM_n)$ for all $n$, is equal to the
$G_0$-equivariant coadmissible module on $\frX_\infty$ corresponding to $\bM$ by the equivalence \ref{prop-equivalenceII}.

\vskip8pt

\subsection{Principal series representations}\label{princ_series}

\begin{para} {\it A general formula.} Let $\lambda$ be a locally analytic $K$-valued character of $T$, viewed as a character of $B$. Let

$$M(\lambda)=U(\frg)\otimes_{U(\frb)} K_{d\lambda}$$

\vskip8pt

be the induced module. Then $M(\lambda)$ is a $D(\frg,B)$-module \cite{SchmSt}. Let $\bM(\lambda)$ be the coadmissible $D(G)$-module associated with $M(\lambda)$. Its dual equals the locally analytic principal series representation

$$V:={\rm Ind}_B^G(\lambda^{-1})=\{ f \in C^{\rm la}(G,K): f(gb)=\lambda(b)f(g) {\rm ~for~all~}g\in G, b \in B\}$$
with $G$ acting by left translations. Here, $C^{\rm la}(\cdot,K)$ denotes $K$-valued locally analytic functions. We now choose $n$ large
enough such that the restriction of $\lambda$ to $T\cap G_{n+1}$ is $\bbT(n)^\circ$-analytic. Denote by $\bbU^-$ the unipotent radical of the Borel subgroup scheme opposite to $\bbB$. Restriction of functions induces an isomorphism

$$V_{\bbG(n)^\circ-\rm an} \simeq \oplus_{g \in G_0/G_{n+1}B_0} C^{an}(\bbU^{-}(n)^{\circ,g},K)$$
where $\bbU^{-}(n)^{\circ,g}=g\bbU^{-}(n)^{\circ}g^{-1}$ and where $C^{\rm an}(\cdot,K)$ denotes $K$-valued rigid analytic functions.
Indeed, it is not difficult to see that, as $G_{n+1}$-representations, one has

$$V|_{G_{n+1}}=\oplus_{g \in G_0/G_{n+1}B_0} {\rm Ind}_{G_{n+1}\cap B_0^g}^{G_{n+1}}( \lambda^{-1,g})$$

where $B_0^g=gB_0g^{-1}$ and where $\lambda^{g}$ denotes the character of $B_0^g$ defined by $\lambda^{g}(h)=\lambda(g^{-1}hg).$
Now the $\bbG(n)^\circ$-analytic vectors in $V_g:={\rm Ind}_{G_{n+1}\cap B_0^g}^{G_{n+1}}(\lambda^{-1,g})$
are the preimage of $V_g$ under the (injective) restriction map

$$C^{\rm an}(\bbG(n)^\circ,K)\ra C^{\rm la}(G_{n+1},K),$$

so that

$$(V_g)_{\bbG(n)^\circ-\rm an}= \{f \in C^{\rm an}(\bbG(n)^\circ,K): f(hb)=\lambda^g(b)f(h) {\rm ~for~all~}h\in G_{n+1}, b \in G_{n+1}\cap B_0^g\}.$$
Our assertion follows now from the rigid-analytic (Iwahori) decomposition

$$\bbG(n)^\circ=\bbU^{-}(n)^{\circ,g} \times \bbT(n)^{\circ,g} \times \bbU(n)^{\circ,g}$$
together with the fact that $\lambda^{g}$ is $\bbT(n)^{\circ,g}$-analytic for all $g$.

\vskip8pt

We have

$$M(\lambda) = D(\frg,B_{n+1})/D(\frg,B_{n+1})I_{n+1}(\lambda)$$

\vskip8pt

as a $D(\frg,B_{n+1})$-module where $I_{n+1}(\lambda)$ denotes the kernel of $D(B_{n+1}) \stackrel{\lambda}{\longrightarrow} K$, cf. \ref{equ-reduction}. Let $\cD^{an}(\bbB(n)^\circ)$ be the analytic distribution algebra of $\bbB(n)^\circ$. By our assumption on $n$, the character $d\lambda$ extends to a character of $\cD^{an}(\bbB(n)^\circ)$ whose kernel is generated by $I(d\lambda)\sub U(\frb)$. It follows

$$M(\lambda)^{an}_n = \cD^{an}(\bbG(n)^\circ)/\cD^{an}(\bbG(n)^\circ) I_{n+1}(\lambda) = \cD^{an}(\bbG(n)^\circ)\otimes_{U(\frg)} M(\lambda) \;.$$

\vskip8pt

The Beilinson-Bernstein localization \cite{BB81} of a finitely generated $U(\frg)$-module $M$ with character $\theta_0$ is a coherent $D$-module ${\rm Loc}(M)$ on the algebraic flag variety $X$ of $\GL_{2,L}$. Let $X^{\rm rig}$ be the associated rigid-analytic space with its canonical morphism $\rho: X^{\rm rig}\rightarrow X$. Then $({\rm sp}_{n})_*\rho^* {\rm Loc}(M)$ is an $\cO_{\frX_n,\bbQ}$-module with an action of the sheaf of algebraic differential operators $\pi_n^*\cD_{\frX_0,\bbQ}$. We denote its base change along the natural morphism

$$\pi_n^*\cD_{\frX_0,\bbQ}\longrightarrow \tsD^\dagger_{n,\bbQ}$$

\vskip8pt

by ${\rm Loc}(M)^\dagger_n$, a coherent $\tsD^\dagger_{n,\bbQ}$-module. Suppose now that $d\lambda$ is associated by the Harish-Chandra isomorphism to the central character $\theta_0$. We then have

$${\rm Loc}(M(\lambda))^\dagger_n=\Loc^\dagger_n (M(\lambda)^{an}_n)$$

\vskip8pt

by the above. We can thus conclude

\begin{numequation}\label{general_formula} \Loc^\dagger_n (\bM(\lambda)_n)= \oplus_{g\in C(n)}  \;g_*{\rm Loc}(M(\lambda))^\dagger_n \;.
\end{numequation}

\end{para}

\begin{para} {\it The cases when $\lambda = {\bf 1}$ or $\lambda = \rho^{-2}$.} Here we consider the character $\rho^{-2}: T = \bT(L) \ra L^\ast$ defined by

$$\rho^{-2}\Big(\left(\begin{array}{cc} t_1 & 0 \\ 0 & t_2 \end{array}\right)\Big) = \frac{t_2}{t_1} \;,$$

\vskip8pt

i.e., the negative root. Then $d\rho^{-2}$ and $0 = d{\bf 1}$ are associated to $\theta_0$ by the Harish-Chandra homomorphism. Next we consider the well-known exact sequence of $U(\frg)$-modules

\begin{numequation}\label{BGG_res} 0 \lra M(\rho^{-2}) \lra M({\bf 1}) \lra L_0 \lra 0 \;,
\end{numequation}

where $L_0$ denotes the one-dimensional trivial $U(\frg)$-module. The functor $\cF^G_B$ of \cite{OrlikStrauchJH} associates to the sequence \ref{BGG_res} the exact sequence of locally analytic representations

\begin{numequation}\label{FGB_on_BGG_res} 0 \lra {\rm \ind}^G_B({\bf 1}) \lra {\rm Ind}^G_B({\bf 1}) \lra {\rm Ind}^G_B(\rho^2) \lra 0 \;,
\end{numequation}

where ${\rm \ind}^G_B({\bf 1})$ denotes the smooth induction of the trivial character ${\bf 1}$ of $B$. Passing in \ref{FGB_on_BGG_res} to $\bbG(n)^\circ$-analytic vectors, and then to the continuous duals of the spaces of $\bbG(n)^\circ$-analytic vectors, furnishes an exact sequence of $D(\bbG(n)^\circ,G_0)$-modules

\begin{numequation}\label{BGG_res_analytic} 0 \lra \bM(\rho^{-2})_n \lra \bM({\bf 1})_n \lra \bigoplus_{g \in C(n)} g_* L_0 \lra 0 \;,
\end{numequation}

where here $L_0$ denotes the $\cD^{\rm an}(\bbG(n)^\circ)$-module corresponding to the trivial one-dimensional representation of $\bbG(n)^\circ$. The Beilinson-Bernstein localization of $M(\rho^{-2})$ is a skyscraper sheaf supported at the fixed point of $B$ on $\bbP^1_L$, cf. \cite[sec. 12.3]{Hotta}, whereas the Beilinson-Bernstein localization of the trivial one-dimensional representation is the structure sheaf of $\bbP^1_L$. Applying $\Loc^\dagger_n$ and \ref{general_formula} to \ref{BGG_res_analytic} gives then the following exact sequence of sheaves of $\tsD^\dagger_{n,\Q}$-modules on $\frX_n$:

\begin{numequation}\label{BB_loc_of_BGG_res_analytic} 0 \ra \oplus_{g\in C(n)}  \; g_*{\rm Loc}(M(\rho^{-2}))^\dagger_n \ra \oplus_{g\in C(n)}  \; g_*{\rm Loc}(M({\bf 1}))^\dagger_n \ra \oplus_{g \in C(n)} g_* \cO_{\frX_n,\Q} \ra 0 \;.
\end{numequation}

The sheaf on the left is a skyscraper sheaf whose support is the set $\frX_n^{\rm sm}(\bbF_q)$ of smooth $\bbF_q$-rational points of $\frX_n$, a set which is naturally in bijection with $\bbP^1(\fro_L/(\varpi^{n+1}))$.
\end{para}

\subsection{Two representations furnished by the $p$-adic upper half plane} In this section we consider the $\GL_2(L)$-equivariant line bundles $\cO$, the structure sheaf, and $\Omega$, the sheaf of rigid-analytic differential forms, on $\bbP^{1,{\rm rig}}_L$ and their restrictions to the $p$-adic upper half plane $\Sigma_{0,L}$. Their global sections over $\Sigma_{0,L}$ furnish coadmissible $D(\GL_2(L),L)$-modules, which are closely related to the coadmissible module $\bM(\rho^{-2})$ considered in the previous section. As this material is not used later on, we only give indications of the proofs.
We also note that one can treat all line bundles $\cO(k)$ in a similar way, but only when $k \in \{-2,0\}$ do the corresponding coadmissible modules have central character $\theta_0$, which is why we restrict to these cases.

\begin{para}\label{the_padic_upper_half_plane} {\it The $p$-adic upper half plane.} Let $\bbP^{1,{\rm rig}}_L$ be the rigid-analytic space attached
to the scheme $\bbP^1_L$ over $L$. It is well-known and not difficult to prove that the subset

$$\Sigma_{0,L} = \bbP^{1,{\rm rig}}_L \; \setminus \; \bbP^1_L(L)$$

\vskip8pt

is an admissible open subset of $\bbP^{1,{\rm rig}}_L$, and is hence a rigid-analytic space in its own right, called the {\it $p$-adic upper half plane over $L$}.
The inductive limit of the formal schemes $\frX_n^\circ$ furnishes a formal model $\widehat{\Sigma}_{0,L}$ for $\Sigma_{0,L}$:

$$\widehat{\Sigma}_{0,L} = \varinjlim_n \frX_n^\circ \;,$$

\vskip8pt

cf. \cite{BoutotCarayol}.
\end{para}

\begin{prop}\label{analytic_St_etc} (i) There is a canonical exact sequence of locally analytic $G$-representations

$$0 \lra \Ind^G_B(\rho^2) \lra \cO(\Sigma_{0,L})' \lra {\bf 1} \lra 0 \;,$$

\vskip8pt

and hence an exact sequence of coadmissible $D(G,L)$-modules

$$0 \lra L_0 \lra \cO(\Sigma_{0,L}) \lra \bM(\rho^{-2}) \lra 0 \;.$$

\vskip8pt

(ii) There is a canonical exact sequence of locally analytic $G$-representations

$$0 \lra {\rm St} \lra \Omega(\Sigma_{0,L})' \lra \Ind^G_B(\rho^2) \lra 0 \;,$$

\vskip8pt

where ${\rm St}$ denotes the smooth Steinberg representation (realized on an $L$-vector space). Dually, there is an exact sequence of coadmissible $D(G,L)$-modules

$$0 \lra \bM(\rho^{-2})  \lra \Omega(\Sigma_{0,L}) \lra {\rm St}' \lra 0 \;.$$

\vskip8pt

\end{prop}

\Pf We will deduce the exact sequences of locally analytic representations from the description given in \cite{OrlikVectorbundles} and \cite[sec. 7]{OrlikStrauchJH} (see also p. 4 of the introduction of \cite{OrlikStrauchJH}). The exact sequences of $D(G,L)$-modules follow immediately from these, by duality. Let $\cL = \cO$ or $\cL = \Omega$. Put $\cL(\Sigma_{0,L})^0 = H^0(\Sigma_{0,L},\cL)$ and $\cL(\Sigma_{0,L})^1 = H^0(\bbP^1_L,\cL)$. Then there is a canonical exact sequence of locally analytic $G$-representations

\begin{numequation}\label{Orlik_seq} 0 \lra {\rm St} \otimes_L H^1(\bbP^1_L,\cL)' \lra \Big(\cL(\Sigma_{0,L})^0\Big/\cL(\Sigma_{0,L})^1\Big)' \lra \Ind^G_B(\rho^2) \lra 0 \;.
\end{numequation}

For $\cL = \cO$ (resp. $\cL = \Omega$) this is equivalent to the exact sequence in (i) (resp. (ii)). \qed

\vskip8pt

For $n \in \bbZ_{\ge 0}$ put

$$\cO(\Sigma_{0,L})_n = \Big(\left[\cO(\Sigma_{0,L})'\right]_{\bbG(n)^\circ-{\rm an}}\Big)' \;, \;\; \Omega(\Sigma_{0,L})_n = \Big(\left[\Omega(\Sigma_{0,L})'\right]_{\bbG(n)^\circ-{\rm an}}\Big)' \;,$$

\vskip8pt

and denote by ${\rm St}_n$ the space of $G_{n+1}$-invariant vectors in ${\rm St}$.

\vskip8pt

\begin{cor}\label{localizing_analytic_Steinberg_etc} Let $n$ be a non-negative integer.

\vskip8pt

(i) There is a canonical exact sequence of $G_0$-equivariant sheaves of $\tsD^\dagger_{n,\Q}$-modules on $\frX_n$

$$0 \lra \cO_{\frX_n,\Q} \lra \Loc^\dagger_n\Big(\cO(\Sigma_{0,L})_n\Big) \lra \oplus_{g\in C(n)}  \; g_*{\rm Loc}(M(\rho^{-2}))^\dagger_n \lra 0 \;.$$

\vskip8pt

(ii) There is a canonical exact sequence of $G_0$-equivariant sheaves of $\tsD^\dagger_{n,\Q}$-modules on $\frX_n$

$$0 \lra  \oplus_{g\in C(n)} \; g_*{\rm Loc}(M(\rho^{-2}))^\dagger_n \lra \Loc^\dagger_n\Big(\Omega(\Sigma_{0,L})_n\Big) \lra {\rm St}_n' \otimes_L \cO_{\frX_n,\Q} \lra 0 \;.$$

\vskip8pt

\end{cor}

\Pf This follows from \ref{analytic_St_etc} together with \ref{general_formula}. \qed

\vskip8pt

\begin{rem} As was pointed out in the previous section, the sheaf on the right (resp. left) of the first (resp. second) exact sequence in \ref{localizing_analytic_Steinberg_etc} is a skyscraper sheaf whose support is the set $Z_{0,n} := \frX_n^{\rm sm}(\bbF_q)$ of smooth $\bbF_q$-rational points of $\frX_n$. Using \cite{ChiarellottoDuality}, one can show that

$$\cO(\Sigma_{0,L})_n = \cO(\frX_n^{\circ,{\rm rig}})^\dagger$$

\vskip8pt

is the space of overconvergent rigid-analytic functions, i.e., the space of rigid-analytic functions on

$$\frX_n^{\circ,{\rm rig}} = {\rm sp}_{\frX_n}^{-1}\Big(\frX_n^\circ\Big)$$

\vskip8pt

which converge (to an unspecified extent) into the residue discs of the points in $Z_n$. It follows from this that the sheaf
$\Loc^\dagger_n\Big(\cO(\Sigma_{0,L})_n\Big)$ is equal to the sheaf $\cO(^\dagger Z_{0,n})$ of functions on $\frX_n$ which are regular on $\frX_n^\circ$ and have overconvergent singularities along the closed subset $Z_{0,n}$. This is an overconvergent isocrystal on $\frX_n$. Similarly, one can show that

$$\Omega(\Sigma_{0,L})_n = \Omega(\frX_n^{\circ,{\rm rig}})^\dagger$$

\vskip8pt

is the space of overconvergent rigid-analytic differential forms on $\frX_n^{\circ,{\rm rig}}$. The sheaf of differential forms $\Omega(^\dagger Z_{0,n})$ on $\frX_n$ which are regular on $\frX_n^\circ$ and have overconvergent singularities along $Z_{0,n}$ is not an isocrystal, however. Otherwise the localization of the space of global sections (on $\frX_n$) of this sheaf, namely $\Loc^\dagger_n(\Omega(\frX_n^{\rm rig})^\dagger)$, would have to be equal to $\Omega(^\dagger Z_{0,n})$. But it follows from the second exact sequence in \ref{localizing_analytic_Steinberg_etc} that the restriction of $\Loc^\dagger_n(\Omega(\frX_n^{\rm rig})^\dagger)$ to $\frX_n^\circ$ is locally free of rank $(q+1)q^n-1$, and not locally free of rank one.
\end{rem}

\vskip12pt

\section{Representations furnished by an \'etale covering of the $p$-adic upper half plane}\label{appl}

The theory in section \ref{loc_n} might be useful when one tries to decide whether certain modules $M$ over $D(\GL_2(\fro))$ are coadmissible, especially if $M$ 'comes from geometry' in a way that makes it accessible to the techniques used here. In this section we will show how the techniques from the theory of arithmetic differential operators can be used to prove that some particular locally analytic representations are admissible. The representations considered here are those furnished by the structure sheaf of the {\it first Drinfeld covering} of the $p$-adic upper half plane.

\vskip8pt

\subsection{The first Drinfeld covering}

\begin{para} Denote by $\Lnr$ the completion of the maximal unramified extension of $L$, and consider the base change $\Sigma_{0,L} \times_{\Sp(L)} \Sp(\Lnr)$ which we will from now on denote by $\Sigma_0$. This rigid space has a natural semistable formal model

$$\widehat{\Sigma}_0 = \varinjlim_n \frX_n^\circ \times_{\Spf(\fro)} \Spf(\fronr) \;,$$

\vskip8pt

where $\fronr$ is the ring of integers of $\Lnr$, cf. section \ref{the_padic_upper_half_plane}. By \cite{Drinfeld76} (cf. also \cite{BoutotCarayol}) the formal scheme $\widehat{\Sigma}_0$ represents a deformation functor of a particular two-dimensional formal group $F$ over $\fronr/(\vpi)$ which is equipped with an action of the maximal compact subring $O_D$ of a quaternion division algebra $D$ over $L$. Let $\vpi_D \sub O_D$ be a uniformizer. Denote by $\frF \ra \widehat{\Sigma}_0$ the universal deformation of $F$, and let $\frF[\vpi_D^j] \sub \frF$ be the finite flat group scheme of $\vpi_D^j$-torsion points. Denote by $\frF[\vpi_D^j]^{\rm rig}$ the associated rigid analytic space. Note that $\frF[1]$ is just the zero section of $\frF$, which is isomorphic to $\widehat{\Sigma}_0$. Then

$$\Sigma_1 = \frF[\vpi_D]^{\rm rig} \; \setminus \; \frF[1]^{\rm rig} = \frF[\vpi_D]^{\rm rig} \; \setminus \; \Sigma_0 \;.$$

\vskip8pt

Multiplication by $\vpi_D$ is a map $\frF[\vpi_D] \ra \frF[1]$ and hence induces a morphism $f: \Sigma_1 \ra \Sigma_0$ which is an \'etale covering space with Galois group $(\cO_D/(\vpi_D))^* \simeq \bbF_{q^2}^*$. We will henceforth write $\bbF_{q^2}^\ast$ for this covering group.

\vskip8pt

\begin{rem}\label{subgroup} The group $\GL_2(L)$ acts on $\Sigma_{0,L}$ by fractional linear transformations, and therefore also on $\Sigma_0 = \Sigma_{0,L} \times_{\Sp(L)} \Sp(\hat{L}^{\rm nr})$, and on the formal model $\widehat{\Sigma}_0$. However, this action, which is $\fronr$-linear, is {\it not} the action which is induced by the modular interpretation of $\widehat{\Sigma}_0$ as a deformation space, cf. \cite[I.6]{BoutotCarayol}. However, the group

$$\GL_2(L)^{(0)} = \Big\{g \in \GL_2(L) \midc \det(g) \in \fro_L^* \Big\}$$

\vskip8pt

does act $\fronr$-linearly on $\widehat{\Sigma}_0$, and $\hat{L}^{\rm nr}$-linearly on $\Sigma_0$, and the action of this group lifts to $\Sigma_1$. But in this section we will anyway only consider the action of $G_0 = \GL_2(\fro_L)$.
\end{rem}

{\it The subspaces $\Sigma_0(n)$ and $\Sigma_1(n)$.} There is a map

$$r: \Sigma_0 \ra |\cT|$$

\vskip8pt

from $\Sigma_0$ to the geometric realization $|\cT|$ of the Bruhat-Tits tree $\cT$ of ${\rm PGL}_2(L)$.
We refer to \cite[I.2.2]{BoutotCarayol} for the definition of this map (it is denoted by $\lambda$ in this paper).
Let $v_0 \in \cT$ be the vertex which is fixed by $\GL_2(\fro)$, and we denote by $|\cT|_{\le n} \sub |\cT|$ the set of points whose distance to $v_0$ is less or equal to $n$. For instance, $|\cT|_{\le 0} = \{v_0\}$. Then

$$\Sigma_0(n) = r^{-1}\Big(|\cT|_{\le n}\Big) \;,$$

\vskip8pt

is an affinoid subdomain of $\Sigma_0$, and is equal to the generic fiber of $\widehat{\Sigma}_0(n) := \frX_n^\circ \times_{\Spf(\fro)} \Spf(\fronr)$. We also put

$$\Sigma_1(n) = f^{-1}\Big(\Sigma_0(n)\Big) \;,$$

\vskip8pt

and we let $f_n: \Sigma_1(n) \ra \Sigma_0(n)$ denote the restriction of $f$.
\end{para}

\begin{para} {\it A semistable model for $\Sigma_1$.} In \cite{Teitelbaum_EtaleCov} J. Teitelbaum has shown that there is a finite (tamely ramified) extension $K$ of $\hat{L}^{\rm nr}$ such that $\Sigma_{1,K} = \Sigma_1 \times_{\Sp(L)} \Sp(K)$ has a semistable formal model $\widehat{\Sigma}_1$ over $\Spf(\fro_K)$. It follows from \cite[p. 72]{Teitelbaum_EtaleCov} that $\Sigma_{1,K}$ has $q^2-1$ connected components, which are transitively permuted by the action of the covering group, and each of these is an \'etale covering of $\Sigma_{0,K}$ of degree $q+1$. Consequently, the formal scheme $\widehat{\Sigma}_1$ has $q-1$ connected components, each of which is finite of degree $q+1$ over $\widehat{\Sigma}_0 \times_{\Spf(\fro_L)} \Spf(\fro_K)$.

\vskip8pt

As a general convention, we consider from now on the rigid spaces $\Sigma_0$, $\Sigma_0(n)$, $\Sigma_1$, $\Sigma_1(n)$, etc., as rigid spaces over $K$ (i.e., we perform a base change to $\Sp(K)$), and we consider the formal schemes $\widehat{\Sigma}_0$, $\widehat{\Sigma}_0(n)$, etc., as formal schemes over $\Spf(\fro_K)$ (i.e., we perform a base change to $\Spf(\fro_K)$).
The morphism $f: \Sigma_1 \ra \Sigma_0$ extends to a finite flat morphism $\widehat{f}: \widehat{\Sigma}_1 \ra \widehat{\Sigma}_0$ which is \'etale over the complement of the singular points of the special fiber of $\widehat{\Sigma}_0$. The irreducible components of the special fiber of $\widehat{\Sigma}_1$ are isomorphic to

\begin{numequation}\label{Drinfeld_curve}\Big\{[x:y:z] \in \bbP^2 \midc xy^q-x^qy=z^{q+1}\Big\} \;.
\end{numequation}

According to \cite[sec. 2]{DeligneLusztig}, it was V. Drinfeld who proved that the discrete series representations of $\SL_2(\bbF_q)$ are furnished by the $\ell$-adic \'etale cohomology of the affine curve $xy^q-x^qy=1$. Each of these projective curves lies over a unique irreducible component of the special fiber of $\widehat{\Sigma}_0$ (isomorphic to $\bbP^1$), and there are $q-1$ such curves over each irreducible component of the special fiber of $\widehat{\Sigma}_0$. The formal scheme

$$\widehat{\Sigma}_1(n) \eqdef \widehat{\Sigma}_1 \times_{\widehat{\Sigma}_0} \widehat{\Sigma}_0(n)$$

\vskip8pt

is then a formal model for $\Sigma_1(n)$, and it is an open formal subscheme of $\widehat{\Sigma}_1$. The morphism $\widehat{f}$ restricts to a morphism

$$\widehat{f}_n: \widehat{\Sigma}_1(n)  \lra \widehat{\Sigma}_0(n)$$

\vskip8pt

which is simply the projection from $\widehat{\Sigma}_1(n) = \widehat{\Sigma}_1 \times_{\widehat{\Sigma}_0} \widehat{\Sigma}_0(n)$ to the second factor. Denote by $\widehat{\Sigma}_1(n)^c$ the formal completion $\widehat{\Sigma}_1$ along the {\it closure} of $\widehat{\Sigma}_1(n)$. Similarly, let $\widehat{\Sigma}_0(n)^c$ be the completion of $\widehat{\Sigma}_0$ along the {\it closure} of $\widehat{\Sigma}_0(n)$. The morphism $\widehat{f}_n$ extends uniquely to a morphism $\widehat{f}_n^c:\widehat{\Sigma}_1(n)^c \ra  \widehat{\Sigma}_0(n)^c$ such that one has a commutative diagram of morphisms of formal schemes

\begin{numequation}\label{Tsuzuki_diagram}
\xymatrixcolsep{3pc}\xymatrix{
 \widehat{\Sigma}_1(n) \ar[r] \ar[d]^{\widehat{f}_n} & \widehat{\Sigma}_1(n)^c \ar[r] \ar[d]^{\widehat{f}_n^c} & \widehat{\Sigma}_1 \ar[d]^{\widehat{f}}\\
 \widehat{\Sigma}_0(n) \ar[r] & \widehat{\Sigma}_0(n)^c \ar[r] & \widehat{\Sigma}_0 \\
}
\end{numequation}

The special fiber of a formal scheme will always be denoted by the subscript ``$s$'', e.g., $\widehat{\Sigma}_1(n)_s$ denotes the special fiber of $
\widehat{\Sigma}_1(n)$. When it is clear from the context that it is only the special fiber of a given formal scheme $\cY$ which matters, we may drop the index ``$s$'' and write $\cY$ instead of $\cY_s$. Put $Z_{1,n} = \widehat{\Sigma}_1(n)^c_s \, \setminus \, \widehat{\Sigma}_1(n)_s$. This is a finite set of points of $\widehat{\Sigma}_{1,s}$.
\end{para}

\begin{para} {\it Sheaves of overconvergent functions.} Let

$${\rm sp}_{\widehat{\Sigma}_1}: \Sigma_1 \lra \widehat{\Sigma}_{1}$$

\vskip8pt

be the specialization map, and let

$$v_n: \widehat{\Sigma}_1(n)_s \hra \widehat{\Sigma}_{1,s}$$

\vskip8pt

be the open immersion. Associated to this morphism is the sheaf $v_n^\dagger \cO_{\Sigma_1}$ on $\Sigma_1$, cf. \cite[4.0.1]{Berthelot_Trento}, and we consider the sheaf

\begin{numequation}\label{F_n} \sF_n \eqdef ({\rm sp}_{\widehat{\Sigma}_1})_* v_n^\dagger \cO_{\Sigma_1}
\end{numequation}

on $\widehat{\Sigma}_{1,s}$ which is supported on $\widehat{\Sigma}_1(n)^c_s$. Another way to describe $\sF_n$ is as the sheaf of rigid-analytic functions on $\widehat{\Sigma}_1(n)^c$ which are regular on (the tubular neighborhood of) $\widehat{\Sigma}_1(n)$ and have overconvergent singularities along $Z_{1,n}$:

$$\sF_n = \cO_{\widehat{\Sigma}_1(n)^c,\Q}\Big({}^\dagger Z_{1,n}\Big) \;.$$

\vskip8pt

By its very definition, the space of global sections of $\sF_n$ on $\widehat{\Sigma}_1$ is equal to the space of the overconvergent functions on $\Sigma_1(n)$, which we denote by $H^0(\Sigma_1(n),\cO)^\dagger$, i.e.,

$$H^0\Big(\widehat{\Sigma}_1, \sF_n \Big) = H^0\Big(\widehat{\Sigma}_1(n)^c, \sF_n \Big) = H^0(\Sigma_1(n),\cO)^\dagger \;.$$

\vskip8pt
\end{para}

\subsection{Main result: coadmissibility of the space of global functions}

Recall that $G_0$ denotes the group $\GL_2(\fro_L)$. Let $\cD^{\rm an}(\bbG(0)^\circ)^{\rm nr}$ be the completed tensor product over $L$ of $\cD^{\rm an}(\bbG(0)^\circ)$ with the finite extension $K$ of $\hat{L}^{\rm nr}$. Similarly, denote by $D(\bbG(n)^\circ,G_0)^{\rm nr}$ the completed tensor product of $D(\bbG(n)^\circ,G_0)$ with $K$ over $L$.

\begin{thm}\label{admissibility_I} (i) For every $n \in \bbZ_{\ge 0}$ the space $H^0(\Sigma_1(n),\cO)^\dagger$ is a $\cD^{\rm an}(\bbG(n)^\circ)^{\rm nr}$-module of finite presentation.

\vskip8pt

(ii) For every $n \in \bbZ_{\ge 0}$ the space $H^0(\Sigma_1(n),\cO)^\dagger$ is a finitely generated $D(\bbG(n)^\circ,G_0)^{\rm nr}$-module.

\vskip8pt

(iii) For every $n \in \bbZ_{\ge 1}$ the restriction map

$$H^0(\Sigma_1(n),\cO)^\dagger \lra H^0(\Sigma_1(n-1),\cO)^\dagger$$

\vskip8pt

induces an isomorphism of $D(\bbG(n)^\circ,G_0)^{\rm nr}$-modules

\begin{numequation}\label{comp} D(\bbG(n-1)^\circ,G_0)^{\rm nr} \otimes_{D(\bbG(n)^\circ,G_0)^{\rm nr}} H^0(\Sigma_1(n),\cO)^\dagger \stackrel{\simeq}{\lra} H^0(\Sigma_1(n-1),\cO)^\dagger \;.
\end{numequation}

(iv) The space $H^0(\Sigma_1,\cO)$ is a coadmissible $D(G_0,K)$-module.
\end{thm}

\vskip8pt

\Pf (i) This will be shown in the following section \ref{starting_point}, where we treat first the case when $n=0$, and in section \ref{extending}, which deals with the case of general $n$.

\vskip8pt

(ii) This follows from (i) because $D(\bbG(n)^\circ,G_0)^{\rm nr}$ contains $\cD^{\rm an}(\bbG(n)^\circ)^{\rm nr}$.

\vskip8pt

(iii) This will be shown in section \ref{compatibility}.

\vskip8pt

(iv) This is a consequence of (ii) and (iii), by \cite[6.1.20]{EmertonA}.
\qed

\vskip8pt

Consider a character $\chi: \bbF_{q^2}^\ast \ra K^\ast$ of the covering group $\bbF_{q^2}^*$, and denote by $H^0(\Sigma_1,\cO)^\chi$ the $\chi$-isotypic component of $H^0(\Sigma_1,\cO)$. As $G_0$ commutes with the action of $\bbF_{q^2}^*$, this is a $D(G_0,K)$-submodule of $H^0(\Sigma_1,\cO)$, and we have the decomposition of $D(G_0,K)$-modules

$$H^0(\Sigma_1,\cO) = \bigoplus_\chi H^0(\Sigma_1,\cO)^\chi \;,$$

\vskip8pt

where $\chi$ runs over all characters of $\bbF_{q^2}^*$.

\vskip8pt

\begin{cor}\label{admissibility_II} For every character $\chi: \bbF_{q^2}^\ast \ra K^\ast$ the space $H^0(\Sigma_1,\cO)^\chi$ is a coadmissible $D(G_0,K)$-module.
\end{cor}

\Pf The projection map

$$H^0(\Sigma_1,\cO) \twoheadrightarrow H^0(\Sigma_1,\cO)^\chi \;, \;\; f \mapsto \frac{1}{q^2-1}\sum_{\zeta \in \bbF_{q^2}^*} \chi(\zeta) \cdot (\zeta^{-1}.f) \;,$$

\vskip8pt

is continuous and has as section the natural inclusion. Hence each isotypic component $H^0(\Sigma_1,\cO)^\chi$ is the quotient of a coadmissible module by a closed submodule, and is thus coadmissible too.
\qed

\begin{rem} Consider the de Rham sequence

$$0 \lra K \lra H^0(\Sigma_1,\cO) \lra H^0(\Sigma_1,\Omega) \lra H^1_{{\rm dR}}(\Sigma_1/K) \lra 0 \;.$$

\vskip8pt

It is exact because $\Sigma_1$ is a Stein space (as is $\Sigma_0$). One can show that the de Rham cohomology on the right is the (algebraic) dual space of a finite number of irreducible smooth representations of $\GL_2(L)^{(0)}$ (cf. \ref{subgroup} for the notation)\footnote{As far as we know, a proof of this has not been written up so far. It should be possible to do so by using the work of E. Grosse-Kl\"onne on Hyodo-Kato cohomology for such spaces as $\Sigma_1$.}. It follows from \cite[3.2]{ST5} that the category of coadmissible modules is closed under extensions. Therefore, \ref{admissibility_I} implies that $H^0(\Sigma_1,\Omega)$ is a coadmissible $D(G_0,K)$-module, and the same is true for the isotypic components $H^0(\Sigma_1,\Omega)^\chi$.
\end{rem}

\vskip8pt

\subsection{The starting point: the affinoid $\Sigma_1(0)$}\label{starting_point}

The case when $n=0$ is somewhat easier to deal with because $\widehat{\Sigma}_1(0)^c_s$ is smooth, since it is a disjoint union of $q-1$ copies of the curve in \ref{Drinfeld_curve}. Moreover, the morphism $\widehat{f}_{0,s}: \widehat{\Sigma}_1(0)_s \ra \widehat{\Sigma}_0(0)_s$ is smooth too. It is because of these simplifications that we treat this case beforehand. Denote by $\frX_0^{\rm nr}$ the fibre product $\frX_0 \times_{\Spf(\fro)} \Spf(\fro_K)$. Let $\sD^\dagger_{\frX^{\rm nr}_0,\Q}$ be the sheaf of arithmetic differential operators on the smooth formal scheme $\frX_0^{\rm nr}$.

\vskip12pt

\begin{prop}\label{case_n_0} The space $H^0\Big(\Sigma_1(0),\cO \Big)^\dagger$ is a finitely presented $\cD^{\rm an}(\bbG(0)^\circ)^{\rm nr}$-module.
\end{prop}

\vskip8pt

\Pf Recall the sheaf $\sF_0 = ({\rm sp}_{\widehat{\Sigma}_1})_* v_0^\dagger \cO_{\Sigma_1}$ introduced in \ref{F_n}.
This sheaf on $\widehat{\Sigma}_1(0)^c_s$ depends only on $\widehat{\Sigma}_1(0)_s$ which is a disjoint union of affine Deligne-Lusztig curves $xy^q-x^qy=1$. By \cite[sec. 2]{GKDeligneLusztig}, $\sF_0$ is an overconvergent $F$-isocrystal on the pair $(\widehat{\Sigma}_{1,s}(0), \widehat{\Sigma}_{1,s}(0)^c)$, where we use here the notation of \cite{TsuzukiBaseChange}. It is overconvergent along $Z_{1,0} = \widehat{\Sigma}_1(0)^c_s \, \setminus \, \widehat{\Sigma}_1(0)_s$. Consider the diagram

\begin{numequation}\label{Tsuzuki_diagram_0}
\xymatrixcolsep{3pc}\xymatrix{
 \widehat{\Sigma}_1(0)_s \ar[r] \ar[d]^{\widehat{f}_{0,s}} & \widehat{\Sigma}_1(0)^c_s \ar[r] \ar[d]^{\widehat{f}_{0,s}^c} & \widehat{\Sigma}_1 \ar[d]^{\widehat{f}}\\
 \widehat{\Sigma}_0(0)_s \ar[r] & \widehat{\Sigma}_0(0)^c_s \ar[r] & \widehat{\Sigma}_0 \\
}
\end{numequation}

cf. \ref{Tsuzuki_diagram}. By \cite[Thm. 4.1.4]{TsuzukiBaseChange}, the direct image $(\widehat{f}^c_0)_* \sF_0$ is an overconvergent $F$-isocrystal on the pair $(\widehat{\Sigma}_0(0)_s, \widehat{\Sigma}_0(0)^c_s)$. The preceding diagram \ref{Tsuzuki_diagram_0} corresponds to the diagram in \cite[Thm. 4.1.1]{TsuzukiBaseChange}. Note that $\widehat{\Sigma}_0(0)^c_s$ is isomorphic to the projective line over the residue field of $\fro_K$, which is just $\frX^{\rm nr}_{0,s}$.

\vskip8pt

By \cite[Thm. 4.3.5]{Caro06} or \cite[Prop. 3.1]{NH_Trihan07}, the $\sD_{\frX^{\rm nr}_0,\Q}^\dagger$-module $(\widehat{f}^c_0)_* \sF_0$ on $\frX_{0,s}^{\rm nr}$, is holonomic, and, in particular, coherent. By the main result of \cite{Huyghe97} (or by \ref{thm-equivalence}), the space of global sections of this sheaf, namely

$$H^0\Big(\frX_{0,s}^{\rm nr},(\widehat{f}^c_0)_* \sF_0\Big) = H^0\Big(\Sigma_1(0),\cO \Big)^\dagger \;,$$

\vskip8pt

is a finitely presented $H^0(\frX^{\rm nr}_{0,s},\sD_{\frX^{\rm nr}_0,\Q}^\dagger)$-module. Using \ref{global_sec_tsD} we see that it is a finitely presented $\cD^{\rm an}(\bbG(0)^\circ)^{\rm nr}_{\theta_0}$-module.  \qed

\vskip8pt

\subsection{Extending the arguments to $\Sigma_1(n)$ for all $n$}\label{extending}

In this section we prove \ref{admissibility_I} for all $n$. To simplify notation we will occasionally write $\tsD^\dagger_{n,\Q}$ instead of $\tsD^\dagger_{n,n,\Q}$. Denote by $\frX_n^{\rm nr}$ the fibre product $\frX_n \times_{\Spf(\fro)} \Spf(\fro_K)$. Note that the special fiber $\widehat{\Sigma}_0(n)^c_s$ of $\widehat{\Sigma}_0(n)^c$ is canonically isomorphic to the special fiber of $\frX_n^{\rm nr}$.

\begin{para} {\it Three types of closed points.} We want to show that $(\widehat{f}_n^c)_* \sF_n$ is a coherent $\tsD^\dagger_{n,\Q}$-module on $\frX_n^{\rm nr}$. Once we have this result at hand we can apply \ref{prop-genglobal} to deduce that

$$H^0(\frX_{n,s}^{\rm nr},(\widehat{f}_n^c)_* \sF_n) = H^0(\Sigma_1(n),\cO)^\dagger$$

\vskip8pt

is a coherent $\cD^{\rm an}(\bbG(n)^\circ)_{\theta_0}$-module and we are done. In order to show that $(\widehat{f}_n^c)_* \sF_n$ is a coherent $\tsD^\dagger_{n,\Q}$-module, we distinguish the following types of closed points on $\frX_{n,s}^{\rm nr}$:

\vskip8pt

(a) Points $x$ which lie on the smooth locus of $\frX_{n,s}^{\rm nr}$, but not in $\frX_n(\Fq)$.

\vskip5pt

(b) Points $x$ which lie on the singular locus of $\frX_{n,s}^{\rm nr}$.

\vskip5pt

(c) Points $x$ which lie on the smooth locus of $\frX_{n,s}^{\rm nr}$ and in $\frX_n(\Fq)$. These are exactly the smooth $\Fq$-rational points of $\frX_n$, the set of which we denote by $\frX_n^{\rm sm}(\Fq)$.
\end{para}

\vskip8pt

\begin{prop}\label{prop-prep}
Let $U$ be an open formal subscheme of $\frX_n$ and let $\mathring{\sE}$ be a $\tcD^{(m)}_{n,n}|_U$-module which is coherent as an $\cO_U$-module. Then:
\vskip8pt

(i) $\mathring{\sE}$ is coherent over $\tcD^{(m)}_{n,n}|_U$,

\vskip8pt
(ii)  the canonical homomorphism
$$ \mathring {\sE}\car \tsD^{(m)}_{n,n}|_U\otimes_{\tcD^{(m)}_{n,n}|_U} \mathring{\sE}$$
given by $a\mapsto 1\otimes a$ is an isomorphism.
\end{prop}
\Pf
We follow the arguments of \cite[Prop. 3.1.3]{Berthelot_Trento}.
We may assume that $U$ is affine and such that the ring $\tcD^{(m)}_{n,n}(U):=\Gamma(U,\tcD^{(m)}_{n,n})$ is noetherian, cf. Prop. \ref{graded_tcD}.
Since $\mathring{\sE}$ is coherent over $\cO_{U}$, there is a surjection $\cO_U^{\oplus r}\rightarrow \mathring{\sE}$ and hence a surjection $(\tcD^{(m)}_{n,n}|_U)^{\oplus r}\rightarrow \mathring{\sE}$. Let $\cN$ be its kernel.
The subsheaf $\tcD^{(m)}_{n,n}|_U\sub\cD^{(m)}_{U^{\rm log}}$ has the induced order filtration $F_i(\tcD^{(m)}_{n,n}|_U), i \geq 0$ which is a filtration by coherent $\cO_U$-submodules. Letting $\cN_i$ be the kernel of the morphism
 $(F_i(\tcD^{(m)}_{n,n}|_U))^{\oplus r}\rightarrow \mathring{\sE}$, it follows that $\cN$ has a filtration by coherent $\cO_U$-modules.
 Since $\tcD^{(m)}_{n,n}(U)$ is noetherian, there is a surjection $(\tcD^{(m)}_{n,n}(U))^{\oplus s}\rightarrow \Gamma(U,\cN)$. Since $\Gamma(U,.)$ commutes with
 inductive limits, one obtains from this a surjection  $(\tcD^{(m)}_{n,n}|_U)^{\oplus s}\rightarrow \cN$.
  The $\tcD^{(m)}_{n,n}|_U$-module $\mathring{\sE}$ has therefore a finite presentation. Since $\tcD^{(m)}_{n,n}|_U$ is a sheaf of coherent rings \cite[(3.1.1)]{BerthelotDI}, any module of finite presentation is coherent. This shows (i).

\vskip8pt

The map in (ii) is injective, since the module structure on $\mathring{\sE}$ extends by continuity from $\tcD^{(m)}_{n,n}|_U$ to $\tsD^{(m)}_{n,n}|_U$. We put $ \partial:=\varpi^n\partial_x$, where $x$ is the standard coordinate on $\bbX_0$, so that $\tcD^{(m)}_{n,n}|_U$ is generated over $\cO_U$ by $\partial^{\langle d \rangle}:=\frac{q_d^{(m)}!}{d!}\partial^{d}$ for $d\geq 0$. In particular,
  $F_i(\tcD^{(m)}_{n,n}|_U)=\sum_{d\leq i}a_d\partial^{\langle d \rangle}$ with local sections $a_d$ of $\cO_U$.
 Put $\partial':=\partial^{\langle p^m \rangle}$. Given an integer $r<p^m$, one has $\partial^{\langle r\rangle}\partial'^{\langle q \rangle}=u_{q,r}\partial^{\langle p^mq+r\rangle}$ where $u_{q,r}$ is a $p$-adic unit. This implies that any operator $P\in\tsD^{(m)}_{n,n}|_U$ can be written in the form $P=\sum_k B_k\partial'^k$ where $B_k\in\tcD^{(m)}_{n,n}|_U$ is an operator of order $<p^m$ and $B_k$ tends $p$-adically to zero for $k\rightarrow \infty$. Let $e$ be a local section of $\mathring{\sE}$. Since $\mathring{\sE}$ is coherent over $\cO_U$, the $\cO_U$-submodule generated by the elements $\partial'^k.e$ for $k\geq 0$ is generated by finitely many of those elements. Thus, there is a number $\ell$ and sections $a_0,...,a_{\ell-1}$ of $\cO_U$ and a relation $$(\partial'^\ell-\sum_{j<\ell}a_j\partial'^j).e=0.$$ It follows as in the proof of \cite[3.1.3]{Berthelot_Trento} that for any element $B_k$ there are operators $R_k\in\tcD^{(m)}_{n,n}|_U$ of order $<p^m\ell$ and $Q_k\in\tcD^{(m)}_{n,n}|_U$ such that

$$ B_k\partial'^k=Q_k(\partial'^\ell-\sum_{j<\ell}a_j\partial'^j)+R_k$$

\vskip8pt

with $Q_k$ and $R_k$ tending $p$-adically to zero if $B_k$tends $p$-adically to zero. It follows that $P\otimes e=1\otimes (\sum_k R_k).e$ in $\tsD^{(m)}_{n,n}|_U\otimes_{\tcD^{(m)}_{n,n}|_U} \mathring{\sE}$ which proves the surjectivity in (ii). \qed

\vskip8pt

\begin{cor}\label{cor-prep}
Let $U$ be an open formal subscheme of $\frX_n$ and let $\mathring{\sE}$ be coherent $\cO_U$-module. Put $\sE:=\mathring{\sE}\otimes\bbQ$. Then:

\vskip8pt

(i) If $\mathring{\sE}$ is a $\tcD^{(m)}_{n,n}|_U$-module for some $m$ then the canonical homomorphism
	
$$\sE\car \tsD^{(m)}_{n,n,\bbQ}|_U \otimes_{\tcD^{(0)}_{n,n,\bbQ}|_U} \sE$$

\vskip8pt

is an isomorphism and $\sE$ is coherent over $\tsD^{(m)}_{n,n,\bbQ}|_U$.

\vskip8pt

(ii) If $\mathring{\sE}$ is a $\tcD^{(m)}_{n,n}|_U$-module for any $m$ then the canonical homomorphism

$$\sE\car \tsD^{\dagger}_{n,n,\bbQ}|_U \otimes_{\tcD^{(0)}_{n,n,\bbQ}|_U} \sE$$
	
\vskip8pt

is an isomorphism and $\sE$ is coherent over $\tsD^{\dagger}_{n,n,\bbQ}|_U$.

\end{cor}

\Pf 
Since $\tcD^{(m)}_{n,n}$ and $\tcD^{(0)}_{n,n}$ agree upon tensoring with $\bbQ$, the isomorphism in claim (i) follows from tensoring the isomorphism in part (ii) of \ref{prop-prep} with $\bbQ$. Taking the limit over $m$ gives the isomorphism in claim (ii). It remains to prove the assertions about coherence. According to part (i) of  \ref{prop-prep}, the module $\mathring{\sE}$ is coherent over $\tcD^{(m)}_{n,n}|_U$. Passing to an open subspace of $U$ we may therefore assume that there is a finite presentation

$$(\tcD^{(m)}_{n,n}|_U)^{\oplus s}\rightarrow (\tcD^{(m)}_{n,n}|_U)^{\oplus r}\rightarrow \mathring{\sE}\rightarrow 0 \;.$$

\vskip8pt

Tensoring with $\bbQ$ gives a presentation of $\sE$ over $\tcD^{(0)}_{n,n,\bbQ}|_U$. Tensoring the latter with $\tsD^{(m)}_{n,n,\bbQ}|_U$ or $\tsD^{\dagger}_{n,n,\bbQ}|_U$ and
gives a finite presentation of $\sE$ over $\tsD^{(m)}_{n,n,\bbQ}|_U$ or $\tsD^{\dagger}_{n,n,\bbQ}|_U$.
Since the latter sheaves are sheaves of coherent rings, any module of finite presentation is coherent. This completes the proof. \qed

\vskip8pt

For $1\leq \nu \leq n$ and $\underline{a}\in\cR_\infty\times\cR^{\nu-1}$ recall the open subscheme $\bbX_ {\underline{a}}^{(\nu)}$ of $\bbX^\circ_n$ with its local coordinate $x_{\underline{a}}^{(\nu-1)}.$
For $a_n\in\cR$ recall the smooth open subscheme $\bbD_{\underline{a},a_n}^{(n)}$ in the outermost component of $\bbX_n$ with its local coordinate $x_{\underline{a},a_n}^{(n)}$. The open sets $\bbX_ {\underline{a}}^{(\nu)}\Big[\frac{1}{x_{\underline{a}}^{(\nu-1)}}\Big]$ together with the $\bbD_{\underline{a},a_n}^{(n)}$ form an open covering of the smooth locus of $\bbX_n$.
For any open subspace $U\sub\bbX_n$ we let $\cD^{(m)}_{U^{\rm log}}$ denote the usual sheaf of logarithmic arithmetic differential operators of level $m$ on $U$ as considered in \cite{PSS2}.
We write $\sD^{(m)}_{U^{\rm log}}$ for its $p$-adic completion and $\sD^{\dagger}_{U^{\rm log},\bbQ}$ for the union of the $\sD^{(m)}_{U^{\rm log},\bbQ}$.

\begin{lemma}\label{lemma-prep1}
Suppose $U$ is an open subspace of $\bbX_n$ which is smooth. The inclusion $\cD^{(m)}_{U^{\rm log}}\sub \cD^{(m)}_{U}$ is an equality.
\end{lemma}

\Pf The logarithmic tangent sheaf $\cT_{\bbX_n^{\rm log}}$ is generated on $\bbX_ {\underline{a}}^{(\nu)}$ freely by $x_{\underline{a}}^{(\nu-1)}\partial_{x_{\underline{a}}^{(\nu-1)}}$. It therefore coincides with the usual tangent sheaf on
$\bbX_ {\underline{a}}^{(\nu)}[\frac{1}{x_{\underline{a}}^{(\nu-1)}}]$. Similarly on
$\bbD_{\underline{a},a_n}^{(n)}$. Hence $\cT_{U^{\rm log}}=\cT_U$ which implies the claim. \qed

\vskip8pt

\begin{lemma}\label{lemma-prep2}
One has $\tcD^{(m)}_{n,k} \sub \cD_{\bbX_n^{\rm log}}^{(m)}$ for any $k\geq n$. For $k=n$ this inclusion becomes an equality over each $\bbD^{(n)}_{\underline{a},a_n}$.
\end{lemma}

\Pf By definition one has $\tcT_{n,k} \subseteq \cT_{\bbX_n^{\rm log}}$ and this implies $\tcD^{(m)}_{n,k} \subseteq \cD_{\bbX_n^{\rm log}}^{(m)}$.
Now suppose $k=n$ and $U = \bbD^{(n)}_{\underline{a},a_n}$. Since $U$ is smooth, $\cD_{U^{\rm log}}^{(m)} = \cD_{U}^{(m)}$ is generated over $\cO_U$ by
$\frac{q_d^{(m)}!}{d!} \partial_{x_{\underline{a},a_n}^{(n)}}^d$ whereas the subsheaf
$\tcD^{(m)}_{n,n}|_U$ is generated over $\cO_U$ by $\frac{q_d^{(m)}!}{d!}\varpi^{dn}\partial_x^d$ where
  $x$ is the standard coordinate on $\bbX_0$. But according to equation (24) in the proof of \cite[Prop. 5.2(a)]{PSS2}, we have $\partial_{x_{\underline{a},a_n}^{(n)}}=\varpi^{n}\partial_x$. This implies the claim.
\qed

\vskip8pt

\begin{para}\label{type_a}{\it Coherence at points of type (a).} By definition of the sheaf $(\widehat{f}^c_n)_* \sF_n$ we have for its restriction to the open subset $\hat{\Sigma}_0(n)_s \sub \frX_{n,s}^{\rm nr}$ that

$$\left((\widehat{f}^c_n)_* \sF_n\right)\Big|_{\hat{\Sigma}_0(n)} = ({\rm sp}_{\widehat{\Sigma}_0(n)})_*  (f_n)_* \cO_{\Sigma_1(n)}$$

\vskip8pt

where $f_n$ is the restriction of $f$ to $\Sigma_1(n)=f^{-1}(\Sigma_0(n))$. The points of type (a) lie in the smooth locus $U := \hat{\Sigma}_0(n)^{\rm sm}$ of $\hat{\Sigma}_0(n)$. Let $V:=\hat{f}_n^{-1}(U)$. The morphism

$$\hat{f}_n: V \lra U$$

\vskip8pt

induced by $\hat{f}_n$ is smooth and proper. Under these assumptions, the coherent $\cO_{U,\bbQ}$-module
$\sE:=\bbQ \otimes (\hat{f}_n)_* \cO_{V}$ is a convergent isocrystal over (the closed fibre of) $U$ \cite[3.7]{Ogus_Iso}. According to \cite[3.1.2]{Berthelot_Trento},
there exists, for every $m$, a $\sD^{(m)}_U$-module $\mathring{\sE}$ which is coherent over $\cO_U$ together with an $\sD^{(m)}_{U,\bbQ}$-linear isomorphism $\bbQ\otimes\mathring{\sE}\simeq\sE$. Via the natural inclusion $\tsD^{(m)}_{n,n}|_{U} \subseteq \sD^{(m)}_U$, cf. \ref{lemma-prep1} and \ref{lemma-prep2}, we may endow $\mathring{\sE}$ with a $\tsD^{(m)}_{n,n}$-module structure and apply part (ii) of \ref{cor-prep}. Thus, $\sE$ is a coherent $\tsD^\dagger_{n,n,\bbQ}$-module. It follows that $\left((\widehat{f}^c_n)_* \sF_n\right)\Big|_{U}$ is a coherent $\tsD^\dagger_{n,n,\bbQ}$-module, too. This proves our claim for points of type (a).
\end{para}

\vskip8pt

\begin{para}\label{type_b}{\it Coherence at points of type (b).}
We have
$$\left((\widehat{f}^c_n)_* \sF_n\right)\Big|_{\hat{\Sigma}_0(n)} = ({\rm sp}_{\widehat{\Sigma}_0(n)})_*  (f_n)_* \cO_{\Sigma_1(n)}$$

as explained above. Since points of type (b) are singular, we will make
use of logarithmic structures. Let $\ok$ be the residue field of $\fro_K$ equipped with the trivial log structure.
The morphism $$\hat{f}_{n,s}: \hat{\Sigma}_1(n)_s\longrightarrow \hat{\Sigma}_0(n)_s$$ is a finite flat morphism of degree $|\bbF_q^\times|$. We equip source and target with their fine log structures coming from the singular loci \cite[3.7(2)]{Kato_Log}. In particular, source and target are log smooth over $\ok$ and $\hat{f}_{n,s}$ is a log morphism. The target $\hat{\Sigma}_0(n)_s$ is covered by (the base change to $\ok$ of) the reductions of the $\bbX_ {\underline{a}}^{(\nu)}$ for $1\leq \nu \leq n$ and $\underline{a}\in\cR_\infty\times\cR^{\nu-1}$. On the reduction of $\bbX_ {\underline{a}}^{(\nu)}$ we have local coordinates $x_{\underline{a}}^{(\nu-1)}$ and $z_{\underline{a}}^{(\nu)}$ and a unique double point at $x_{\underline{a}}^{(\nu-1)}z_{\underline{a}}^{(\nu)}=0$. The log structure is therefore given by the monoid formed by the powers $(x_{\underline{a}}^{(\nu-1)})^{n_1}(z_{\underline{a}}^{(\nu)})^{n_2}$ with $n_i\in\bbN$. There is a chart for $\hat{f}_{n,s}$ over the reduction of $\bbX_ {\underline{a}}^{(\nu)}$ in which the corresponding map of monoids (or rather of their group completions) is given by multiplication with $|\bbF_q^\times|$ on $\bbZ^{\oplus 2}$. The latter map is therefore injective and the order of its cokernel is prime to $p$. According to \cite[3.5]{Kato_Log} and \cite[4.1(1)iv]{Kato_Log}, the morphism $\hat{f}_{n,s}$ is therefore log smooth and integral.

\vskip5pt
In the following we will make use of the work of Shiho on the logarithmic version of the convergent topos in characteristic $p$ \cite{Ogus_Conv}. To make
the comparison easier for the reader, we adapt some notation. Let $V:=\fro_K$ and equip ${\rm Spf}\; V$ with its log structure coming from the closed point of $Spec V$. Let $X_i$ be $\hat{\Sigma}_i(n)_s$ for $i=1,2$. Following \cite[2.1.3]{Shiho_CrystallineII} we denote by $(X_i/V)^{\rm log}_{\rm conv,Zar}$ the log convergent site of $X_i$ over ${\rm Spf}\; V$ with respect to the Zariski topology and by
$\cK_{X_1/V}$ the isocrystal $$T\mapsto \Gamma(T,\cO_T)\otimes\bbQ$$ on $(X_1/V)^{\rm log}_{\rm conv,Zar}$. Since $X_0$ is log smooth over $\ok$, the morphism $\hat{f}_{n,s}$ has even log smooth parameter in the sense of \cite[3.4]{Shiho_RelativeI}. We may therefore apply \cite[4.10]{Shiho_RelativeI} to $\hat{f}_{n,s}$ and $\cK_{X_1/V}$. Hence, for each number $q\geq 0$, there is a unique isocrystal $\cF^q$ on $(X_0/V)^{\rm log}_{\rm conv,Zar}$ whose value on the enlargement $\hat{\Sigma}_0(n)$ is given by the relative log analytic cohomology $R^q(f_s)_{X_1/\hat{\Sigma}_0(n),{\rm an}*}$ of $\cK_{X_1/V}$ (as defined in \cite[4.1]{Shiho_RelativeI}). By definition, the latter analytic cohomology is computed on the level of rigid analytic generic fibres via the log de Rham complex. Since the rigid analytic morphism $f$ is \'etale, one finds $\cF^0(\hat{\Sigma}_0(n))=f_*\cO_{\Sigma_1(n)}$ and $\cF^q(\hat{\Sigma}_0(n))=0$ for $q>0$.

\vskip5pt

The isocrystal $\cF^0$ admits an interpretation via tubular neighbourhoods, as in the non-logarithmic setting \cite[2.2.5]{Berthelot_preprint}.
 In particular, $\cF^0$ induces a convergent logarithmic connection $\nabla$ on $f_*\cO_{\Sigma_1(n)}$ and
 $\sE:=({\rm sp}_{\widehat{\Sigma}_0(n)})_*f_*\cO_{\Sigma_1(n)}$, cf. \cite[2.2.7]{Shiho_RelativeI} and \cite[p.67]{Shiho_RelativeI}. We claim that the actions of the logarithmic derivations induced by $\nabla$ extend to an action of the sheaf of logarithmic arithmetic differential operators $\sD^{\dagger}_{\hat{\Sigma}_0^{\rm log},\bbQ}$ on $\sE$. Since sections of $\sD^{\dagger}_{\hat{\Sigma}_0^{\rm log},\bbQ}$ locally admit convergent power series expansions \cite[2.3.2.C.]{Montagnon}, it suffices to verify the appropriate convergence condition for the Taylor series of $\nabla$ \cite[(3.0.1.1)]{Berthelot_Trento}. But this convergence condition follows from the fibration lemma \cite[2.31]{Shiho_RelativeI} applied to the projection $p_1$, as in the non-logarithmic setting \cite[2.2.13]{Berthelot_preprint}.

\vskip5pt

The $\sD^{\dagger}_{\hat{\Sigma}_0^{\rm log},\bbQ}$-module $\sE$ is coherent over $\cO_{\hat{\Sigma}_0(n),\bbQ}$ and therefore coherent over
$\sD^{\dagger}_{\hat{\Sigma}_0^{\rm log},\bbQ}$, cf. \cite[3.1]{Crew_formalcurve}. It follows from the proof of loc.cit. that there
is for every $m$, a $\sD^{(m)}_{\hat{\Sigma}_0^{\rm log}}$-module $\mathring{\sE}$ which is coherent over $\cO_{\hat{\Sigma}_0(n)}$ together with an $\sD^{(m)}_{\hat{\Sigma}_0^{\rm log},\bbQ}$-linear isomorphism $\bbQ\otimes\mathring{\sE}\simeq\sE$. Via the natural inclusion
$\tsD^{(m)}_{n,n}|_{\hat{\Sigma}_0} \subseteq \sD^{(m)}_{\hat{\Sigma}_0^{\rm log}}$, cf. \ref{lemma-prep2}, we may endow $\mathring{\sE}$ with a $\tsD^{(m)}_{n,n}$-module structure and apply part (ii) of \ref{cor-prep}. Thus, $\sE$ is a coherent
$\tsD^\dagger_{n,n,\bbQ}$-module. 
This proves our claim for points of type (b).

\end{para}

\vskip8pt

\begin{para}\label{direct}{\it An alternative argument valid when the ramification index is small.}
Suppose the ramification index of $L/\bbQ_p$ is $\leq p-1$. The fact that $(\widehat{f}^c_n)_* \sF_n$ is a coherent $\tsD^{\dagger}_{n,n,\bbQ}$-module at points of type (a) and (b) can then be seen explicitly as follows. The open subset $\hat{\Sigma}_0(n)$ is covered by open formal subschemes $U$ equal to (the base change to $\fro_K$ of) the formal completions $\hat{\bbX}_ {\underline{a}}^{(\nu)}$ for $1 \leq \nu \leq n$ and $\underline{a}\in\cR_\infty \times \cR^{\nu-1}$.
Over $U$ the sheaf $\tcD^{(m)}_{n,n}$ is generated by $\frac{q_d^{(m)}!}{d!}\varpi^{dn}\partial_x^d$ where
$x$ is the standard coordinate on $\bbX_0$. According to \cite[2.4.3]{BerthelotDI} any element $P \in \tsD^{\dagger}_{n,n}$ has therefore a unique expression over $U$ as infinite series in the form $P=\sum_{d\geq 0} a_d (\varpi^n\partial_x)^d/d!$ where the coefficients $a_d \in \cO(U)$ have the property that there exist real constants $c>0, \eta<1$ (depending on $P$) with $\Vert a_d \Vert < c\eta^d$ for all $d$. Here $\Vert . \Vert$ denotes some Banach norm on the affinoid algebra $A:=\cO(U)_\bbQ$. We denote by $\Vert . \Vert$ also a Banach norm on the finite $A$-module $B:=\sE(U)$.
We write $\cA:=\cO(U)$ and $\cB:=\mathring{\sE}(U)$. Making $U$ smaller if necessary we may assume that $\cB$ is a finite free $\cA$-module of rank $t=|\bbF_{p^2}^\times|$. More precisely, there is a unit $\Pi\in\cA^\times$ such that $\cB=\cA[X]/1-X^t\Pi$ \cite{Teitelbaum_EtaleCov}. The powers $\Pi^j$ for $j=0,...,t-1$ induce a decomposition

$$\cB\simeq \oplus_{j=0}^{t-1} \cA(j)$$

\vskip8pt

as $\cA$-modules where $\cA(j)$ equals $\cA$. Similarly for the $A$-module $B$. Since the morphism $f$ is finite \'etale, we have the Gauss-Manin connection on the \'etale $A$-algebra $B$.
A short calculation along the lines of \cite[2.(I)]{GKDeligneLusztig} shows that it respects the direct sum decomposition $B\simeq \oplus_j A(j)$ and is given on the $j$-th summand as

$$ f\mapsto d(f)-jt^{-1}fdlog(\Pi)$$

\vskip8pt

for $f\in A$. Here, $d:A\rightarrow\Omega_{A/K}$ is the canonical derivation of the $K$-algebra $A$ and $dlog(\Pi):=d(\Pi)/\Pi.$
The induced action of a derivation $\partial\in Der_K(A)$ on the $j$-th factor $A(j)$ is therefore given by

\begin{numequation}\label{equ-2} f \mapsto \partial.f:=\partial(f)-jt^{-1}f\partial(\Pi)/\Pi.
\end{numequation}

To obtain from this an action of an infinite sum like $P$ on $B$, it suffices to verify for any $f \in B$ and any $\eta<1$ the convergence property

\begin{numequation}\label{conv} \Big\Vert  \left(\frac{(\varpi^n\partial_x)^d}{d!}\right).f \Big\Vert \eta^d\rightarrow 0\end{numequation}

for $d\rightarrow\infty$ \cite[3.1.1]{Berthelot_Trento}. The $\GL_2(\fro)$-equivariance of the situation allows us to assume that the affinoid $Sp(A)={\rm sp}^{-1}(U) \sub \bbP^{1,an}_{K}$ equals

$$\{z\in \bbP^{1,an}_{K}: |z|\leq 1\} - \bigcup_{i=0,...,q-1} ( B_1(i)\cup B_{1/q}(qi)) - B_{1/q}(0)$$

\vskip8pt

where $B_\delta(z_0)$ is the open ball of radius $\delta$ centered at $z_0$ and where the affine coordinate $z$ on $\bbP^{1,an}_{K}$ is induced by $x_{\underline{a}}^{(\nu-1)}$ \cite[p. 649]{TeitelbaumDrin}.
 According to the Mittag-Leffler decomposition \cite[Prop. 2.2.6]{FresnelVanderPut}, we may assume that $f$ is a holomorphic function on the complement in $\bbP^{1,an}_{K}$ of {\it one} of the above open discs. We may assume that the radius of the open disc is $1$ (otherwise we work with the coordinate $z=z_{\underline{a}}^{(\nu)}$) and may thus write $f(z)=\sum_{n\leq 0} a_n(z-z_0)^n$ with $a_n\in K$ such that
$\lim_{n\rightarrow\infty} |a_{-n}|=0$. Moreover, $\Vert f \Vert=\max_{n\leq 0} |a_{n}|$ and so $\Vert \partial_z(f) \Vert \leq \Vert f \Vert$. According to the proof of \cite[Cor. 6]{Teitelbaum_EtaleCov} we may assume that $\Pi$ is a polynomial in $z$ with coefficients in $\fro_K$ and so $\Vert \partial_z(\Pi) \Vert \leq \Vert \Pi \Vert$. All in all, this implies
 $\Vert \partial_z.f \Vert \leq \Vert f \Vert$
according to (\ref{equ-2}). On the other hand,
by the equation (5.1.3) in the proof of \cite[Prop. 5.1.2]{PSS2}, we have $\partial_z:=\partial_{x_{\underline{a}}^{(\nu-1)}}=\varpi^{\nu-1}\partial_x$ and therefore $\varpi^{n}\partial_x=a\varpi\partial_z$ with a $p$-adic integer $a$. The convergence property \ref{conv} follows now from the fact $\frac{|\varpi^d|}{d!}\eta^d\rightarrow 0$.

\end{para}

\vskip8pt

\begin{para}\label{type_c}{\it Coherence at points of type (c).} We recall that the special fibre of $\widehat{\Sigma}_0(n)^c$ is canonically isomorphic to the special fibre of $\frX^{\rm nr}_n$. Let $\widehat{\Sigma}_{0,s}(v) \sub \widehat{\Sigma}_0(n)^c_s \simeq \frX^{\rm nr}_{n,s}$ correspond to one of the 'outermost' irreducible components of $\frX_{n,s}$. These are indexed by the vertices $v$ of $\cT$ which have distance $n$ to $v_0$. Let $\widehat{\Sigma}_{0,s}(v)^\circ$ be the smooth locus of $\widehat{\Sigma}_{0,s}(v)$. Then $\widehat{\Sigma}_{0,s}(v)^\circ$ is isomorphic to an affine line over the residue field of $\fro_K$. Let $\widehat{\Sigma}_{0,s}(v)^{\circ \circ} \sub \widehat{\Sigma}_{0,s}(v)^\circ$ be the complement of the $\Fq$-rational points of $\widehat{\Sigma}_{0,s}(v)^\circ$. Put

$$\Sigma_1(v) = f^{-1}\Big({\rm sp}_{\widehat{\Sigma}_0}^{-1}\Big(\widehat{\Sigma}_{0,s}(v)^{\circ \circ}\Big)\Big) \;.$$

\vskip8pt

The \'etale covering

$$f|_{\Sigma_1(v)}: \Sigma_1(v) \lra {\rm sp}_{\widehat{\Sigma}_0}^{-1}\Big(\widehat{\Sigma}_{0,s}(v)^{\circ \circ}\Big)$$

\vskip8pt

is isomorphic to the covering $\Sigma_1(0) \ra \Sigma_0(0)$. Let $\widehat{\Sigma}_{1,s}(v) \sub \widehat{\Sigma}_{1,s}$ be the image of $\Sigma_1(v)$ under the specialization map. Denote, as before, by the superscript ``$c$'' the topological closure. Note that $\widehat{\Sigma}_{0,s}(v)^c$ is equal to $\left[\widehat{\Sigma}_{0,s}(v)^{\circ \circ}\right]^c$ and is isomorphic to a projective line over the residue field of $\fro_K$. Then we have a commutative diagram

\begin{numequation}\label{Tsuzuki_diagram_n}
\xymatrixcolsep{3pc}\xymatrix{
 \widehat{\Sigma}_{1,s}(v) \ar[r] \ar[d]^{\widehat{f}(v)_s} & \widehat{\Sigma}_{1,s}(v)^c \ar[r] \ar[d]^{\widehat{f}(v)_s^c} & \widehat{\Sigma}_1 \ar[d]^{\widehat{f}}\\
 \widehat{\Sigma}_{0,s}(v)^{\circ \circ} \ar[r] & \widehat{\Sigma}_{0,s}(v)^c \ar[r] & \widehat{\Sigma}_0 \\
}
\end{numequation}

\vskip8pt

where the vertical morphisms on the left and in the middle are those induced by $\widehat{f}$. This diagram is analogous to the diagram \ref{Tsuzuki_diagram_0}. By \cite[Thm. 4.1.4]{TsuzukiBaseChange}, the sheaf $(\widehat{f}(v)^c)_* \sF_n$ is an overconvergent $F$-isocrystal on the pair

$$\left(\widehat{\Sigma}_{0,s}(v)^{\circ \circ},\widehat{\Sigma}_{0,s}(v)^c\right) \;.$$

\vskip8pt

By \cite[Thm. 4.3.5]{Caro06} or \cite[Prop. 3.1]{NH_Trihan07}, the $\sD_{\widehat{\Sigma}_{0,s}(v)^c,\Q}^\dagger$-module $(\widehat{f}(v)^c)_* \sF_n$ is holonomic, and, in particular, coherent. Here, $\sD^\dagger_{\widehat{\Sigma}_{0,s}(v)^c,\Q}$ is the usual sheaf of arithmetic differential operators on $\widehat{\Sigma}_{0,s}(v)^c$.

\vskip8pt

We now use the fact that the sheaf $\tsD^\dagger_{n,n,\Q}$ on $\frX_n$, when restricted to $\widehat{\Sigma}_{0,s}(v) \sub \widehat{\Sigma}_0(n)^c_s \simeq \frX^{\rm nr}_{n,s}$ is isomorphic to
$\sD^\dagger_{\widehat{\Sigma}_{0,s}(v),\Q}$, cf. \ref{lemma-prep1} and \ref{lemma-prep2}. Therefore, $(\widehat{f}(v)^c)_*(\sF_n)|_{\widehat{\Sigma}_{0,s}(v)^c}$ is a coherent module over $\tsD^\dagger_{n,n,\Q}|_{\widehat{\Sigma}_{0,s}(v)^c}$.
In particular, if $x$ is a smooth $\bbF_q$-rational point of $\frX_n$, which necessarily lies on $\widehat{\Sigma}_{0,s}(v)$ for some vertex $v$ which has distance $n$ to $v_0$, the sheaf $(\widehat{f}^c_n)_*(\sF_n)$ is coherent over $\tsD^\dagger_{n,n,\Q}$ at $x$.
This proves our claim for points of type (c). \qed
\end{para}

\vskip8pt

\subsection{$\bbG(n-1)^\circ$-analytic vectors in the dual of $H^0(\Sigma_1(n),\cO)^\dagger$}\label{compatibility}

We abbreviate

$$R:=D(G_0)^{\rm nr},~R_n:=D(\bbG^\circ(n),G_0)^{\rm nr} \hskip5pt {\rm and} \hskip5pt E_n:=H^0(\Sigma_1(n),\cO)^\dagger \;.$$

\vskip8pt

The rigid-analytic space $\Sigma_1$ is quasi-Stein with defining affinoid covering given by the $\Sigma_1(n)$. This implies that the restriction map $E_n\rightarrow E_{n-1}$ has dense image. Hence, the induced homomorphism of finitely generated $R_{n-1}$-modules

$$h_n: R_{n-1}\otimes_{R_n} E_n \lra E_{n-1}$$

\vskip8pt

has dense image. On the other hand, this image is closed by \ref{prop-cantop}, and so $h_n$ is surjective. We now show that $h_n$ is injective. To start with, the restriction map $E_n\rightarrow E_{n-1}$ is injective. Since the natural ring homomorphism $R\rightarrow R_n$ is flat, the map

$$R_{n-1}\otimes_R E_n \lra  R_{n-1}\otimes_R E_{n-1}$$

\vskip8pt

is injective. Hence, the injectivity of $h_n$ follows if we can show that, for each $n$, the natural map $$w_n: R_n\otimes_R E_n\car E_n$$ is a bijection. Since $w_n$ is obviously surjective, we will prove its injectivity in the following. Since the group $G_{n+1}$ acts $\bbG(n)^\circ$-analytically on the strong dual $(E_n)'_b$, we may dualise the isomorphism in \cite[3.4.5]{EmertonA}
and obtain the isomorphism

$$\cD^{\rm an}(\bbG(n)^\circ)\hat{\otimes}_{D(G_{n+1})} E_n\car E_n \;.$$

\vskip8pt

The latter factors through the natural surjection
$$ \hat{w}_n: R_n\hat{\otimes}_R E_n\rightarrow E_n$$ which must then be an isomorphism, too. It therefore remains to see that the natural map

$$ \iota: R_n\otimes_R E_n \lra R_n\hat{\otimes}_R E_n$$

\vskip8pt

is injective (and hence bijective). Fix a set of generators $e_1,...,e_s$ for the
$R_n$-module $E_n$. Consider the topological $R_n$-module $M:=R_n\otimes_{K,\pi} E_n$. Its underlying topological space is Hausdorff \cite[17.5]{NFA} and has a completion which is of compact type \cite[1.1.32]{EmertonA}. Let $N$ be its closed $R_n$-submodule generated by the finitely many elements $1\otimes \delta_{g} e_j-\delta_{g}\otimes e_j$ where $g$ runs through a finite set of topological generators for the group $G_0$. There is a natural quotient map $q: M/N\rightarrow  R_n\otimes_R E_n$. On the other hand, the natural map $\iota': M/N \ra \widehat{M/N}$ of $M/N$ into its Hausdorff completion is injective. Arguing as in \cite[3.4]{SchmSt}, the completion $\widehat{M/N}$ is the same as $R_n\hat{\otimes}_R E_n$. We obtain an injective map $\iota': M/N\rightarrow R_n\hat{\otimes}_R E_n$ which is seen to equal $\iota\circ q$. Hence $\iota$ must be injective, too.

\bibliographystyle{plain}
\bibliography{mybib}

\end{document}